\documentclass{article}

\usepackage{amsmath}
\usepackage{amssymb}
\usepackage{latexsym}
\usepackage{mathabx}
\usepackage{cancel}
\usepackage{tikz,tikz-cd}

\usepackage{amsthm}
\newtheorem{thm}{Theorem}[section] 
\newtheorem{prop}[thm]{Proposition}
\newtheorem{lem}[thm]{Lemma}
\newtheorem{cor}[thm]{Corollary}
\newcommand{\tg}       {{\rm t}}
\newcommand{\s}        {{\rm s}}
\newcommand{\rr}       {\rightrightarrows}

\theoremstyle{definition}
\newtheorem{df}[thm]{Definition} 
\newtheorem{rem}[thm]{Remark}
\newtheorem{ex}[thm]{Example}

\usepackage[all]{xy}
\newdir{ >}{{}*!/-5pt/@{>}}

\usepackage{xcolor}

\usepackage{xspace}

\usepackage{mathrsfs}

\newcommand{\LAgpds}{$\mathscr{L}\!\mathscr{A}$--groupoids\xspace}

\newcommand{\inv}{^{-1}}

\newcommand{\add}[1]{\mathbin{\lower 5pt%
    \hbox{${\stackrel{\textstyle +}{\scriptscriptstyle #1}}$}}}

\renewcommand{\Check}[1]{\widecheck{#1}}

\renewcommand{\Hat}[1]{\widehat{#1}}

\renewcommand{\phi}{\varphi}
\newcommand{\chigh}{{\raise1.5pt\hbox{$\chi$}}}

\newcommand{\da}{\partial}

\newcommand{\sfn}[1]{C^{\infty}(#1)}

\usepackage{stmaryrd}
\usepackage{wasysym}

\DeclareMathOperator{\Der}{Der}
\DeclareMathOperator{\CD}{trCD}
\DeclareMathOperator{\DLA}{trDLA}
\DeclareMathOperator{\CDG}{trCDG}
\DeclareMathOperator{\DLG}{trDLG}
\DeclareMathOperator{\D}{\mathcal D}
\newcommand{\Derb}{\Der_{[\cdot,\cdot]}}
\DeclareMathOperator{\Hom}{Hom}

\newcommand{\duer}{%
\mathbin{\raisebox{3pt}{\varhexstar}\kern-3.70pt{\rule{0.15pt}{4pt}}}\,}

\newcommand{\id}{\text{\upshape id}}

\newcommand{\R}{\mathbb{R}}

\newcommand{\dr}{\mathbf{d}}

\newcommand{\an}[1]{\arrowvert_{#1}}
\newcommand{\mx}{\mathfrak{X}}

\newcommand{\tmtc}{\Gamma_{TM}(TC)}

\newcommand{\thbr}[2]%
{\rule[-1pt]{1pt}{10pt}\hspace{2pt} #1,\, #2\hspace{1pt}\rule[-1pt]{1pt}{10pt}}

\newcommand{\newpair}[2]%
{\talloblong\hspace{1pt} #1,\, #2\hspace{0.5pt}\talloblong}

\allowdisplaybreaks[1]

\usepackage{array}

\usepackage{subfig}
\usepackage{caption}
\newcommand{\ldr}[1]{{{\pounds}}_{#1}}

\newcommand{\Ga}{\Gamma}

\newcommand{\co}{\colon\thinspace}

\newcommand{\act}{\mathbin{\hbox{$<\kern-.4em\mapstochar\kern.4em$}}}
\newcommand{\ract}{\mathbin{\hbox{$\mapstochar\kern-.3em>$}}}

\usepackage{url}

\newcommand{\pf}{\textsc{Proof:\ }}
\newcommand{\epf}{\hspace*{\fill}\lower0pt\hbox{$\Box$}\medskip}

\newcommand{\ldbl}{R} 

\begin{document}

\title{Transitive double Lie algebroids via core diagrams
}

\author{M.~Jotz Lean\footnote{Mathematisches Institut, Georg-August
    Universit\"at
    G\"ottingen. \texttt{madeleine.jotz-lean@mathematik.uni-goettingen.de}}
  \,\, and K.~C.~H.~Mackenzie\footnote{In memoriam -- School of
    Mathematics and Statistics, The University of Sheffield.}}


\maketitle

\begin{abstract}
  The core diagram of a double Lie algebroid consists in the core of
  the double Lie algebroid, together with the two core-anchor maps to
  the sides of the double Lie algebroid. If these two core anchors are
  surjective, then the double Lie algebroid and its core diagram are
  called \emph{transitive}. This paper establishes an equivalence
  between transitive double Lie algebroids, and transitive core
  diagrams over a fixed base manifold. In other words, it proves that
  a transitive double Lie algebroid is completely determined by its
  core diagram.
  
  The comma double Lie algebroid associated to a morphism of Lie
  algebroids is defined. If the latter morphism is one of the
  core-anchors of a transitive core diagram, then the comma double
  algebroid can be quotiented out by the second core-anchor, yielding
  a transitive double Lie algebroid, which is the one that is
  equivalent to the transitive core diagram.

  Brown's and Mackenzie's equivalence of transitive core diagrams (of
  Lie groupoids) with transitive double Lie groupoids is then used in
  order to show that a transitive double Lie algebroid with integrable
  sides and core is automatically integrable to a transitive double
  Lie groupoid.

  \medskip \emph{This research was a joint project with the sadly
    deceased second author. This paper is dedicated to his memory.}
  \end{abstract}

\tableofcontents

\section{Introduction}
Double structures in geometry where first studied by the school of
Ehresmann and later extensively by the second author, see
\cite{Mackenzie05}. In particular, \cite{Mackenzie92,Mackenzie00}
define \emph{double Lie algebroids} and prove that they are the
infinitesimal objects associated to \emph{double Lie
  groupoids}. Differentiating one side of a double Lie groupoid yields
an \emph{LA-groupoid} \cite{Mackenzie92}, and a second differentiation
process applied to this LA-groupoid yields the tangent double Lie
algebroid of the double Lie groupoid \cite{Mackenzie00}. The converse
integration of double Lie algebroids to double Lie groupoids has not
been completely solved yet. Stefanini proposes in \cite{Stefanini09}
an integration of LA-groupoids with integrable top Lie algebroid and
strongly transitive source and target maps. Burzstyn, Cabrera and del
Hoyo integrate in \cite{BuCaHo16} double Lie algebroids with one
integrable (top) side to LA-groupoids.  This paper integrates
\emph{transitive} double Lie algebroids \emph{in one step} to
\emph{transitive} double Lie groupoids, provided the side and core Lie
algebroids are all integrable. This integration is the motivation
behind the results in this work. It is indeed very simple using the
equivalence of transitive double Lie groupoids with their core
diagrams in \cite{BrMa92}, and the equivalence of transitive double Lie
algebroids and their core diagrams proved here.

\medskip

A double groupoid is a \emph{groupoid object in the category of
  groupoids}.  That is, a double groupoid consists of a set $\Gamma$
that has two groupoid structures over two bases $G$ and $H$, which are
themselves groupoids over a base $M$, such that the structure maps of
each groupoid structure on $S$ are morphisms with respect to the
other. \[\begin{tikzcd}
    \Gamma & G \\
    H & M \arrow[shift left=1, from=1-2, to=2-2] \arrow[shift right=1,
    from=1-2, to=2-2] \arrow[shift right=1, from=2-1, to=2-2]
    \arrow[shift left=1, from=2-1, to=2-2] \arrow[shift right=1,
    from=1-1, to=1-2] \arrow[shift right=1, from=1-1, to=2-1]
    \arrow[shift left=1, from=1-1, to=2-1] \arrow[shift left=1,
    from=1-1, to=1-2]
  \end{tikzcd}\] The groupoids $G$ and $H$ are the side groupoids and
the set $M$ is the double base. The core $K$ of a double groupoid
$(\Gamma,G,H,M)$ is the set of elements which project under the two
sources of $\Gamma$ to units of $G$ and $H$. It inherits from the
double structure of $\Gamma$ a groupoid structure over $M$, and the
two targets of $\Gamma$ induce two groupoid morphisms
$\tg_G\arrowvert_K\colon K\to G$ and $\tg_H\arrowvert_K\colon K\to H$
over the identity on $M$. Elements of the kernel of the morphism
$K\to G$ commutes with elements of the kernel of $K\to H$. These two
groupoid morphisms build together the \emph{core diagram} of the
double groupoid \cite{BrMa92}.
   \[\begin{tikzcd}
	K & G \\
	H & M
	\arrow[shift left=1, from=1-2, to=2-2]
	\arrow[shift right=1, from=1-2, to=2-2]
	\arrow[shift right=1, from=2-1, to=2-2]
	\arrow[shift left=1, from=2-1, to=2-2]
	\arrow[shift right=1, from=1-1, to=2-2]
	\arrow[shift left=1, from=1-1, to=2-2]
	\arrow["{\tg_G\arrowvert_K}", from=1-1, to=1-2]
	\arrow["{\tg_H\arrowvert_K}"', from=1-1, to=2-1]
      \end{tikzcd}\]

    A \emph{double Lie groupoid} is a double groupoid
    $(\Gamma, G, H, M)$ such that all four groupoids $\Gamma\rr G$,
    $\Gamma\rr H$, $G\rr M$ and $H\rr M$ are Lie groupoids and such
    that the double source map
    $(\s_G, \s_H)\colon \Gamma\to G\times_\s H=\{(g,h)\in G\times
    H\mid \s(g)=\s(h)\}$ is a smooth surjective submersion. In that
    case, the core diagram is a core diagram of Lie groupoids.  A
    double Lie groupoid is \emph{locally trivial} if the Lie groupoid
    $K\rr M$ is locally trivial and $\da_G:=\tg_G\arrowvert_K$ and
    $\da_H:=\tg_H\arrowvert_K$ are both surjective submersions. It is
    \emph{transitive} if only $\da_G$ and $\da_H$ are both surjective
    submersions.  Brown and Mackenzie proved in \cite{BrMa92} that
    locally trivial double Lie groupoids are completely determined by
    their core diagrams. Section \ref{expand_BrMa92} reviews their
    construction and shows that $K\rr M$ does not need to be locally
    trivial for the equivalence to work. Hence, transitive double Lie
    groupoids are completely determined by their core diagrams.

    Mackenzie proved then in \cite{Mackenzie92} that locally trivial
    LA-groupoids are completely determined by their core
    diagrams. This paper completes this series of results by proving
    that \emph{transitive double Lie algebroids are completely defined
      by their core diagrams}.

    \medskip

Double Lie groupoids are described infinitesimally by double Lie
algebroids \cite{Mackenzie92,Mackenzie00}.
A double Lie algebroid with core $C$ is a
double vector bundle $(D;A,B;M)$
\begin{equation*}
  \begin{xy}
    \xymatrix{
      D \ar[r]^{\pi_A}\ar[d]_{\pi_B}&   A\ar[d]^{q_A}\\
       B\ar[r]_{q_B}                   &  M\\
    }
  \end{xy}
\end{equation*}
with core $C$ and four Lie algebroid structures on $A\to M$, $B\to M$,
$D\to A$ and $D\to B$ such that $(\pi_B,q_A)$ and $(\pi_A,q_B)$ are
Lie algebroid morphisms, and the induced Lie algebroids
$D\duer A\to C^*$ and $D\duer B\to C^*$ form a VB-Lie bialgebroid
\cite{Mackenzie00}. The anchor $\Theta_A\co D\to TA$ is a morphism of
double vector bundles and its core morphism is denoted by
$\da_A\colon C\to A$. Likewise the core morphism of the linear anchor
$\Theta_B\co D\to TB$ is written $\da_B\colon C\to B$.  The core $C$
of the double Lie algebroid inherits a Lie algebroid structure over
$M$, such that the two \emph{core-anchors} $\da_A$ and $\da_B$ are Lie
algebroid morphisms over $M$ \cite{Mackenzie11,GrJoMaMe18}.  The
compatiblity of the Lie algebroid structures on $D\to A$ and $D\to B$
implies that $[c_1,c_2]=0$ for all $c_1\in\Gamma(\ker\da_A)$ and
$c_2\in\Gamma(\ker\da_B)$.  In other words, \emph{$\ker\da_A$ and
  $\ker\da_B$ commute in $C$}.  Hence, the core of a double Lie
algebroid defines a diagram of Lie algebroids morphisms as in Figure
\ref{fig:cd_simple}, with commuting kernels. It is called the
\emph{core diagram of the double Lie algebroid}.
\begin{figure}[h]
$$ 
\xymatrix{
&C\ar[r]_{\da_B}\ar[d]^{\da_A}&B\\
&A&
}
$$ 
\caption{The core diagram of a double Lie algebroid.\label{fig:cd_simple}}
\end{figure}

If both $\da_A$ and $\da_B$ are surjective, then the double Lie
algebroid and its core diagram are both called \emph{transitive}. 
It is easy
to see that the map sending a double Lie algebroid to its core diagram
defines a functor from the category of double Lie algebroids with
double base $M$, to the category of core diagrams over $M$. It turns
out that the restriction of this functor to transitive double Lie
algebroids versus transitive core diagrams has an inverse, which is
constructed in this paper. In particular, it proves a more precise version of the following
theorem (see Theorem \ref{eq_cat}).

\begin{thm}\label{main_vague}
  A transitive double Lie algebroid $(D,A,B,M)$ is uniquely determined
  (up to isomorphism) by its core diagram.
\end{thm}

This work relies heavily on Gracia-Saz and Mehta's equivalence of
2-representations with decomposed VB-algebroids \cite{GrMe10a} and
shows how powerful the tools developped in
\cite{GrMe10a,DrJoOr15,GrJoMaMe18} are in the study of VB-algebroids
and double Lie algebroids. Along the way, the authors define the
\emph{comma-double Lie algebroid} defined by a Lie algebroid morphism,
which is an interesting structure in its own right. The dg-Lie
algebroid defined \cite{Voronov12} by the comma double Lie algebroid
associated to a Lie pair $A\hookrightarrow L$ is used by Sti\'enon,
Vitagliano and Xu \cite{StViXu21} in their extension to arbitrary Lie
pairs of the Kontsevich-Duflo type theorem for matched pairs in
\cite{LiStXu19}.

The following two
sections explain comma double Lie groupoids, 
as well as the generalised version of Brown's and Mackenzie's
equivalence of locally trivial double Lie algebroids with locally
trivial core diagrams \cite{BrMa92}. The construction steps of the
inverse functor in that setting are basically the same as for
the inverse functor in this paper.

\medskip
The equivalence of the category of transitive double Lie algebroids
with the one of transitive core diagrams, and the equivalence of
categories of locally trivial double Lie groupoids with the one of
locally trivial core diagrams of Lie groupoids established in
\cite{BrMa92} yield finally a very simple integration method for transitive
double Lie algebroids. Recall however that the notion of transitivity
discussed here corresponds to a weaker notion of local triviality than
the one used in \cite{BrMa92}, see Section
\ref{expand_BrMa92}. Section \ref{integration} explains this
integration of transitive double Lie algebroids to transitive double
Lie groupoids, which is now a straightforward exercise using many of
Mackenzie's prior results with coauthors
\cite{BrMa92,Mackenzie92,MaXu00}. It proves the following theorem (see
Section \ref{integration} and Theorem \ref{integration_thm} for the proof and more details).
\begin{thm}
  Let $(D,A,B,M)$ be a transitive double Lie algebroid with a
  transitive core diagram $A\mapsfrom C\mapsto B$ where $A$, $B$ and
  $C$ are all integrable Lie algebroids over $M$.
  Then $(D,A,B,M)$ integrates to a transitive double Lie groupoid,
  the core diagram of which integrates $A\mapsfrom C\mapsto B$.
  \end{thm}

\subsection{Comma double Lie groupoid and comma double Lie algebroid}
Given categories and functors
\[ \mathcal D\overset{F}{\longrightarrow} \mathcal C\overset{G}{\longleftarrow}\mathcal E,
\]
the \emph{comma category} $(F,G)$ has as objects the triples
$(d,e,\varphi)$ with $d\in{\rm Obj}(\mathcal D)$,
$e\in{\rm Obj}(\mathcal E)$ and
$\varphi\in{\rm Mor}_{\mathcal C}(F(d),G(e))$, see \cite{MacLane98}.
Its arrows $(d,e,\varphi)\to (d',e',\varphi')$ are the pairs
$(h,l)\in {\rm Mor}_{\mathcal D}(d,d')\times {\rm Mor}_{\mathcal
  E}(e,e')$ such that
$$ 
\xymatrix{
F(d)\ar[d]_{\varphi}\ar[r]^{F(h)}&F(d')\ar[d]^{\varphi'}\\
G(e)\ar[r]_{G(l)}&G(e')
}
$$
commutes. If defined, the composition $(h',l')\circ (h,l)$ is
$(h'h, l'l)$. If $K\rr M$ and $G\rr M$ are two (Lie) groupoids over a
common base $M$ and $\Phi\colon K\to G$ is a (Lie) groupoid morphism
fixing the base, then the diagram
\[   K\overset{\Phi}{\longrightarrow}G\overset{\Phi}{\longleftarrow}K,
\]
defines as above the comma category $(\Phi,\Phi)$, which has a
\emph{double} (Lie) groupoid structure with sides $K$ and $G$ and with core
$K$.  Denote the space of arrows of $(\Phi,\Phi)$ by $\Gamma$.  The
objects of $(\Phi,\Phi)$ are hence triples
$(m_1,g,m_2)\in M\times G\times M$ with $\s(g)=m_1$ and
$\tg(g)=m_2$. Hence, the objects of $(\Phi,\Phi)$ are just the arrows
of $G\rr M$, and the elements of $\Gamma$ (that is, the arrows of
$(\Phi,\Phi)$) are written as triples
$(k_2,g,k_1)\in K\times G\times K$ such that $\s(k_2)=\tg(g)$ and
$\s(k_1)=\s(g)$.
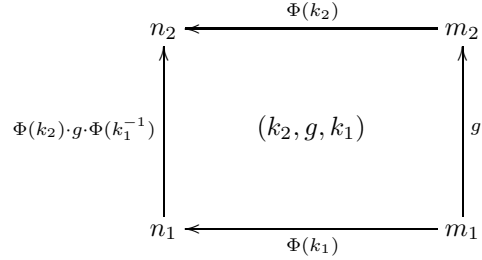
$$ 
\xymatrix{
  n_2&&m_2\ar[ll]_{\Phi(k_2)}\\
  &(k_2,g,k_1)&\\
n_1\ar[uu]^{\Phi(k_2)\cdot g\cdot \Phi(k_1\inv)}&&m_1\ar[ll]^{\Phi(k_1)}\ar[uu]_{g}
}
$$
This diagram is pictured in the usual manner for elements of a double
groupoid
\cite{BrMa92} 
and will just be written 
$$ 
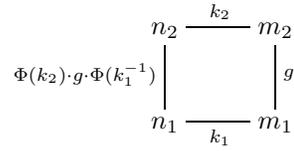
\xymatrix{
  n_2&m_2\ar@{-}[l]_{k_2}\\
n_1\ar@{-}[u]^{\Phi(k_2)\cdot g\cdot \Phi(k_1\inv)}&m_1\ar@{-}[l]^{k_1}\ar@{-}[u]_{g}
}
$$
in the following, and the left vertical arrow will not be labeled since
it is clear from the rest of the diagram.  The source and target maps
$\Gamma\to G$ are the maps $(k_2,g,k_1)\mapsto g$ and
$(k_2,g,k_1)\mapsto \Phi(k_2)\cdot g\cdot \Phi(k_1\inv)$, respectively.
The source and target maps $\Gamma\to K$ are the maps
$(k_2,g,k_1)\mapsto k_1$ and $(k_2,g,k_1)\mapsto k_2$, respectively.  The
composition over $G$ is
\[(k_2',\Phi(k_2)\cdot g\cdot \Phi(k_1\inv),k_1')\cdot_G(k_2,g,k_1)=(k_2'k_2, g, k_1'k_1)
\]
and the composition over $K$ is
\[ (k_3,g', k_2)\cdot_K(k_2,g,k_1)=(k_3,g'g,k_1).
\]
It is easy to check that given a composable square as below,
hence with $g=\Phi(l_3)g'\Phi(l_2\inv)$ and $h=\Phi(l_2)h'\Phi(l_1\inv)$,
the order of the horizontal and vertical multiplications is not relevant, as pictured
below. Hence, $(\Phi,\Phi)$ is a double groupoid.
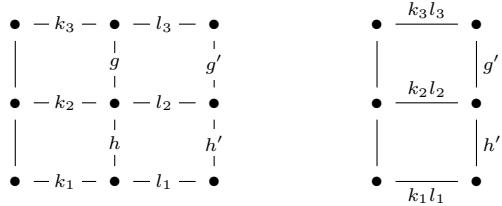
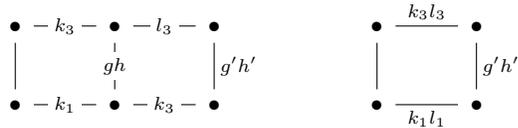
\[\begin{tikzcd}
	\bullet & \bullet & \bullet && \bullet & \bullet \\
	\bullet & \bullet & \bullet && \bullet & \bullet \\
	\bullet & \bullet & \bullet && \bullet & \bullet \\
	{} && {} \\
	\bullet & \bullet & \bullet && \bullet & \bullet \\
	\bullet & \bullet & \bullet && \bullet & \bullet
	\arrow["{k_3}"{description}, no head, from=1-1, to=1-2]
	\arrow["{g'}"{description}, no head, from=1-3, to=2-3]
	\arrow["{l_2}"{description}, no head, from=2-3, to=2-2]
	\arrow["g"{description}, no head, from=1-2, to=2-2]
	\arrow["{k_2}"{description}, no head, from=2-1, to=2-2]
	\arrow[no head, from=1-1, to=2-1]
	\arrow[no head, from=2-1, to=3-1]
	\arrow["{k_1}"{description}, no head, from=3-1, to=3-2]
	\arrow["{l_1}"{description}, no head, from=3-2, to=3-3]
	\arrow["{h'}"{description}, no head, from=2-3, to=3-3]
	\arrow["h"{description}, no head, from=2-2, to=3-2]
	\arrow["{l_3}"{description}, no head, from=1-2, to=1-3]
	\arrow["{g'}", no head, from=1-6, to=2-6]
	\arrow["{h'}", no head, from=2-6, to=3-6]
	\arrow["{l_3}"{description}, no head, from=5-2, to=5-3]
	\arrow["{g'h'}", no head, from=5-3, to=6-3]
	\arrow["{k_3}"{description}, no head, from=6-2, to=6-3]
	\arrow["gh"{description}, no head, from=5-2, to=6-2]
	\arrow["{k_3}"{description}, no head, from=5-1, to=5-2]
	\arrow["{k_1}"{description}, no head, from=6-1, to=6-2]
	\arrow[no head, from=5-1, to=6-1]
	\arrow["{k_3l_3}", no head, from=1-5, to=1-6]
	\arrow[no head, from=1-5, to=2-5]
	\arrow["{k_2l_2}", no head, from=2-5, to=2-6]
	\arrow[no head, from=2-5, to=3-5]
	\arrow["{k_1l_1}"', no head, from=3-5, to=3-6]
	\arrow["{k_3l_3}", no head, from=5-5, to=5-6]
	\arrow["{g'h'}", no head, from=5-6, to=6-6]
	\arrow["{k_1l_1}", no head, from=6-6, to=6-5]
	\arrow[no head, from=6-5, to=5-5]
      \end{tikzcd}\] If $\Phi\colon K\to G$ was a morphism of Lie
    groupoids, it is not difficult to check that $(\Phi,\Phi)$ is a
    double \emph{Lie} groupoid \cite{BrMa92}.

The core of $(\Gamma,K,G,M)$ consists of elements $(k,1_m,1_m)$
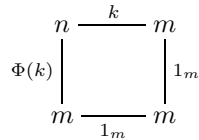
$$ 
\xymatrix{
  n&m\ar@{-}[l]_{k}\\
m\ar@{-}[u]^{\Phi(k)}&m\ar@{-}[l]^{1_m}\ar@{-}[u]_{1_m}
}
$$
with multiplication given by filling the top right and bottom left
squares in the diagram below and multiplying all obtained squares
together,
    \[\begin{tikzcd}
	{{\rm t}(k)} & m & {m'} \\
	m & m & {m'} \\
	{m'} & {m'} & {m'}
	\arrow["{\Phi(k)}", no head, from=2-1, to=1-1]
	\arrow["k"{description}, no head, from=1-1, to=1-2]
	\arrow["{1_m}"{description}, no head, from=1-2, to=2-2]
	\arrow["{1_m}"{description}, no head, from=2-1, to=2-2]
	\arrow["{\Phi(k')}", no head, from=3-1, to=2-1]
	\arrow["{1_{m'}}"{description}, no head, from=3-1, to=3-2]
	\arrow["{1_{m'}}"{description}, no head, from=3-2, to=3-3]
	\arrow["{k'}"{description}, no head, from=2-2, to=2-3]
	\arrow["{\Phi(k')}"', no head, from=2-2, to=3-2]
	\arrow["{1_{m'}}"{description}, no head, from=2-3, to=3-3]
	\arrow["{k'}"{description}, no head, from=1-2, to=1-3]
	\arrow["{1_{m'}}"{description}, no head, from=1-3, to=2-3]
      \end{tikzcd}\]
     hence yielding
    \[ (k,1_m,1_m)\cdot(k', 1_{m'}, 1_{m'})=(kk', 1_{m'}, 1_{m'}).
      \]
      As a consequence, the core (Lie) groupoid is isomorphic to
      $K\rr M$ as a (Lie) groupoid.  Note that $\Gamma\rr G$ is the
      action groupoid of the (Lie) groupoid action of
      $K\times K\rr M\times M$ on $(\tg,s)\colon G\to M\times M$,
      $(k_2,k_1)\cdot g=\Phi(k_2)g\Phi(k_1\inv)$.

\medskip

Section \ref{sec:comma} constructs analogously a double Lie algebroid from a
Lie algebroid morphism
$$ 
\xymatrix{
  C\ar[dr]\ar[rr]^{\da}&& A\ar[dl]\\
&M&
}
$$
The morphism $\da$ induces an action of the Lie algebroid $TC\to TM$
on the anchor $\rho_A\colon A\to TM$ of $A$. The total space
$R:=TC\oplus_{TM}A$ of the action Lie algebroid $TC\oplus_{TM}A\to A$
of this action carries as well automatically the pullback Lie algebroid
structure from $A\to M$ under the smooth map $q_C\colon C\to M$. It is
hence also a Lie algebroid over $C$, and these two Lie algebroid
structures define a double Lie algebroid
$$ 
\xymatrix{
  TC\oplus_{TM}A\ar[d]\ar[r]& A\ar[d]\\
C\ar[r]&M
}
$$
with sides $A$ and $C$ and with core $C$. This double Lie algebroid is
the \emph{comma double Lie algebroid} defined by the morphism
$\da\colon C\to A$. The details of this construction are given in
Section \ref{sec:comma}. The core diagram of this comma double Lie algebroid is
pictured in the following diagram.
$$ 
\xymatrix{
&C\ar[r]_{\id}\ar[d]^{\da}&C\\
&A&
}
$$

  \subsection{Transitive double Lie groupoids and their core
    diagrams}\label{expand_BrMa92}
  A double Lie groupoid is called here \emph{transitive} is the Lie
  groupoid morphisms in its core diagram are surjective submersions,
  hence fibrations of Lie groupoids. The core diagram is then also
  called \emph{transitive}.  In particular, locally trivial double Lie
  groupoids in the sense of \cite{BrMa92} are transitive double Lie
  groupoids.  A careful study of the proof of the main result in
  \cite{BrMa92}, establishing an equivalence between locally trivial
  core diagrams and locally trivial double Lie groupoids, reveal that
  this equivalence is true in the more general setting of transitive
  double Lie groupoids versus transitive core diagrams.

  For completeness, and because the construction of the equivalence in
  \cite{BrMa92} serves as a model for the construction of the
  equivalence of the categories of transitive double Lie algebroids
  with transitive core diagrams, it is sketched in this section.

  Consider a transitive core diagram
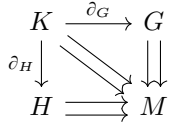
\begin{figure}[h]
  \[\begin{tikzcd}
	K & G \\
	H & M
	\arrow[shift left=1, from=1-2, to=2-2]
	\arrow[shift right=1, from=1-2, to=2-2]
	\arrow[shift right=1, from=2-1, to=2-2]
	\arrow[shift left=1, from=2-1, to=2-2]
	\arrow[shift right=1, from=1-1, to=2-2]
	\arrow[shift left=1, from=1-1, to=2-2]
	\arrow["{\partial_G}", from=1-1, to=1-2]
	\arrow["{\partial_H}"', from=1-1, to=2-1]
      \end{tikzcd}\]
    \caption{Core diagram of Lie groupoids}\label{cd_lg}
  \end{figure}
  of Lie groupoids; i.e.~a diagram as in Figure \ref{cd_lg} of Lie
  groupoid morphisms, such that $\da_G$ and $\da_H$ are surjective
  submersions and the subgroupoids $\ker(\da_G)$ and $\ker(\da_H)$
  commute in $K$.  Build the comma double Lie groupoid $(\da_G,\da_G)$
    \[\begin{tikzcd}
	\Gamma & G \\
	K & M
	\arrow[shift left=1, from=1-2, to=2-2]
	\arrow[shift right=1, from=1-2, to=2-2]
	\arrow[shift right=1, from=2-1, to=2-2]
	\arrow[shift left=1, from=2-1, to=2-2]
	\arrow[shift right=1, from=1-1, to=1-2]
	\arrow[shift right=1, from=1-1, to=2-1]
	\arrow[shift left=1, from=1-1, to=2-1]
	\arrow[shift left=1, from=1-1, to=1-2]
      \end{tikzcd}\]
    with core $K$.
    Next, $G\to M$ acts on $K^G:=\ker(\da_H)\subseteq K$ via
    \[ \rho\colon G\times_{\s, \tg} K^G\to K^G, \qquad \rho(g)(\kappa)=k\kappa k\inv
    \]
    for any $k\in K$ such that $\da_G(k)=g$.
    Consider the closed, embedded wide and normal sugbroupoid
    \begin{equation}\label{def_N}
      N:=\{(\kappa_2,g,\kappa_1)\in K^G\times G\times K^G\mid \rho(g)(\kappa_1)=\kappa_2\}
    \end{equation}
    of $\Gamma\rr G$. The quotient $\Theta=\Gamma/N$ has then a Lie
    groupoid structure over $G$. The elements of $\Theta$ are classes
    \[ \langle k_2, g, k_1\rangle=\{(k_2\kappa_2, g, k_1\kappa_1)\mid (\kappa_2,g,\kappa_1)\in N\}
    \]
    with $(k_2,g,k_1)\in\Gamma$.
    Set $\s_H,\tg_H\colon \Gamma/N\to H$ 
    \[ \s_H\langle k_2, g, k_1\rangle=\da_H(k_1), \qquad \tg_H\langle k_2, g, k_1\rangle=\da_H(k_2),
    \]
    and a partial multiplication $\cdot_H\colon \Theta\times_{\s_H,\tg_H}\Theta\to \Theta$ by 
    \[ \langle k_2', g', k_1'\rangle\cdot_H\langle k_2, g,
      k_1\rangle=\langle k_2', g', k_1'\rangle\cdot_H\langle k_1',
      g, k_1\rho(g\inv)(\kappa)\rangle=\langle k_2', g'g, k_1\rho(g\inv)(\kappa)\rangle.
    \]
    Here $\kappa\in K^G$ is such that $k_2\kappa=k_1'$. It is easy to
    check that $\Theta/N\rr H$ becomes a groupoid with these structure maps and with the inversion
    \[ \langle k_2, g, k_1\rangle\mapsto \langle k_1, g\inv,k_2\rangle
    \]
    and the unit inclusion
    \[ h\mapsto \langle k, 1_{\s(h)},k\rangle
    \]
    for any $k\in K$ such that $\da_H(k)=h$.
    Since those structure maps are defined such that
    \[\begin{tikzcd}
	\Gamma & {\Theta=\Gamma/N} \\
	K & H
	\arrow[shift left=1, from=1-2, to=2-2]
	\arrow[shift right=1, from=1-2, to=2-2]
	\arrow["{\partial_H}"', from=2-1, to=2-2]
	\arrow["\pi", from=1-1, to=1-2]
	\arrow[shift right=1, from=1-1, to=2-1]
	\arrow[shift left=1, from=1-1, to=2-1]
      \end{tikzcd}\] is a morphism of groupoids and the projections
    $\pi\colon \Gamma\to\Theta$ and $\da_H\colon K\to H$ are
    surjective submersions, it is easy to check that $\Theta\rr H$ is
    a Lie groupoid. The interchange law and the surjectivity of the
    double source map are also easily deduced from the one in
    $(\da_G,\da_G)$, and $(\Theta, G,H,M)$ is a double Lie
    groupoid. Its core consists in classes
    $\langle k, 1_{\s(k)},\kappa\rangle=\langle k\kappa\inv,
    1_{\s(k)},1_{\s(k)}\rangle$, and is obviously isomorphic to $K$ as
    a Lie groupoid. The core diagram of $(\Theta, G,H,M)$ is hence
    again \[\begin{tikzcd}
	K & G \\
	H & M \arrow[shift left=1, from=1-2, to=2-2] \arrow[shift
        right=1, from=1-2, to=2-2] \arrow[shift right=1, from=2-1,
        to=2-2] \arrow[shift left=1, from=2-1, to=2-2] \arrow[shift
        right=1, from=1-1, to=2-2] \arrow[shift left=1, from=1-1,
        to=2-2] \arrow["{\partial_G}", from=1-1, to=1-2]
        \arrow["{\partial_H}"', from=1-1, to=2-1]
      \end{tikzcd}\]
    by construction.

    Now given a transitive double Lie groupoid
     \[\begin{tikzcd}
	S& G \\
	H & M
	\arrow[shift left=1, from=1-2, to=2-2]
	\arrow[shift right=1, from=1-2, to=2-2]
	\arrow[shift right=1, from=2-1, to=2-2]
	\arrow[shift left=1, from=2-1, to=2-2]
	\arrow[shift right=1, from=1-1, to=1-2]
	\arrow[shift right=1, from=1-1, to=2-1]
	\arrow[shift left=1, from=1-1, to=2-1]
	\arrow[shift left=1, from=1-1, to=1-2]
      \end{tikzcd}\] with core $K$ and core diagram as in Figure
    \ref{cd_lg}, set $\Psi_S\colon \Theta\to S$,
    $\langle k_2, g, k_1\rangle\mapsto k_2\cdot_H 1^S_g\cdot_Hk_1^{-1,
      H}\in S$. Then $\Psi_S$ is an isomorphism of double Lie groupoids,
    see \cite{BrMa92} -- where, again, only the surjectivity of
    $\da_G$ and $\da_H$ are needed.

    A morphism $\Phi\colon \Gamma_1\to \Gamma_2$ of double Lie
    groupoids with side morphisms $\varphi_G\colon G_1\to G_2$,
    $\varphi_H\colon H_1\to H_2$ and with core morphism
    $\varphi_K\colon K_1\to K_2$ induces a morphism of the
    corresponding core diagrams as in the following diagram.
    \begin{equation}\label{cor_mor_gpd}\begin{tikzcd}
	{K_1} &&&& {K_2} \\
	&& {G_1} &&&& {G_2} \\
	{H_1} &&&& {H_2} \\
	&& M &&&& M
	\arrow["{\varphi_K}", from=1-1, to=1-5]
	\arrow["{\varphi_H}"{pos=0.7}, from=3-1, to=3-5]
	\arrow["{\partial_{H_1}}"{description}, from=1-1, to=3-1]
	\arrow["{\partial_{H_2}}"{description, pos=0.7}, from=1-5, to=3-5]
	\arrow["{\partial_{G_1}}"{description}, from=1-1, to=2-3]
	\arrow["{\partial_{G_2}}"{description}, from=1-5, to=2-7]
	\arrow[shift left=1, from=3-1, to=4-3]
	\arrow[shift right=1, from=3-1, to=4-3]
	\arrow[shift right=1, from=3-5, to=4-7]
	\arrow[shift left=1, from=3-5, to=4-7]
	\arrow[shift right=1, from=2-7, to=4-7]
	\arrow[shift left=1, from=2-7, to=4-7]
	\arrow[shift right=1, from=2-3, to=4-3]
	\arrow[shift left=1, from=2-3, to=4-3]
	\arrow[shift right=1, from=1-1, to=4-3]
	\arrow[shift left=1, from=1-1, to=4-3]
	\arrow[shift right=1, from=1-5, to=4-7]
	\arrow[shift left=1, from=1-5, to=4-7]
	\arrow["{\operatorname{id}_M}"{description}, from=4-3, to=4-7]
	\arrow["{\varphi_G}"{pos=0.3}, from=2-3, to=2-7]
      \end{tikzcd}\end{equation}  Consider conversely a morphism of
    transitive core diagrams (of Lie groupoids) as in
    \eqref{cor_mor_gpd}.  Let $\Gamma_i$ be the total space of the
    comma double Lie groupoid $(\da_{G_i},\da_{G_i})$ for $i=1,2$ and
    set $\Phi\colon \Gamma_1\to \Gamma_2$,
    $\Phi(k_1,g,k_2)=(\phi_K(k_1),\phi_G(g),\phi_K(k_2))$. The map
    $\Phi$ is obviously a morphism of double Lie groupoids, and a
    computation shows $\Phi(N_1)=N_2$ for the normal subgroupoids
    $N_i$ of $\Gamma_i\rr G_i$ defined as in \eqref{def_N} for
    $i=1,2$.  Therefore, it induces a morphism of the transitive
    double Lie groupoids $\overline{\Phi}\colon \Gamma_1/N_1\to \Gamma_2/N_2$ with core
    morphism again the one in \eqref{cor_mor_gpd}.

    Let $M$ be a smooth manifold and consider the category $\CDG(M)$
    of transitive core diagrams of Lie groupoids with base $M$, and
    the category $\DLG(M)$ of transitive double Lie groupoids with
    double base $M$. The functor $\mathcal C\colon \DLG(M)\to \CDG(M)$
    sends transitive double Lie groupoids and their morphisms to the
    corresponding transitive core diagrams of Lie groupoids and
    morphisms of core diagrams.  The above constructs a functor
    $\mathcal D\colon \CDG(M)\to \DLG(M)$ sending a transitive core
    diagram as in Figure \ref{cd_lg} to the double Lie groupoid
    $(\Theta,G,H,M)$ as above, and a morphism of core diagrams as in
    \eqref{cor_mor_gpd} to the morphism
    $\overline{\Phi}\colon \Theta_1\to \Theta_2$ as in the previous
    paragraph. Then $\mathcal C\circ \mathcal D$ is obviously the
    identity functor. The assigment sending each transitive double Lie
    groupoid $(S,G,H,M)$ to the isomorphism
    $\Psi_S\colon \mathcal D\mathcal C(S)\to S$ of double Lie
    groupoids defines a natural isomorphism
    $\Psi_\cdot\colon \mathcal D\circ \mathcal C\to
    \operatorname{id}_{\DLG}$.  Therefore, Brown and Mackenzie's
    equivalence of locally trivial core diagrams with locally trivial
    double Lie groupoids \cite{BrMa92} is extended to an equivalence
    of transitive core diagrams with transitive double Lie groupoids.

  \subsection*{Outline of the paper}
  Section \ref{sec:prelim} collects necessary knowledge on ideals in
  Lie algebroids, on VB-algebroids and representations up to homotopy
  and on double Lie algebroids and their morphisms. Section
  \ref{func_trdLA_codiag} defines transitive core diagrams and studies
  the crossed modules associated to a transitive core diagram. It
  describes the (obvious) functor from transitive double Lie
  algebroids to transitive core diagrams.

  Section \ref{sec:comma} defines the comma double Lie algebroid
  associated to a Lie algebroid morphism $C\to A$ (over a fixed
  base). It describes the two VB-algebroid structures of this double
  Lie algebroid via apropriate representations up to homotopy. Section
  \ref{sec:quotient} is the core of this paper: it factors the comma
  double Lie algebroid defined by one side of a transitive core
  diagram, by a sub-structure defined by the other side. This
  structure is an ideal in one of the VB-algebroids and an
  infinitesimal ideal system (with trivial fiber) in the other
  VB-algebroid.

  The equivalence between the categories of transitive double Lie
  algebroids and of transitive core diagrams is then established in
  Section \ref{eq_cat_la}. Section \ref{integration} uses this for
  proving an integration of transitive double Lie algebroids (with
  integrable sides and core).  The appendix collects some longer
  proofs, like the verification of the double Lie algebroid conditions
  \cite{GrJoMaMe18} for the comma double Lie algebroid and its
  quotient.
 
  \subsection*{Notation and conventions}

  All manifolds and vector bundles in this paper are smooth and real.
  Vector bundle projections are written $q_E\colon E\to M$, and
  $p_M\colon TM\to M$ for tangent bundles.  Given a section
  $\varepsilon$ of $E^*$, the map $\ell_\varepsilon\colon E\to \R$ is
  the linear function associated to it, i.e.~the function defined by
  $e_m\mapsto \langle \varepsilon(m), e_m\rangle$ for all $e_m\in E$.
  The set of global sections of a vector bundle $E\to M$ is denoted by
  $\Gamma(E)$, $\mx(M)=\Gamma(TM)$ is the space of smooth vector fields on a
  smooth manifold $M$, and
  $\Omega^\bullet(M)=\Gamma(\bigwedge^\bullet T^*M)$ is the space of smooth
  forms on $M$.

\subsection*{Acknowledgement}
This paper is the outcome of a joint project of the authors in 2012
and 2013, which had been left unfinished until the departure of the
second author in 2020. 
The studied problem was
Kirill Mackenzie's idea and this research builds on many of his
beautiful mathematical achievements. The first author dedicates to his
memory her work on finishing this paper.

\section{Preliminaries}\label{sec:prelim}
This section collects necessary background on infinitesimal ideal
systems in Lie algebroids \cite{JoOr14}, on double Lie algebroids and
their morphisms \cite{Mackenzie11}, and on the correspondence of
decomposed VB-algebroids with 2-term representations up to homotopy
\cite{GrMe10a}.
\subsection{Ideals in Lie algebroids}
\textbf{Infinitesimal ideal systems} \cite{JoOr14,Hawkins08} are
considered the right notion of ideal in Lie algebroids. They are the
infinitesimal version of \textbf{ideal systems}
\cite{HiMa90b,Mackenzie05}. These are exactly the kernels of
fibrations of Lie algebroids \cite{HiMa90b}.

\begin{df}\label{iis_def}
  Let $(q\colon A\to M, \rho,[\cdot\,,\cdot])$ be a Lie algebroid,
  $F\subseteq TM$ an involutive subbundle, $J\subseteq A$ a
  subbundle over $M$ such that $\rho(J)\subseteq F$, and $\nabla$ a
  flat $F$-connection on $A/J$ with the following
  properties:
\begin{enumerate}
\item If $a\in\Gamma(A)$ is $\nabla$-parallel\footnote{A section
    $a\in\Gamma(A)$ is said to be \textbf{$\nabla$-parallel} if
    $\nabla_X\bar a=0$ for all $X\in\Gamma(F)$. Here, $\bar a$ is
    the class of $a$ in $\Gamma(A/J)\simeq \Gamma(A)/\Gamma(J)$.},
  then $[a,j]\in\Gamma(J)$ for all $j\in\Gamma(J)$.
\item If $a,b\in\Gamma(A)$ are $\nabla$-parallel, then $[a,b]$ is also
  $\nabla$-parallel.
\end{enumerate}
Then  the triple
$(F,J,\nabla)$ is an \textbf{infinitesimal ideal system in $A$.}
\end{df}
The first axiom implies immediately that $J\subseteq A$ is a
subalgebroid of $A$, and the following property of $(F,J,\nabla)$ follows from (2) above.
\begin{enumerate}\setcounter{enumi}{2}
  \item If $a\in\Gamma(A)$ is  $\nabla$-parallel, then $\rho(a)$ is
  $\nabla^{F}$-parallel, where
  \[\nabla^{F}\colon\Gamma(F)\times\Gamma(TM/F)\to\Gamma(TM/F), \quad \nabla^{F}_X\bar Y=\overline{[X,Y]}\]
  is the Bott connection associated to $F$.
  \end{enumerate}
Infinitesimal ideal systems already appear in \cite{Hawkins08} (not
under this name) in the context of geometric quantization as the
infinitesimal version of polarizations on groupoids.

\begin{ex}[Flat connections on vector bundles]\label{geom_red}
  Let $E$ be a vector bundle over $M$ and let $K\subseteq E$ be a
  subbundle. Then any flat connection of an involutive subbundle
  $F_M\subseteq TM$ on $E/K$ defines an \emph{infinitesimal ideal
    system $(F_M, K, \nabla)$ in $E$}, i.e.~an infinitesimal ideal
  system in the trivial Lie algebroid
  $(E, \rho=0, [\cdot\,,\cdot]=0)$.  
  An infinitesimal ideal system in a Lie algebroid is therefore, by
  forgetting the Lie algebroid structure, automatically an
  infinitesimal ideal system in the underlying vector bundle.
\end{ex}

\begin{ex}
  \begin{enumerate}
  \item An \textbf{ideal} in a Lie algebroid $A\to M$ is a vector subbundle
    $I\subseteq A$ such that
    $[\Gamma(I),\Gamma(A)]\subseteq \Gamma(I)$. Given such an ideal
    $I$ in $A$, the triple $(0,I,0)$ is an infinitesimal ideal system
    in $A$.
  \item Let $A\to M$ be a Lie algebroid. Let $F_M\subseteq TM$ be an
    involutive subbundle and let
    $\nabla\colon \Gamma(F_M)\times\Gamma(A)\to\Gamma(A)$ be a flat
    connection such that $[a,b]$ is $\nabla$-parallel for
    $a,b\in\Gamma(A)$ $\nabla$-parallel.  Then
    $(F_M,0,\nabla)$ is an infinitesimal ideal system in $A$.
    \end{enumerate}
\end{ex}

\begin{df}
  A morphism
  \begin{equation*}
\begin{xy}
  \xymatrix{
    A\ar[d]_{q_E}\ar[r]^{\Phi}&B\ar[d]^{q_B}\\
    M\ar[r]_{\phi}&N }
\end{xy}
\end{equation*}
of Lie algebroids $A\to M$ and $B\to N$ is a \textbf{fibration of Lie
  algebroids} if the underlying morphism $(\Phi,\phi)$ of vector
bundles is a fibration of vector bundles, i.e.~$\phi$ is a surjective
submersion and $\Phi$ is fiberwise surjective.
\end{df}

Let $E\to M$ be a vector bundle, and let $J\subseteq E$ be a subbundle
and $F_M\subseteq TM$ an involutive subbundle. Assume that there is a
flat $F_M$-connection on $E/J$.  Then define an equivalence relation
$\sim_\nabla$ on $E/J$ as follows: for $e,e'\in E$ and their classes
$\bar e, \bar e'\in E/J$, $\bar e\sim_\nabla \bar e'$ if and only if
$q_E(e)$ and $q_E(e')$ are in the same leaf $L$ of $F_M$ in $M$ and
$\bar e'$ is the image of $\bar e$ under $\nabla$-parallel transport
along a path in the leaf $L$ joining $q_E(e)$ and $q_E(e')$.  The
quotient of $E$ by this equivalence relation is written
$(E/J)/\nabla$. The projections $E\to (E/J)/\nabla$ and $M\to M/F_M$
are written $\pi$ and $\pi_M$, repectively.  The vector bundle map
$q_E\colon E\to M$ factors to a map
$[q_E]\colon (E/J)/\nabla\to M/F_M$.

The following theorem shows that infinitesimal ideal systems in Lie
algebroids define in this manner quotients of Lie algebroids, up to
some topological obstructions \cite{JoOr14}. The paper \cite{DrJoOr15}
proves that an infinitesimal ideal system defines a sub-representation
(up to homotopy) of the adjoint representation of the Lie algebroid,
after the choice of an extension of the infinitesimal ideal system
connection.  These two results suggest that indeed, an infinitesimal
ideal systems is the right notion of ideal in a Lie algebroid.  The
following result from \cite{JoOr14} is slightly reformulated here in
order to emphasize the fact the two topological obstructions ensure
the existence of a quotient \emph{vector bundle}. The `reduced' Lie algebroid
structure then follows immediately.
\begin{thm}\label{red_lie_alg}
  \begin{enumerate}
  \item Let $q_E\colon E\to M$ be a smooth vector bundle and let
    $(F_M,J,\nabla)$ be an infinitesimal ideal system in $E$ (see
    Example \ref{geom_red}). Assume that $\bar M=M/F_M$ is a smooth
    manifold and that $\nabla$ has trivial holonomy.  Then the
    quotient $(E/J)/\nabla\to M/F_M$ carries a vector bundle structure
    such that the projection $(\pi,\pi_M)$ is a fibration of vector
    bundles:
\begin{equation}\label{reduced_vb}
\begin{xy}
  \xymatrix{
    E\ar[d]_{q_E}\ar[r]^{\pi}&(E/J)/\nabla\ar[d]^{[q_E]}\\
    M\ar[r]_{\pi_M}&M/F_M }
\end{xy}
\end{equation}
\item Let now $A\to M$ be a Lie algebroid and let $(F_M,J,\nabla)$ be
  an infinitesimal ideal system in $A$. If the quotient vector bundle
  $(A/J)/\nabla\to M/F_M$ exists as in \eqref{reduced_vb}, then it
  carries a unique Lie algebroid structure such that
  \eqref{reduced_vb} is a fibration of Lie algebroids.
  \end{enumerate}
\end{thm}

An infinitesimal ideal system $(F_M,J,\nabla)$ is defined by the
kernel of a fibration of Lie algebroids if and only if it integrates
to an ideal system in the sense of Higgins and Mackenzie
\cite{HiMa90b, Mackenzie05}, see \cite{JoOr14}.

The reduced Lie algebroid structure in (2) of Theorem
\ref{red_lie_alg} is defined as follows.  Write $\Gamma(A)^\nabla$ for
the space of sections of $A\to M$ the class in $A/J$ of which is
$\nabla$-flat. Then $\Gamma(A)^\nabla$ is an $\mathbb R$-Lie algebra
and a $C^\infty(M)^{F_M}=\pi_M^*C^\infty(\bar M)$-module. The space
$\Gamma(J)$ is an ideal in $\Gamma(A)^\nabla$.  By definition, for
each $a\in\Gamma(A)^\nabla$ there is a section
$\bar a\in\Gamma(\bar A)$ such that $\pi\circ a=\bar a\circ \pi_M$
(this is written $a\sim_{\pi}\bar a$ for short). In particular, the
sections $j\in\Gamma(J)$ project in this manner to the zero section of
$\bar A\to \bar M$.  Conversely for each section
$\bar a\in\Gamma(\bar A)$ there is a $\nabla$-flat $a\in\Gamma(A)$
such that $a\sim_\pi\bar a$.  The Lie algebroid bracket and the anchor
on $\bar A\to \bar M$ are then defined by
\[  \Gamma(A)^\nabla\ni a_i \sim_\pi \bar a_i\in\Gamma(\bar A), \, i=1,2\qquad  \Rightarrow \qquad [a_1,a_2]\sim_\pi [\bar a_1,\bar a_2],
\]
and
\[ \overline{\rho}(\pi(a))=T\pi_M\rho(a)
\]
for all $a\in A$.

\subsection{VB-algebroids and representations up to homotopy}
This section starts by briefly recalling the definitions of double
vector bundles, of their \textbf{linear} and \textbf{core} sections,
and of their \textbf{linear splittings} and \textbf{lifts}. The reader
is referred to \cite{Pradines77,Mackenzie05,GrMe10a} for more detailed
treatments. A \textbf{double vector bundle} is a commutative square
\begin{equation*}
\begin{xy}
\xymatrix{
D \ar[r]^{\pi_B}\ar[d]_{\pi_A}& B\ar[d]^{q_B}\\
A\ar[r]_{q_A} & M}
\end{xy}
\end{equation*}
of vector bundles such that
\begin{equation}\label{add_add} (d_1+_Ad_2)+_B(d_3+_Ad_4)=(d_1+_Bd_3)+_A(d_2+_Bd_4)
\end{equation}
for $d_1,d_2,d_3,d_4\in D$ with $\pi_A(d_1)=\pi_A(d_2)$,
$\pi_A(d_3)=\pi_A(d_4)$ and $\pi_B(d_1)=\pi_B(d_3)$,
$\pi_B(d_2)=\pi_B(d_4)$.  Here, $+_A$ and $+_B$ are the additions in
$D\to A$ and $D\to B$, respectively. In particular, $(\pi_A,q_B)$ and
$(\pi_B,q_A)$ must be vector bundle morphisms.  The vector bundles $A$
and $B$ are called the \textbf{side bundles}. The \textbf{core} $C$ of
a double vector bundle is the intersection of the kernels of $\pi_A$
and of $\pi_B$. From \eqref{add_add} follows easily the existence of a
natural vector bundle structure on $C$ over $M$. The inclusion
$C \hookrightarrow D$ is denoted by
$ C_m \ni c \longmapsto \overline{c} \in \pi_A^{-1}(0^A_m) \cap
\pi_B^{-1}(0^B_m).  $

The space of sections
$\Gamma_B(D)$ is generated as a $C^{\infty}(B)$-module by two
special classes of sections (see \cite{Mackenzie11}), the
\textbf{linear} and the \textbf{core sections} which are now described.
For a section $c\colon M \rightarrow C$, the corresponding
\textbf{core section} $c^\dagger\colon B \rightarrow D$ is defined as
$c^\dagger(b_m) = 0^D_{\vphantom{1}_{b_m}} +_A \overline{c(m)}$, $m \in M$, $b_m \in B_m$.
The corresponding core section $A\to D$ is written $c^\dagger$ also,
relying on the argument to distinguish between them. The space of core
sections of $D$ over $B$ is written $\Gamma_B^c(D)$.

A section $\xi\in \Gamma_B(D)$ is called \textbf{linear} if $\xi\colon
B \rightarrow D$ is a bundle morphism from $B \rightarrow M$ to $D
\rightarrow A$ over a section $a\in\Gamma(A)$.  The space of linear
sections of $D$ over $B$ is denoted by $\Gamma^\ell_B(D)$.  Given
$\psi\in \Gamma(B^*\otimes C)$, there is a linear section
$\widetilde{\psi}\colon B\to D$ over the zero section $0^A\colon M\to A$
given by
$\widetilde{\psi}(b_m) = \tilde{0}_{b_m}+_A \overline{\psi(b_m)}$.
Such a section is called a \textbf{core-linear section}. 

\begin{ex}\label{trivial_dvb}
  Let $A, \, B, \, C$ be vector bundles over $M$ and consider
  $D=A\times_M B \times_M C$. With the vector bundle structures
  $D=q^{!}_A(B\oplus C) \to A$ and $D=q_B^{!}(A\oplus C) \to B$, one
  finds that $(D; A, B; M)$ is a double vector bundle called the
  \textbf{decomposed double vector bundle with
    core $C$}. The core sections are given by
$$
c^\dagger\colon b_m \mapsto (0^A_m, b_m, c(m)), \text{ where } m \in M, \, b_m \in
B_m, \, c \in \Gamma(C),
$$
and similarly for $c^\dagger\colon A\to D$.  The space of linear
sections $\Gamma^\ell_B(D)$ is naturally identified with
$\Gamma(A)\oplus \Gamma(B^*\otimes C)$ via
$$
(a, \psi): b_m \mapsto (a(m), b_m, \psi(b_m)), \text{ where } \psi \in
\Gamma(B^*\otimes C), \, a\in \Gamma(A).
$$

In particular, the fibered product $A\times_M B$ is a double vector
bundle over the sides $A$ and $B$ and has core $\{0\}\times M\to M$.
\end{ex}

A \textbf{linear splitting} of $(D; A, B; M)$ is an injective morphism
of double vector bundles $\Sigma\colon A\times_M B\hookrightarrow D$
over the identity on the sides $A$ and $B$.  The existence of a linear
splitting is proved e.g.~in \cite{MePi20}, see also \cite{HeJo20} for
historical remarks.  A linear splitting $\Sigma$ of $D$ is also
equivalent to a splitting $\sigma_A$ of the short exact sequence of
$C^\infty(M)$-modules
\begin{equation}\label{fat_seq_gamma}
0 \longrightarrow \Gamma(B^*\otimes C) \hookrightarrow \Gamma^\ell_B(D) 
\longrightarrow \Gamma(A) \longrightarrow 0,
\end{equation}
where the third map is the map that sends a linear section $(\xi,a)$
to its base section $a\in\Gamma(A)$.  The splitting $\sigma_A$ is
called a \textbf{horizontal lift}. Given $\Sigma$, the horizontal lift
$\sigma_A\colon \Gamma(A)\to \Gamma_B^\ell(D)$ is given by
$\sigma_A(a)(b_m)=\Sigma(a(m), b_m)$ for all $a\in\Gamma(A)$ and
$b_m\in B$.  By the symmetry of a linear splitting, a
lift $\sigma_A\colon \Gamma(A)\to\Gamma_B^\ell(D)$ is equivalent to a
lift $\sigma_B\colon \Gamma(B)\to \Gamma_A^\ell(D)$. 

  \bigskip
Let $q_E\colon E\to M$ be a vector bundle.  Then the tangent bundle
$TE$ has two vector bundle structures; one as the tangent bundle of
the manifold $E$, and the second as a vector bundle over $TM$. The
structure maps of $TE\to TM$ are the derivatives of the structure maps
of $E\to M$.
\begin{equation*}
\begin{xy}
\xymatrix{
TE \ar[d]_{Tq_E}\ar[r]^{p_E}& E\ar[d]^{q_E}\\
 TM\ar[r]_{p_M}& M}
\end{xy}
\end{equation*} 
The space $TE$ is a double vector bundle with core bundle
$E \to M$. The map $\bar{}\,\colon E\to p_E^{-1}(0^E)\cap
(Tq_E)^{-1}(0^{TM})$ sends $e_m\in E_m$ to $\bar
e_m=\left.\frac{d}{dt}\right\an{t=0}te_m\in T_{0^E_m}E$.
Hence the core vector field corresponding to $e \in \Gamma(E)$ is the
vertical lift $e^{\uparrow}\colon E \to TE$, i.e.~the vector field with
flow $\phi^{e^\uparrow}\colon E\times \R\to E$, $\phi_t(e'_m)=e'_m+te(m)$. An
element of $\Gamma^\ell_E(TE)=\mx^\ell(E)$ is called a \textbf{linear
  vector field}. It is well-known (see e.g.~\cite{Mackenzie05}) that a
linear vector field $\xi\in\mx^l(E)$ covering $X\in\mx(M)$ corresponds
to a derivation $D\colon \Gamma(E) \to \Gamma(E)$ over $X\in
\mx(M)$. The precise correspondence is
given by the following equations
\begin{equation}\label{ableitungen}
\xi(\ell_{\varepsilon}) 
= \ell_{D^*(\varepsilon)} \,\,\,\, \text{ and }  \,\,\, \xi(q_E^*f)= q_E^*(X(f))
\end{equation}
for all $\varepsilon\in\Gamma(E^*)$ and $f\in C^\infty(M)$, where
$D^*\colon\Gamma(E^*)\to\Gamma(E^*)$ is the dual derivation to $D$
\cite[\S3.4]{Mackenzie05}.  The linear vector field in $\mx^l(E)$
corresponding in this manner to a derivation $D$ of $\Gamma(E)$ is
written $\widehat D$.  Given a derivation $D$ over $X\in\mx(M)$, the
explicit formula for $\widehat D$ is
\begin{equation}\label{explicit_hat_D}
\widehat D(e_m)=T_me(X(m))+_E\left.\frac{d}{dt}\right\an{t=0}(e_m-tD(e)(m))
\end{equation}
for $e_m\in E$ and any $e\in\Gamma(E)$ such that $e(m)=e_m$.  The
choice of a linear splitting $\Sigma$ for $(TE; TM, E; M)$ is
equivalent to the choice of a linear connection $\nabla\colon \mx(M)\times\Gamma(E)\to\Gamma(E)$: 
the corresponding lift $\sigma_{TM}^\nabla\colon\mx(M)\to\mx^l(E)$ is defined by
$X\mapsto \widehat{\nabla_X}$. 
It is easy to see, using the equalities
in \eqref{ableitungen}, that
\begin{equation}\label{Lie_bracket_over_VB}
\begin{split}
\left[\widehat{\nabla_X}, \widehat{\nabla_Y}\right]&=\widehat{\nabla_{[X,Y]}}-\widetilde{R_\nabla(X,Y)},\qquad 
\left[\widehat{\nabla_X}, e^\uparrow\right]=(\nabla_Xe)^\uparrow,\qquad 
\left[e_1^\uparrow,e_2^\uparrow\right]=0
\end{split}
\end{equation}
for all $X,Y\in\mx(M)$ and $e, e_1, e_2\in\Gamma(E)$.

The splitting $\Sigma^\nabla$ defined by a linear $TM$-connection on $E$ is, as above,
equivalent as well to a linear horizontal lift
$\sigma^\nabla_E\colon\Gamma(E)\to\Gamma^l_{TM}(TE)$. It is given by
\[ \sigma^\nabla_E(e)=Te -\widetilde{\nabla_\cdot e}\qquad \text{ for all } e\in\Gamma(E).
  \]
  The core section of $TE\to TM$ defined by $e\in\Gamma(E)$ is written
  $e^\dagger\in\Gamma^c_{TM}(TE)$ and explicitly defined by
  $e^\dagger(v_m)=T_m0^Ev_m-\left.\frac{d}{dt}\right\an{t=0}0^E_m+te(m)$
  for all $v_m\in TM$.

  \bigskip
  
  The remainder of this section collects necessary background on
  VB-algebroids and representations up to homotopy.  A double vector
  bundle $(D;A,B;M)$ is a \textbf{VB-algebroid} (\cite{Mackenzie98x};
  see also \cite{GrMe10a}) if $D\to B$ is equipped with a Lie algebroid
  structure such that the anchor $\Theta_B\colon D \to TB$ is a
  morphism of double vector bundles over $\rho_A\colon A \to TM$ on
  one side and the Lie bracket on $D\to B$ is linear:
\begin{equation*} [\Gamma^\ell_B(D), \Gamma^\ell_B(D)] \subset
  \Gamma^\ell_B(D), \qquad [\Gamma^\ell_B(D), \Gamma^c_B(D)] \subset
  \Gamma^c_B(D), \qquad [\Gamma^c_B(D), \Gamma^c_B(D)]= 0.
\end{equation*}
The vector bundle $A\to M$ is then also a Lie algebroid, with anchor
$\rho_A$ and bracket defined as follows: if $\xi_1,
\xi_2\in\Gamma^\ell_B(D)$ are linear over $a_1,a_2\in\Gamma(A)$, then
the bracket $[\xi_1,\xi_2]$ is linear over $[a_1,a_2]$.

\medskip

Let $E_0,E_1$ be two vector bundles over the same base $M$ as $A$. A
\textbf{$2$-term representation up to homotopy of $A$ on $E_0\oplus
  E_1$} \cite{ArCr12,GrMe10a} is the collection of
\begin{enumerate}
\item [(R1)] a map $\partial\colon E_0\to E_1$,
\item [(R2)] two $A$-connections, $\nabla^0$ and $\nabla^1$ on $E_0$
  and $E_1$, respectively, such that $\partial \circ \nabla^0 =
  \nabla^1 \circ \partial$, \item [(R3)] an element $R \in \Omega^2(A,
  \Hom(E_1, E_0))$ such that $R_{\nabla^0} = R\circ \partial$,
  $R_{\nabla^1}=\partial \circ R$ and $\dr_{\nabla^{\Hom}}R=0$,
  where $\nabla^{\Hom}$ is the connection induced on $\Hom(E_1,E_0)$
  by $\nabla^0$ and $\nabla^1$.
\end{enumerate}
Note that Gracia-Saz and Mehta call this structure a
``superrepresentation'', following \cite{Quillen85}.  In this paper,
$2$-term representation up to homotopy are called
\textbf{$2$-representations} for short.

Consider again a VB-algebroid $(D\to B, A\to M)$ and choose a linear
splitting $\Sigma\colon A\times_MB\to D$. Since the anchor $\Theta_B$
is linear, it sends a core section $c^\dagger$, $c\in\Gamma(C)$ to a
vertical vector field on $B$.  This defines the \textbf{core-anchor}
$\partial_B\colon C\to B$ which is given by
$\Theta_B(c^\dagger)=(\partial_Bc)^\uparrow$ for all $c\in\Gamma(C)$ and
does not depend on the splitting (see \cite{Mackenzie92}). Since the
anchor $\Theta_B$ of a linear section is linear, for each $a\in
\Gamma(A)$ the vector field $\Theta_B(\sigma_A(a))\in\mx^l(B)$ defines a
derivation of $\Gamma(B)$ with symbol $\rho(a)$. This defines a linear
connection $\nabla^{AB}\colon \Gamma(A)\times\Gamma(B)\to\Gamma(B)$:
\[\Theta(\sigma_A(a))=\widehat{\nabla_a^{AB}}\]
for all $a\in\Gamma(A)$.  
Since the bracket of a linear section with a core section is again a
core section, there is a linear connection
$\nabla^{AC}\colon\Gamma(A)\times\Gamma(C)\to\Gamma(C)$ such
that \[[\sigma_A(a),c^\dagger]=(\nabla_a^{AC}c)^\dagger\] for all
$c\in\Gamma(C)$ and $a\in\Gamma(A)$.  The difference
$\sigma_A[a_1,a_2]-[\sigma_A(a_1), \sigma_A(a_2)]$ is a core-linear
section for all $a_1,a_2\in\Gamma(A)$.  This defines a vector valued
form $R\in\Omega^2(A,\operatorname{Hom}(B,C))$ such that
\[[\sigma_A(a_1), \sigma_A(a_2)]=\sigma_A[a_1,a_2]-\widetilde{R(a_1,a_2)},
\]
for all $a_1,a_2\in\Gamma(A)$. For more details on these
constructions, see \cite{GrMe10a}, where the following result is
proved.
\begin{thm}\label{rajan}
  Let $(D \to B; A \to M)$ be a VB-algebroid and choose a linear
  splitting $\Sigma\colon A\times_MB\to D$.  The triple
  $(\nabla^{AB},\nabla^{AC},R)$ defined as above is a
  $2$-representation of $A$ on $\partial_B\colon C\to B$.

  Conversely, let $(D;A,B;M)$ be a double vector bundle such that $A$
  has a Lie algebroid structure and choose a linear splitting
  $\Sigma\colon A\times_MB\to D$. Then if
  $(\nabla^{AB},\nabla^{AC},R)$ is a $2$-representation of $A$ on
  $\partial_B\colon C\to B$, then the equations above define a
  VB-algebroid structure on $(D\to B; A\to M)$.

\end{thm}

\subsection{Double Lie algebroids}
A \textbf{double Lie algebroid} with core $C$ is a double vector
bundle $(D;A,B;M)$
\begin{equation*}
  \begin{xy}
    \xymatrix{
      D \ar[r]^{\pi_A}\ar[d]_{\pi_B}&   A\ar[d]^{q_A}\\
       B\ar[r]_{q_B}                   &  M\\
    }
  \end{xy}
\end{equation*}
with core $C$ and four Lie algebroid structures on $A\to M$, $B\to M$,
$D\to A$ and $D\to B$ such that $(D\to A, B\to M)$ and
$(D\to B, A\to M)$ are VB-algebroids, and the \emph{dual} Lie
algebroids $D\duer A\to C^*$ and $D\duer B\to C^*$ form a VB-Lie
bialgebroid \cite{Mackenzie11}. 
As it is central in this paper, this section recalls the equivalent
definition given by the correspondence of decompositions of double Lie
algebroids with matched pairs of $2$-representations
\cite{GrJoMaMe18}.

Consider a double vector bundle $(D;A,B;M)$ with core $C$ and a VB-Lie
algebroid structure on each of its sides.
Since $(D\to B, A\to M)$ is a VB-Lie algebroid, it is described in any 
linear splitting $\Sigma\colon A\times_MB\to D$
by a representation up to homotopy
$(\partial_B,\nabla^{AB},\nabla^{AC},R_A)$ of $A$ on $C{[0]}\oplus B{[1]}$.
Similarly, via the linear splitting $\Sigma$, the VB-algebroid
$(D\to A, B\to M)$ is described by a representation up to homotopy
$(\partial_A,\nabla^{BA},\nabla^{BC},R_B)$ on $C{[0]}\oplus A{[1]}$.

Let $\rho_A$, respectively $\rho_B$, be the anchor of the Lie
algebroid $A\to M$, respectively $B\to M$.  \cite{GrJoMaMe18} proves
that $(D\duer A, D\duer B)$ is a Lie bialgebroid over $C^*$ if and
only if, for any linear splitting of $D$, the two induced
2-representations as above form a matched pair as in the following
definition \cite{GrJoMaMe18}. In other words, $(D,A,B,M)$ is a double
Lie algebroid if and only if for any linear splitting of $D$, the two
induced 2-representations satisfy the equations (M1) to (M7) below.

\begin{df}\label{matched_pair_2_rep}
  Let $(A\to M, \rho_A, [\cdot\,,\cdot])$ and $(B\to M, \rho_B,
  [\cdot\,,\cdot])$ be two Lie algebroids and assume that $A$ acts on
  $C[0]\oplus B[1]$ up to homotopy via
  $(\partial_B,\nabla^{AB},\nabla^{AC}, R_{A})$ and $B$ acts on
  $C[0]\oplus A[1]$ up to homotopy via
  $(\partial_A,\nabla^{BA},\nabla^{BC}, R_{B})$\footnote{For the sake of
    simplicity, all the four
    connections are denoted by $\nabla$. It is always  clear from the indexes which
    connection is meant. The connection $\nabla^A$ is the $A$-connection
    induced by $\nabla^{AB}$ and $\nabla^{AC}$ on $\wedge^2 B^*\otimes
    C$ and $\nabla^B$ is
    the $B$-connection induced on $\wedge^2
    A^*\otimes C$.  }.  Then the two representations up to
  homotopy form a \textbf{matched pair} if
\begin{enumerate}
\item[(M1)]
  $\nabla_{\partial_Ac_1}c_2-\nabla_{\partial_Bc_2}c_1=-(\nabla_{\partial_Ac_2}c_1-\nabla_{\partial_Bc_1}c_2)$,
\item[(M2)] $[a,\partial_Ac]=\partial_A(\nabla_ac)-\nabla_{\partial_Bc}a$,
\item[(M3)] $[b,\partial_Bc]=\partial_B(\nabla_bc)-\nabla_{\partial_Ac}b$,
\item[(M4)]
$\nabla_b\nabla_ac-\nabla_a\nabla_bc-\nabla_{\nabla_ba}c+\nabla_{\nabla_ab}c=
R_{B}(b,\partial_Bc)a-R_{A}(a,\partial_Ac)b$,
\item[(M5)]
  $\partial_A(R_{A}(a_1,a_2)b)=-\nabla_b[a_1,a_2]+[\nabla_ba_1,a_2]+[a_1,\nabla_ba_2]+\nabla_{\nabla_{a_2}b}a_1-\nabla_{\nabla_{a_1}b}a_2$,
\item[(M6)]
  $\partial_B(R_{B}(b_1,b_2)a)=-\nabla_a[b_1,b_2]+[\nabla_ab_1,b_2]+[b_1,\nabla_ab_2]+\nabla_{\nabla_{b_2}a}b_1-\nabla_{\nabla_{b_1}a}b_2$,
\end{enumerate}
for all $a,a_1,a_2\in\Gamma(A)$, $b,b_1,b_2\in\Gamma(B)$ and
$c,c_1,c_2\in\Gamma(C)$, and
\begin{enumerate}\setcounter{enumi}{6}
\item[(M7)] $\dr_{\nabla^A}R_{B}=\dr_{\nabla^B}R_{A}\in \Omega^2(A,
  \wedge^2B^*\otimes C)=\Omega^2(B,\wedge^2 A^*\otimes C)$, where
  $R_{B}$ is seen as an element of $\Omega^1(A, \wedge^2B^*\otimes
  C)$ and $R_{A}$ as an element of $\Omega^1(B, \wedge^2A^*\otimes
  C)$.
\end{enumerate}
\end{df}

  Recall that the Lie bialgebroid $(D\duer A, D\duer B)$ induces a
  Poisson structure (natural up to sign) on its base $C^*$; this
  Poisson structure is linear \cite[\S4]{Mackenzie11} and so induces a
  Lie algebroid structure on $C$. 
  Writing $e\co D\duer B\to T(C^*)$ and $e_*\co D\duer A\to T(C^*)$ for 
the two anchors, then the Poisson structure on $C^*$ is defined by its sharp morphism
$\pi^\#_{C^*} := e_*\circ e^*=-e\circ e_ *^*\colon T^*C^*\to TC^*$.
The sign of this Poisson structure is determined by
  the requirement that $\partial_A$ and $\partial_B$ be morphisms of
  Lie algebroids (see the next proposition).
The proof of \cite[Proposition 4.5.]{GrJoMaMe18} shows the following identities
  for $c,c_1,c_2\in\Gamma(C)$ and $f,g\in C^\infty(M)$:
  \begin{equation*}
    \begin{split}
     & (e_*\circ e^*)(\dr\ell_{c_1})(\ell_{c_2})=\ell_{\nabla_{\da_A c_1}c_2-\nabla_{\da_B c_2}c_1}=\ell_{-\nabla_{\da_A c_2}c_1+\nabla_{\da_B c_1}c_2}\\
      &(e_*\circ e^*)(\dr\ell_{c})(q_{C^*}^*f))=q_{C^*}^*(\ldr{\rho_B\da_B(c)}f)=q_{C^*}^*(\ldr{\rho_A\da_A(c)}f)\\
      &(e_*\circ e^*)(\dr q_{C^*}^*f)(q_{C^*}^*g)=0.
      \end{split}
    \end{equation*}
  Hence the 
  Lie algebroid on $C\to M$ is defined explicitly using Definition \ref{matched_pair_2_rep}:
  \begin{enumerate}
  \item The anchor is $\rho_C:=\rho_A\circ \partial_A=\rho_B\circ\partial_B$.
   This equality follows easily from (M2) and (M3), see \cite[Remark 3.2]{GrJoMaMe18}.
\item The first equation (M1) gives a simple formula for the induced Lie
  algebroid bracket on section of $C$:
\begin{equation}\label{bracket_on_C}
[c_1,c_2]=\nabla_{\partial_Ac_1}c_2-\nabla_{\partial_Bc_2}c_1.
\end{equation}
The Jacobi identity follows from (M2), (M3) and (M4), and the bracket does not depend on the choice of splitting $\Sigma\colon A\times_MB\to D$,
see \cite[Remark 3.5]{GrJoMaMe18}.
\end{enumerate}

\begin{prop}\label{das_morphisms}
  Let $(D;A,B;M)$ be a double Lie algebroid.  Then the morphisms
  $\partial_A\colon C\to A$ and $\partial_B\colon C\to B$ of vector
  bundles are morphisms of Lie algebroids.
\end{prop}
\begin{proof}
  Since $\rho_C=\rho_A\circ \partial_A=\rho_B\circ\partial_B$, the
  compatibility of the anchors with the morphisms is immediate.

  Equation (M3) with
  $\partial_B\circ \nabla^{AC}=\nabla^{AB}\circ\partial_B$ gives
  \begin{equation*}
    \begin{split}
    \partial_B[c_1,c_2]&=\partial_B(\nabla^{AC}_{\partial_Ac_1}c_2)-\partial_B(\nabla^{BC}_{\partial_Bc_2}c_1)\\
    &=\nabla^{AB}_{\partial_Ac_1}(\partial_Bc_2)-[\partial_Bc_2,\partial_Bc_1]-\nabla^{AB}_{\partial_Ac_1}(\partial_Bc_2)=[\partial_Bc_1,\partial_Bc_2]
  \end{split}
\end{equation*}
for all $c_1,c_2\in\Gamma(C)$. Similarly, (M2) and
$\partial_A\circ \nabla^{BC}=\nabla^{BA}\circ\partial_A$ yield
together the compatibility of $\partial_A$ with the brackets on $C$
and $A$.
  \end{proof}

Finally, the following proposition is an immediate consequence of
\eqref{bracket_on_C}.
\begin{prop}
\label{prop:kercomm}
Let $(D;A,B;M)$ be a double Lie algebroid.  Let $c_1, c_2$ be sections
of the core $C$ with $\da_A(c_1) = 0$ and $\da_B(c_2) = 0$.  Then
$[c_1, c_2] = 0$. 
\end{prop}

\subsection{Morphisms of double Lie algebroids}\label{sec:mor_dla}
This section quickly discusses \textbf{morphisms of matched pairs of
  $2$-representations} versus \textbf{morphisms of double Lie
  algebroids} \cite{DrJoOr15}.

Let $(D_1\to B_1, A_1\to M)$ and $(D_2\to B_2, A_2\to M)$ be two
VB-algebroids over a common double base $M$. A morphism
$\Phi\colon D_1\to D_2$ of double vector bundles with side morphisms
$\phi_A\colon A_1\to A_2$ and $\phi_B\colon B_1\to B_2$ is a
\textbf{morphism of VB-algebroids} if
\begin{equation*}
  \begin{xy}
    \xymatrix{
      D_1 \ar[r]^{\Phi}\ar[d]_{\pi_{B_1}}&   D_2\ar[d]^{\pi_{B_2}}\\
       B_1\ar[r]_{\phi_B}                   &  B_2\\
    }
  \end{xy}
\end{equation*}
is a morphism of Lie algebroids. The vector bundle morphism
$\phi_A\colon A_1\to A_2$ is consequently also a morphism of Lie algebroids,
over the identity on $M$.

Let now $(D_1, A_1, B_1, M)$ and $(D_2, A_2, B_2, M)$ be two double
Lie algebroids over a common double base $M$. A morphism
$\Phi\colon D_1\to D_2$ of double vector bundles with side morphisms
$\phi_A\colon A_1\to A_2$ and $\phi_B\colon B_1\to B_2$ is a
\textbf{morphism of double Lie algebroids} if it is a morphism of
VB-algebroids $(D_1\to B_1, A_1\to M) \to (D_2\to B_2, A_2\to M)$ and
a morphism of VB-algebroids
$(D_1\to A_1, B_1\to M) \to (D_2\to A_2, B_2\to M)$.

\medskip Let $A\to M$ be a Lie algebroid and let
$\mathcal D_E:=(\partial_E,\nabla^{E_0},\nabla^{E_1}, R_E)$ and
$\mathcal D_F:=(\partial_F,\nabla^{F_0},\nabla^{F_1}, R_F)$ be $2$-representations
of $A$ on $E_0[0]\oplus E_1[1]$ and $F_0[0]\oplus F_1[1]$,
respectively.  A \textbf{morphism of representations up to homotopy}
$(A, \mathcal D_E)\to (A, \mathcal D_F)$ is determined by a triple
$(\phi_0, \phi_1, \phi)$, where
$\phi_0\colon E_0 \longrightarrow F_0$, $\phi_1\colon E_1 \to F_1$
are vector bundle morphisms and $\phi \in \Omega^1(A, \Hom(F_1, E_0))$, satisfying
\begin{equation}\label{comp1}
\phi_1\circ \partial_E =  \partial_F \circ\phi_0,\\
\end{equation}
\begin{equation}\label{comp2}
\nabla^{\rm Hom}_a (\phi_0, \phi_1) = (\phi_{a} \circ \partial_E, \,\partial_F
\circ \phi_{a}) \quad \text{ for all } a\in\Gamma(A)\\
\end{equation}
and
\begin{equation}\label{comp3}
\dr_{\nabla^{\rm Hom}}\phi=  R_F\circ \phi_1 -\phi_0 \circ R_E,
\end{equation}
where $\nabla^{\rm Hom}$ is the $A$-connection induced on $\Hom(E_0[0]\oplus E_1[1],F_0[0]\oplus F_1[1])$
by the four connections
(see \cite{ArCr12} for more details).

\medskip

Let $A\to M$ and $A'\to M$ be two Lie algebroids over a smooth
manifold $M$ and consider a Lie algebroid morphism
$\varphi\colon A\to A'$ over the identity on $M$.  Assume that $A'$
acts on $E_0[0]\oplus E_1[1]$ up to homotopy via
$\mathcal D:=(\partial,\nabla^{0},\nabla^{1}, R)$.  For $i=0,1$ the
\textbf{pullback connection}\
$\phi^*\nabla^{i}\colon \Gamma(A) \times \Gamma(E_i) \to \Gamma(E_i)$
is the $A$-connection on $E_i$ given by
\begin{equation}\label{pull_con}
(\phi^*\nabla^{i})_a e := \nabla^i_{\phi(a)} e
\end{equation}
for $a \in \Gamma(A)$ and $e \in \Gamma(E_i)$. The pullback
$\phi^*R \in \Omega^2(A, \operatorname{Hom}(E_1, E_0))$ is similarly defined by
\begin{equation}
(\phi^*R) (a_1, a_2) = R(\phi(a_1), \phi(a_2))
\end{equation}
for $a_1,a_2\in\Gamma(A)$.  Then the triple
$\varphi^*\mathcal D:= (\partial, \phi^*\nabla^0, \phi^*\nabla^1,
\phi^*R)$ is a representation up to homotopy of $A$ on
$E_0[0]\oplus E_1[1]$, see \cite{ArCr12}.

If $\mathcal D_E$ is a $2$-representation of $A$ on
$E_0[0]\oplus E_1[1]$ and $\mathcal D_F$ is a $2$-representation of
$A'$ on $F_0[0]\oplus F_1[1]$, then a morphism of representations up
to homotopy from $\mathcal D_E$ to $\mathcal D_F$ is by definition
\cite{DrJoOr15} a morphism
\[ \mathcal D_E\to \phi^*\mathcal D_F
\]
of $2$-representations of $A$.

\medskip

Let $(D\to B, A\to M)$ and $(D'\to B', A'\to M)$ be two
VB-algebroids over a common double base $M$ and consider a morphism
$\Phi\colon D\to D'$ of double vector bundles with side morphisms
$\phi_A\colon A\to A'$ and $\phi_B\colon B\to B'$.  Choose
linear splittings
\[ \Sigma\colon A\times_MB\to D, \qquad \Sigma'\colon A'\times_MB'\to
  D',
\] of $D$ and $D'$, respectively.  The morphism $\Phi$ and the linear
splittings define a form \linebreak$\phi\in\Omega^1(A, \Hom(B,C'))\simeq \Gamma(A^*\otimes B^*\otimes C')$ by
\begin{equation}\label{eq_def_phi_interchange}
  \Phi\left(\Sigma(a,b)\right)=\Sigma'(\varphi_A(a),\varphi_B(b))+_{B'}\left(0^{D'}_{\varphi_A(a)}+_B\overline{\phi(a,b)}\right)
\end{equation}
for all $a\in\Gamma(A)$. 

Then by
\cite[Theorem 4.11]{DrJoOr15}, $\Phi$ is a VB-algebroid morphism if
and only if $\phi_A\colon A \to A'$ is a Lie algebroid morphism and
for any pair of linear splittings
\[ \Sigma\colon A\times_MB\to D, \qquad \Sigma'\colon A'\times_MB'\to D',
\]
the triple  $(\varphi_B,\varphi_C, \phi)$ is a morphism
\[ \mathcal D\to \varphi_A^*\mathcal D'
  \]
  of the $2$-representations $\mathcal D$ of $A$ and $\mathcal D'$ of
  $A'$ defined by $\Sigma$ and the VB-algebroid $(D\to B, A\to M)$, and by
  $\Sigma'$ and the  VB-algebroid $(D'\to B', A'\to M)$.

  \medskip
  
  Assume finally that $(D, A, B, M)$ and $(D', A', B', M)$ are two
  double Lie algebroids and consider a morphism $\Phi\colon D\to D'$
  as above of double vector bundles. Choose
linear splittings
\[ \Sigma\colon A\times_MB\to D, \qquad \Sigma\colon A'\times_MB'\to
  D',
\] of $D$ and $D'$, respectively, and consider the form
$\phi\in\Gamma(A^*\otimes B^*\otimes C')$ defined as in
\eqref{eq_def_phi_interchange} by $\Phi$ and the linear
splittings. Let $\mathcal D_A$ and $\mathcal D_A'$
be the $2$-representations of $A$ and  of
  $A'$ defined by $\Sigma$ and the VB-algebroid $(D\to B, A\to M)$, and by
  $\Sigma'$ and the  VB-algebroid $(D'\to B', A'\to M)$, respectively.
  Similarly, Let $\mathcal D_B$ and $\mathcal D_B'$
be the $2$-representations of $B$ and  of
  $B'$ defined by $\Sigma$ and the VB-algebroid $(D\to A, B\to M)$, and by
  $\Sigma'$ and the  VB-algebroid $(D'\to A', B'\to M)$, respectively.
 
Then $\Phi$ is a double Lie algebroid morphism if and only if
$(\phi_B,\phi_C,\phi)$ defines a morphism of $2$-representations
\begin{equation}\label{dla_mor_1}
  \mathcal D_A\to \varphi_A^*\mathcal D'_A,
\end{equation}
and
$(\phi_A,\phi_C,\phi)$ defines a morphism of $2$-representations
\begin{equation}\label{dla_mor_2}
  \mathcal D_B\to \varphi_B^*\mathcal D'_B.
\end{equation}

The following proposition is standard, but can now be proved easily
using \eqref{bracket_on_C} and \eqref{comp2}.

\begin{prop}\label{mor_dla_induces_mor_C}
  Let $(D, A, B, M)$ and $(D', A', B', M)$ be two double Lie
  algebroids and consider a morphism $\Phi\colon D\to D'$ of double
  Lie algebroids. Then the core morphism $\phi_C\colon C\to C'$ is a morphism of Lie algebroids.
\end{prop}

\begin{proof}
  Let $\phi_A\colon A\to A'$ and $\phi_B\colon B\to B'$ be the side
  morphisms of $\Phi$. By the considerations above, they are Lie
  algebroid morphisms. Choose
linear splittings
\[ \Sigma\colon A\times_MB\to D, \qquad \Sigma'\colon A'\times_MB'\to
  D',
\] of $D$ and $D'$, respectively, and consider the form
$\phi\in\Gamma(A^*\otimes B^*\otimes C')$ defined as in
\eqref{eq_def_phi_interchange} by $\Phi$ and the linear
splittings. Consider as above the four $2$-representations\footnote{For the sake of
  simplicity, all the eight connections are denoted by $\nabla$. It is
   clear from the indexes which connection is meant.} defined by
the VB-algebroid and the splittings.
Then for $c_1,c_2\in\Gamma(C)$:
\begin{equation*}
  \begin{split}
    \phi_C[c_1,c_2]&=\phi_C\left(\nabla_{\da_Ac_1}c_2-\nabla_{\da_Bc_2}c_1\right)\\
    &=(\phi_A^*\nabla)_{\da_Ac_1}(\phi_Cc_2)+\omega(\da_Ac_1,\da_Bc_2)-(\phi_B^*\nabla)_{\da_Bc_2}(\phi_Cc_1)-\omega(\da_Ac_1,\da_Bc_2)\\
    &=\nabla_{\phi_A\da_Ac_1}(\phi_Cc_2)-\nabla_{\phi_B\da_Bc_2}(\phi_Cc_1)\\
    &=\nabla_{\da_{A'}\phi_Cc_1}(\phi_Cc_2)-\nabla_{\da_{B'}\phi_Cc_2}(\phi_Cc_1)=[\phi_Cc_1,\phi_Cc_2].
     \end{split}
  \end{equation*}
  The anchor condition is checked as follows using \eqref{comp1}
  \[ \rho_{C'}\circ\phi_C= \rho_{A'}\circ\da_{A'}\circ \phi_C=\rho_{A'}\circ\phi_A\circ\da_A=\rho_A\circ\da_A=\rho_C. \qedhere
    \]
  \end{proof}

  \section{Transitive double Lie algebroids and transitive core diagrams}\label{func_trdLA_codiag}
  This section discusses \emph{transitive double Lie algebroids}, and
  establishes a functor from the category of transitive double Lie
  algebroids, to the category of transitive core diagrams.
  
Recall from \cite[4.4.1]{Mackenzie05} that a morphism of Lie algebroids $\phi\co A'\to A$
over a surjective submersion $f\co M'\to M$ is a \textbf{fibration of
Lie algebroids} if and only if the associated map $a'\mapsto (q_{A'}a',\phi(a'))$
into the vector bundle pullback $f^!A\to M'$ is a surjection.

\begin{prop}
Let $(D;A,B;M)$ be a double Lie algebroid with core $C$. 
Assume that the anchors $\rho_A\colon A\to TM$ and $\rho_B\colon B\to TM$
are both surjective.

{\rm (i)} The morphism $\da_A\co C\to A$ is surjective if and only if
the anchor $\Theta_A\co D\to TA$ is a fibration of Lie algebroids.   

{\rm (ii)} The morphism $\da_B\co C\to B$ is surjective if and only if
the anchor $\Theta_B\co D\to TB$ is a fibration of Lie algebroids.   
\end{prop}

\pf In order to prove (i), assume first that $\da_A$ is
surjective. Take $b\in B$ and $\xi\in TA$, projecting to the same
point $v\in TM$: $Tq_A\xi=\rho_B(b)$. Write $a$ for the projection of
$\xi$ to $A$. Take any $d'\in D$ which projects to $a$ and $b$ and
write $\xi' = \Theta_A(d') \in TA$. Then $\xi'$ and $\xi$ project to
the same elements of $A$ and $TM$ and therefore there is a core
element $\da_A(c)\in A$ such that
$$
\xi = \xi' +_A ( 0^{TA}_a +_{TM}\overline{\da_A(c)}).
$$
Now define $d := d' +_A ( 0_a +_B\overline{c})$. Clearly $\Theta_A(d) = \xi$. This
shows that $\Theta_A$ is a fibration. The converse is straightforward since $\da_A$ is the core morphism of $\Theta_A$.
\epf

Although the considerations above are the authors' original intuition
for transitive double Lie algebroids, the definition below drops the
surjectivity of the anchors of $A$ and $B$ because they are not needed
for the main results of this paper. Hence, a transitive double Lie
algebroid $(D;A,B;M)$ as in the following definition has two
surjective anchors $D\to TA$ and $D\to TB$ if and only if the two base
anchors $\rho_A\colon A\to TM$ and $\rho_B\colon B\to TM$ are
surjective.
\begin{df}
  A double Lie algebroid $(D;A,B;M)$ is \textbf{transitive} if its two
  core-anchors $\da_A$ and $\da_B$ are surjective.
\end{df}

\begin{df}[\cite{Androulidakis05}]
\label{df:xm}
A \textbf{crossed module of Lie algebroids} $(A,C,\da, \rho)$ comprises a Lie algebroid 
$A$ on base $M$, 
a totally intransitive Lie algebroid $C$ on the same base, a representation
$\nabla\co A\to \Der(C)$, and a morphism $\da\co C\to A$ of Lie algebroids over $M$, such that
\begin{subequations}
\begin{gather}
\nabla_a[c_1,c_2]=[\nabla_ac_1, c_2]+[c_1, \nabla_ac_2],\label{old_2a}\\
\nabla_{\da(c_1)}c_2=[c_1, c_2],\label{old_2b}\\
\da(\nabla_ac)=[a,\da(c)].\label{old_2c}
\end{gather}
\end{subequations}
for all $c,c_1,c_2\in\Gamma(C)$ and $a\in\Gamma(A)$. 
\end{df}

Here $\Der(C)$ denotes the Lie algebroid of derivations of the vector
bundle $C$: sections of $\Der(C)$ are $\mathbb R$-linear maps
$D\colon\Gamma(C)\to\Gamma(C)$ for which there is a vector field
$X\in\mx(M)$ such that $D(fc) = fD(c) + X(f)c$ for all $f\in\sfn{M}$
and $c\in\Gamma(C)$. For $C$ a totally intransitive Lie algebroid,
$\Derb(C)$ is the Lie subalgebroid of $\Der(C)$ the sections of 
which are also derivations of the bracket,
\[D[c_1, c_2] = [D(c_1), c_2] + [c_1, D(c_2)]\] for
$D\in\Gamma(\Derb(C))$ and $c_1, c_2\in\Ga(C)$. That is, $r$ in
Definition \ref{df:xm} is a morphism of Lie algebroids
$A\to\Derb(C)$. For clarity, sections of $\Derb(C)$ are called
\textbf{Lie derivations of $C$}.

\medskip Consider now a transitive double Lie algebroid $(D;A,B;M)$
with core $C$.  Write $C^A = \ker(\da_B)$ and $C^B =
\ker(\da_A)$. Then $C^A$ is a smooth subbundle of $B$ and $C^B$ is a
subbundle of $A$.

\begin{prop}
\label{prop:xm}
Let $(D;A,B;M)$ be a transitive double Lie algebroid with core $C$.
For $a\in\Gamma(A)$ and $\gamma\in\Gamma(C^A)$ set
$\nabla^A_a\gamma = [c,\gamma]$ where $c\in\Ga(C)$ is such that
$\da_A(c) = a$.  Then $\nabla^A_a\gamma$ is well-defined and, together
with the restriction of $\da_A$ to $C^A$, gives
$(A,C^A,\da_A,\nabla^A)$ the structure of a crossed module of Lie
algebroids.
\end{prop}

\begin{proof}
  First, $\da_B[c,\gamma] = [\da_B c, 0]=0$ since $\da_B$ is a Lie algebroid morphism by Proposition \ref{das_morphisms}.
  Therefore, $[c,\gamma]$ is a section of $C^A$.
If $c,c'$ are two sections of $C$ satisfying $\da_A(c)=\da_A(c')=a$, 
then $c-c'\in\ker(\da_A)=C^B$ and so, by Proposition \ref{prop:kercomm}, 
$[c-c',\gamma]=0$ for all $\gamma\in C^A=\ker(\da_B)$.

The Lie algebroid $C^A$ is totally intransitive as the kernel of the
morphism $\da_B$ of Lie algebroids. Then
$\nabla^A_a(f\gamma)=[c,f\gamma]=f[c,\gamma]+\rho_C(c)(f)\gamma=f
\nabla^A_a(c)+\rho_A(a)(f)\gamma$, since
$\rho_C=\rho_A\circ\partial_A$.  Likewise,
$\nabla^A_{fa}(\gamma) = [fc,\gamma]=f \nabla^A_a(\gamma)$ because
$C^A$ is totally intransitive. Hence,
$\nabla^A\colon\Gamma(A)\times\Gamma(C^A)\to\Gamma(C^A)$ is a linear
connection.  \eqref{old_2b} and \eqref{old_2c} are immediate by
the definition of $\nabla^A$ and the fact that $\da_A\colon C\to A$ is a
Lie algebroid morphism.

For the flatness of $\nabla^A$, take $a_1,a_2\in\Gamma(A)$ and
$\gamma\in\Gamma(C^A)$, as well as $c_1,c_2\in\Gamma(C)$ with
$\da_A(c_1)=a_1$ and $\da_A(c_2)=a_2$.
Then $\da_A[c_1,c_2]=[a_1,a_2]$ since $\da_A$ is a morphism of Lie algebroids, and
\begin{align*}
  \begin{split}
    \nabla^A_{a_1}\nabla^A_{a_2}\gamma-\nabla^A_{a_2}\nabla^A_{a_1}\gamma-\nabla^A_{[a_1,a_2]}\gamma
    =[c_1,[c_2,\gamma]]-[c_2,[c_1,\gamma]]-[[c_1,c_2],\gamma]=0.
  \end{split}
\end{align*}
Similarly,
\[\nabla_a^A[\gamma_1,\gamma_2]-[\nabla^A_a\gamma_1,\gamma_2]-[\gamma_1,\nabla^A_a\gamma_2]=
  [c,[\gamma_1,\gamma_2]]-[[c,\gamma_1],\gamma_2]-[\gamma_1,[c,\gamma_2]]=0
\]
for $a\in\Gamma(A)$, $c\in\Gamma(C)$ with $\da_A(c)=a$ and
$\gamma_1,\gamma_2\in \Gamma(C^A)$, shows the equality \eqref{old_2a}.
\end{proof}

Let $(D;A,B;M)$ be a transitive double Lie algebroid with core $C$.
The core $C$ and the two crossed modules $(A,C^A,\da_A,\nabla^A)$ and
$(B,C^B,\da_B,\nabla^B)$ constructed in Proposition \ref{prop:xm} are
shown in Figure~\ref{fig:cd}, picturing the \emph{transitive core
  diagram of the transitive double Lie algebroid $(D;A,B;M)$}. The
remainder of this paper shows that a transitive double Lie algebroid
can be reconstructed from its (transitive) core diagram.
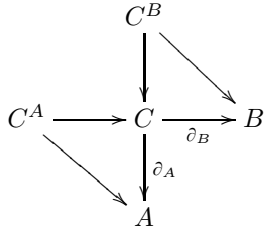
\begin{figure}[h]
$$ 
\xymatrix{
&C^B\ar[d]\ar[dr]&\\
C^A\ar[r]\ar[dr]&C\ar[r]_{\da_B}\ar[d]^{\da_A}&B\\
&A&
}
$$ 
\caption{The transitive core diagram of a transitive double Lie
  algebroid.\label{fig:cd}}
\end{figure}

\begin{df}\label{df:cd}
  Let $A$ and $B$ be Lie algebroids over a same base $M$. A
  \textbf{core diagram for $A$ and $B$} is a Lie algebroid $C$ on $M$
  together with  morphisms $\da_A\co C\to A$ and
  $\da_B\co C\to B$ of Lie algebroids such that the kernels
  $C^A:= \ker(\da_B)$ and $C^B:=\ker(\da_A)$ commute in $C$.  A
  \textbf{core diagram for $A$ and $B$} is \textbf{transitive} if the
  morphisms $\da_A$ and $\da_B$ are surjective.
\end{df}

In the transitive case the proof of Proposition \ref{prop:xm} applies,
and there are representations
\begin{equation}
\label{eq:cdxm}
\nabla^A\co A\to\Der(C^A),\quad \nabla^B\co B\to\Der(C^B),   
\end{equation}
defined as in \ref{prop:xm}, which together with the restrictions of
$\da_A$ and $\da_B$, respectively, form crossed modules as in
Figure~\ref{fig:cd}.

This paper shows that every transitive core diagram is the core
diagram of a transitive double Lie algebroid.  The construction
follows the general outline of the analogous results for double Lie
groupoids and \LAgpds \cite{BrMa92,Mackenzie92}: first, Section
\ref{sec:comma} defines a `large' double Lie algebroid in which the
two structures are an action Lie algebroid and a pullback, and then a
quotient is taken in Section \ref{sec:quotient}.

\bigskip Let now $(D, A, B, M)$ and $(D', A', B', M)$ be two
transitive double Lie algebroids over a common double base
$M$. Consider a morphism $\Phi\colon D\to D'$ of double Lie
algebroids, with side morphisms $\phi_A\colon A\to A'$ and
$\phi_B\colon B\to B'$, and with core morphism $\phi_C\colon C\to
C'$. Then these two side morphisms and the core morphism are morphisms
of Lie algebroids over $M$ and, using \eqref{comp1} and Proposition
\ref{mor_dla_induces_mor_C}, the triple $(\phi_A, \phi_B, \phi_C)$ is
a morphism of the core diagrams of $D$ and $D'$, as in the following
definition.

\begin{df}\label{df:mcd}
  Let $A,A',B,B'$ be Lie algebroids over a base $M$ and consider core
  diagrams $(C,\da_A,\da_B)$ and $(C', \da_{A'}, \da_{B'})$ for $A$
  and $B$ and for $A'$ and $B'$, respectively.  A morphism
  \[\Phi\colon (C,\da_A,\da_B)\to (C', \da_{A'}, \da_{B'})\] of core
  diagrams is a triple of Lie algebroid morphisms
  $\phi_A\colon A\to A'$, $\phi_B\colon B\to B'$ and
  $\phi_C\colon C\to C'$ such that $\da_{A'}\circ\phi_C=\phi_A\circ\da_A$
  and $\da_{B'}\circ\phi_C=\phi_B\circ \da_B$.
\end{df}

\begin{figure}[h]
$$ 
\xymatrix{
  & C^B\ar[rrr]^{\phi_C\an{C^B}}\ar[dd] &&& C^{B'}\ar[dd]& \\
  C^A\ar[rrr]^{\phi_C\an{C^A}}\ar[rd]& &&C^{A'}\ar[rd]&&\\
  &C\ar[dd]_{\da_A}\ar[rd]^{\da_B}\ar[rrr]^{\phi_C}&&&C'\ar[dd]\ar[rd]^{\da_{B'}}&\\
  &&B\ar[rrr]^{\phi_B}&&&B'\\
  &A\ar[rrr]^{\phi_A}&&&A'&
}
$$ 
\caption{A morphism of core diagrams.\label{fig:mcd}}
\end{figure}
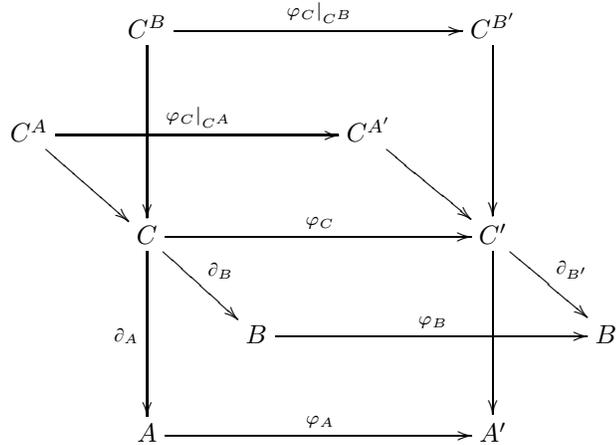

Given a morphism
$\Phi\colon (C,\da_A,\da_B)\to (C', \da_{A'}, \da_{B'})$ of transitive
core diagrams as in Figure \ref{fig:mcd}, the morphism
$\phi_C\colon C\to C'$ restricts to morphisms $C^A\to C^{A'}$ and
$C^B\to C^{B'}$ of Lie algebroids. Therefore, all squares in
\ref{fig:mcd} are commutative squares of Lie algebroid morphisms over
$M$.  A computation yields as well for $a\in\Gamma(A)$,
$c\in\Gamma(C)$ such that $\da_Ac=a$ and $\gamma\in\Gamma(C^A)$:
\begin{equation*}
  \begin{split}
    \phi_C(\nabla^A_a\gamma)=\phi_C[c,\gamma]=[\phi_C(c),\phi_C(\gamma)]=\nabla^{A'}_{\da_{A'}\phi_C(c)}(\phi_C(\gamma))=\nabla^{A'}_{\phi_A(a)}(\phi_C(\gamma))
    =(\phi_A^!\nabla^{A'})_a(\phi_C(\gamma)).
  \end{split}
  \end{equation*}

\section{The comma double Lie algebroid}\label{sec:comma}

This section constructs the \textbf{comma double Lie algebroid} defined
by a morphism of Lie algebroids. It constructs more precisely a
functor from morphisms of Lie algebroids over a fixed base, to double
Lie algebroids over the fixed double base.
\subsection{Action of $TC\to TM$ on $A\to TM$}\label{sec:action}

Let $C$ and $A$ be Lie algebroids over a manifold $M$ and let
$\da\co C\to A$ be a morphism of Lie algebroids over $M$. This section
and the next construct a double Lie algebroid $\ldbl$ from this
data, 
the \emph{comma double Lie algebroid defined by $\da\co C\to A$}. The
construction follows the one of the \emph{comma double groupoids} in
\cite[2.5]{Mackenzie92} and \cite{BrMa92}.

First an action of the Lie algebroid $TC\to TM$ is defined on the
anchor $\rho_A\co A\to TM$.  Recall that the Lie algebroid structure
on $TC\to TM$ may be defined as follows, using tangent sections $Tc$ and core
sections $c^\dagger$, for $c\in\Ga(C)$. For sections
$c_1, c_2\in\Ga(C)$, the bracket on $\tmtc$ is given by
\begin{equation}
[Tc_1, Tc_2] = T[c_1, c_2],\quad
[Tc_1, c_2^\dagger] = [c_1, c_2]^\dagger,\quad
[c_1^\dagger, c_2^\dagger] = 0. 
\end{equation}The anchor of $TC\to TM$ is
$\Theta:= \sigma_M\circ T\rho_C$ where $\sigma_M\co T^2M\to T^2M$ is the
canonical involution (denoted $J_M$ in \cite{Mackenzie05}).

An action of $TC\to TM$ on $\rho_A\co A\to TM$ is a map $\Phi\co\Gamma_{TM}(TC)\to \mx(A)$ 
such that 
\begin{subequations}
\begin{gather}
\Phi[\xi_1, \xi_2] = [\Phi(\xi_1), \Phi(\xi_2)],\label{eq:actproj_a}\\
\Phi(\xi_1+\xi_2) = \Phi(\xi_1) + \Phi(\xi_2),\label{eq:actproj_b}\\
\Phi(F\xi) = (\rho_A^*F)\Phi(\xi),\label{eq:actproj_c}\\
\ldr{\Phi(\xi)}(\rho_A^*F) = \rho_A^*(\ldr{\Theta(\xi)}F), \label{eq:actproj}
\end{gather}
\end{subequations}
for $\xi,\xi_1,\xi_2\in\tmtc$ and $F\in\sfn{TM}$ \cite[4.1.1]{Mackenzie05}.

\medskip

\begin{prop}
  Let $C$ and $A$ be Lie algebroids over a manifold $M$ and let
  $\da\co C\to A$ be a morphism of Lie algebroids over $M$.
  For $c\in\Ga(C)$, set
 \begin{equation}
 \label{eq:Phi}
 \Phi(Tc) := \Hat{[\partial c,\cdot]}\in \mx^l(A),\qquad  \text{ and } \qquad   
 \Phi(c^\dagger):=(\partial c)^\uparrow\in\mx^c(A). 
 \end{equation}
 Then \eqref{eq:Phi} extends to a unique action $\Phi$  of
 $TC\to TM$ on $\rho_A\colon A\to TM$.
\end{prop}

Note that $\Hat{[\partial c,\cdot]}\in\mx^l(A)$ is linear over $\rho_A(\da c)=\rho_C(c)\in\mx(M)$.

 \begin{proof}
 An arbitrary section $\chi$ of $TC\to TM$
 can be written 
 \[
 \chi=\sum_{i=1}^kF_i\, Tc_i+\sum_{j=1}^nG_j{\gamma_j}^\dagger
 \]
 with $F_1,\ldots,F_k,G_1,\ldots,G_n\in C^\infty(TM)$ and 
 $c_1,\ldots,c_k, \gamma_1,\ldots,\gamma_n\in\Gamma(C)$. 
 Then set 
 \begin{equation}\label{def_phi}
   \begin{split}
     \Phi(\chi)&=\sum_{i=1}^k\rho_A^*F_i\,\Phi(Tc_i)+\sum_{j=1}^n \rho_A^*G_j\,\Phi({\gamma_j}^\dagger)
     =\sum_{i=1}^k\rho_A^*F_i\,\Hat{[\partial c_i,\cdot]}+\sum_{j=1}^n \rho_A^*G_j\,(\partial \gamma_j)^\uparrow.
     \end{split}
\end{equation}

\medskip
\noindent\underline{\textbf{\eqref{def_phi} defines a map $\Phi\colon \Gamma_{TM}(TC)\to\mx(A)$:}}

For $c\in\Gamma(C)$ and $f\in C^\infty(M)$ \begin{equation}
 \label{eq:TfZ}
 T(fc)=\ell_{\dr f}\cdot c^\dagger+_{TM} p_{M}^*f\cdot Tc,\qquad
 (fc)^\dagger=p_{M}^*f\cdot c^\dagger,
 \end{equation}
 where $+_{TM}$ and $\cdot$ are the addition and scalar
 multiplication in $TC\to TM$, and $p_M\colon TM\to M$ is the
 canonical projection.

 Hence the following identities need to hold:
 \begin{equation}\label{phi_welldef1}
 \Phi(T(fc))=\Phi(\ell_{\dr f}\cdot c^\dagger +_{TM} p_{M}^*f\cdot Tc)
 \end{equation}
 and 
 \begin{equation}\label{phi_welldef2}
 \Phi((fc)^\dagger)=\Phi(p_{M}^*f\cdot c^\dagger)
 \end{equation}
 for all $c\in\Gamma(C)$ and $f\in C^\infty(M)$. 
 For the first identity, compute
 \begin{align*}
   \Phi(\ell_{\dr f}\cdot c^\dagger +_{TM} p_{M}^*f\cdot Tc)
   &=\rho_A^*\ell_{\dr f}\cdot (\partial c)^\uparrow+\rho_A^*p_{M}^*f\cdot \Hat{[\partial c,\cdot]}=\ell_{\dr_Af}\cdot (\partial c)^\uparrow+q_A^*f\cdot \Hat{[\partial c,\cdot]}\\
   &=\Hat{[\partial (fc),\cdot]}=\Phi(T(fc)).
 \end{align*}
 Here $\dr_A$ is
 the de Rham like differential
 $\Omega^\bullet(A)\to\Omega^{\bullet +1}(A)$ defined by $A$, and
 $\rho_A^t\colon T^*M\to A^*$ is the pointwise transpose of $\rho_A$.
 The notation $\rho_A^t$ avoids confusions with the pullback
 $\rho_A^*\colon \Gamma\left(\bigwedge^\bullet
   T^*M\right)\to\Gamma\left(\bigwedge^\bullet T^*A\right)$.  It is
 easy to check that
 $\rho_A^*\ell_{\dr f}=\ell_{\rho_A^t\dr f}=\ell_{\dr_A f}$ for all
 $f\in C^\infty(M)$.  The second equality is then immediate with
 $q_A=p_{M}\circ \rho_A$.  The third equality can be checked easily on
 linear functions $\ell_\alpha$ for $\alpha\in\Gamma(A^*)$ and $q_A^*f$,
 for $f\in C^\infty(M)$.

 For \eqref{phi_welldef2}, compute
 \begin{equation*}
   \Phi(p_{M}^*f\cdot c^\dagger)=\rho_A^*p_M^*f \cdot (\da c)^\uparrow=q_A^* f\cdot (\da c)^\uparrow=(f\da c)^\uparrow=(\da(fc))^\uparrow=\Phi\left((fc)^\dagger\right).
   \end{equation*}

 \medskip

 \noindent\underline{\textbf{ $\Phi\co\Gamma_{TM}(TC)\to\mx(A)$ is an action:}}
 The identities \eqref{eq:actproj_b} and \eqref{eq:actproj_c} hold by
 construction.  By the construction in \eqref{def_phi} as well, it is
 sufficient to check \eqref{eq:actproj} on tangent and core sections
 $\xi$ of $TC\to TM$.  Consider first $\xi = Tc$ where
 $c\in \Gamma(C)$. Then 
\eqref{eq:actproj} reads 
 \[ 
 \ldr{\Hat{[\partial c,\cdot]}}(\rho_A^*F)=\rho_A^*(\ldr{(\sigma_{M}\circ T(\rho_C(c)))}F)
 \]
 for all $F\in C^\infty(TM)$.  Here again, it is enough to check this
 equality on pullbacks $p_{M}^*f$ for $f\in C^\infty(M)$ and linear
 functions $\ell_\theta$ for $\theta\in\Omega^1(M)$.

 For $a_m\in A$ and any $a\in\Gamma(A)$ such that $a(m)=a_m$, 
 \begin{align*}
 \ldr{\Hat{[\partial c,\cdot]}}(\rho_A^*p_{M}^*f)
 &=\ldr{\Hat{[\partial c,\cdot]}}(q_A^*f)=q_A^*(\ldr{(\rho_A\circ\partial )(c)}f)
 =q_A^*(\ldr{\rho_C(c)}f),\\
 \ldr{\Hat{[\partial c,\cdot]}}(\rho_A^*\ell_\theta)(a_m)
 &=\Hat{[\partial c,\cdot]}(\ell_{\rho_A^t\theta})(a_m)=\ell_{\ldr{\partial c}\rho_A^t\theta}(a_m)
 =\rho_C(c_m)\langle \rho_A^t\theta, a\rangle -\langle \rho_A^t\theta_m, [\partial c,a](m)\rangle\\
 &=\rho_C(c_m)\langle \theta, \rho_A(a)\rangle -\langle \theta_m, [\rho_C(c),\rho_A(a)](m)\rangle= \ell_{\alpha}(a_m),
 \end{align*}
 where $\alpha := \rho_A^t\ldr{\rho_C(c)}(\theta)\in\Gamma(A^*)$.

 Now consider the right hand side of (\ref{eq:actproj}).  Choose
 $a_m\in A$, let $\gamma$ be a path through $m$ with
 $\dot \gamma(0)=\rho_A(a_m)$ and let $\phi^{\rho_C(c)}_\cdot$ be the
 flow of $\rho_C(c)\in\mx(M)$.  Then
 \begin{align*}
 \sigma_{M}(T(\rho_C(c))\rho_A(a_m))
 &=\sigma_{M}\left(\left.\frac{d}{dt}\right\an{t=0}\left.\frac{d}{ds}\right\an{s=0} 
 \phi^{\rho_C(c)}_s(\gamma(t))\right)\\
 &=\left.\frac{d}{ds}\right\an{s=0}\left.\frac{d}{dt}\right\an{t=0} \phi^{\rho_C(c)}_s(\gamma(t))
 =\left.\frac{d}{ds}\right\an{s=0}T_m\phi^{\rho_C(c)}_s(\rho_A(a_m)),
 \end{align*}
 and so 
 \begin{align*}
 \sigma_{M}(T(\rho_C(c))\rho_A(a_m))(p_{M}^*f)&=\left.\frac{d}{ds}\right\an{s=0}(f\circ\phi^{\rho_C(c)}_s)(m)
 =(\ldr{\rho_C(c)}f)(m)=q_A^*(\ldr{\rho_C(c)}f)(a_m),\\
 \sigma_{M}(T(\rho_C(c))\rho_A(a_m))(\ell_\theta)&=\left.\frac{d}{ds}\right\an{s=0}\langle
 \theta(\phi^{\rho_C(c)}_s(m)), T_m\phi^{\rho_C(c)}_s(\rho_A(a_m))\rangle\\
 &=\langle \ldr{\rho_C(c)}\theta, \rho_A(a_m)\rangle
 =\ell_{\alpha}(a_m),
 \end{align*} 
 where $\alpha = \rho_A^t\ldr{\rho_C(c)}(\theta)$ as before. This completes the proof of 
 \eqref{eq:actproj} for $\xi = Tc$. 

 Now consider $c^\dagger\in\Gamma_{TM}(TC)$.
 The equality $\Theta(c^\dagger)=(\rho_C(c))^\uparrow\in \mx(TM)$
 can be verified using the same method as above.
 It only remains to prove that
 \[
 \ldr{(\partial c)^\uparrow}(\rho_A^*F) = \rho_A^*(\ldr{\rho_C(c)^\uparrow}F)
 \]
which is immediate for $F$ a pullback or a linear function on $TM$.
\medskip

Finally the compatibility of $\Phi$ with the Lie algebroid brackets
\eqref{eq:actproj_a} needs to be checked:
 \[
 [\Phi(V), \Phi(W)]=\Phi[V,W]
 \]
 for all $V,W\in \Gamma_{TM}(TC)$.  Because of \eqref{eq:actproj}, it
 is sufficient to prove \eqref{eq:actproj_a} on tangent and core
 sections of $TC\to TM$.  Computations yield for $c,c'\in\Gamma(C)$:
 \begin{align*}
 \left[\Phi(Tc),\Phi(Tc')\right]
 &=\left[\Hat{[\partial c, \cdot]}, \Hat{[\partial c',\cdot]}\right]
 =\Hat{[\da c,\cdot]\circ [\partial c',\cdot]
   -[\partial c',\cdot]\circ [\partial c,\cdot]}\\
 &=\Hat{[\partial [c,c'],\cdot]}=\Phi\left(T[c,c']\right)
   =\Phi([Tc, Tc']),
 \end{align*}
 by the Jacobi identity for the Lie algebroid bracket on $A$, and 
 \begin{align*}
 \left[\Phi(Tc),\Phi\left({c'}^\dagger\right)\right]
 &=\left[\Hat{[\partial c,\cdot]}, (\partial c')^\uparrow\right]  
   =\left([\partial c,\cdot] (\partial c')\right)^\uparrow\\
 &=[\partial c, \partial c']^\uparrow
  =(\partial [c,c'])^\uparrow
  =\Phi\left([c,c']^\dagger\right)
  =\Phi\left(\left[Tc,  {c'}^\dagger\right]\right).
 \end{align*}
 The Lie algebroid brackets of core sections of $TC\to TM$, and of
 vertical vector fields on $A$ vanish, so the proof is complete.
 \end{proof}

 \subsection{The double vector bundle $R$ }\label{lie_algebroid_structures}
 Consider two Lie algebroids $A$ and $C$ over a common base $M$ and
 construct the pullback manifold
 $\ldbl:=TC\oplus_{TM}A:=\{(V,a)\in TC\times A\mid
 Tq_C(V)=\rho_A(a)\}$, as shown in Figure~\ref{fig:Om}(a).  This
 section shows that $R$ has a double vector bundle structure with
 sides $A$ and $C$ and with core $C$.

 \begin{figure}[h]
 $$
 \subfloat[The pullback manifold]{
 \xymatrix@=15mm{
 \ldbl \ar[r] \ar[d]& A\ar[d]^{\rho_A\phantom{XXXXX}}\\
 TC\ar[r]^{Tq_C}&TM
 }}
 \qquad\qquad
 \subfloat[The double Lie algebroid]{
 \xymatrix@=15mm{
 \ldbl \ar[r]^{\text{pullback}} \ar[d]_{\text{action}}& C \ar[d]^{\phantom{XXXX}}\\
 A\ar[r]&M
 }}
 $$
 \caption{\label{fig:Om}}
\end{figure}

The space $R$ is equipped with two projections $\pi_A\colon R\to A$,
$(V,a)\to a$ and $\pi_C\colon R\to C$, $(V,a)\to p_C(V)$, where as
always, $p_C\colon TC\to C$ is the canonical projection.  Define
additions
  \begin{equation}\label{add_A} +_A\colon R\times_A R\to R, \qquad (V_1,a)+_A(V_2,a)=(V_1+_{TM}V_2,a),
  \end{equation}
  and
  \begin{equation}\label{add_C}
    +_C\colon R\times_C R\to R, \qquad (V_1,a_1)+_C(V_2,a_2)=(V_1+_{C}V_2,a_1+a_2).
  \end{equation}

\begin{prop}\label{dvb_R}
  An $A$-connection
  $\nabla\colon\Gamma(A)\times\Gamma(C)\to \Gamma(C)$ on $C$
  defines a bijection
  \[ I_\nabla\colon A\times_MC\times_MC \to R,\qquad 
  (a,c,\gamma)\mapsto \left(\widehat{\nabla_a}(c)+_C\gamma^\uparrow(c),  a\right),
  \]
  such that
  \[ \pi_C(I_\nabla(a,c,\gamma))=c \quad \text{ and } \quad  \pi_A(I_\nabla(a,c,\gamma))=a
  \]
  for $(a,c,\gamma)\in A\times_MC\times_MC$, 
  and 
  \[ I_\nabla(a_1+a_2,c, \gamma_1+\gamma_2)=I_\nabla(a_1,c, \gamma_1)+_CI_\nabla(a_2,c, \gamma_2)
  \]
  as well as
  \[I_\nabla(a,c_1+c_2, \gamma_1+\gamma_2)=I_\nabla(a,c_1, \gamma_1)+_AI_\nabla(a,c_2, \gamma_2)
  \]
  for $(a_1,c, \gamma_1), (a_2,c, \gamma_2)$ and
  $ (a,c_1, \gamma_1), (a,c_2, \gamma_2)\in A\times_MC\times_M C$.
\end{prop}

\begin{proof}
  Take $ (a_1,c_1, \gamma_1)$ and
  $(a_2,c_2, \gamma_2)\in A\times_MC\times_M C$ such that
  $I_\nabla(a_1,c_1, \gamma_1)=I_\nabla(a_2,c_2, \gamma_2)\in R$.
  Then
  \[\left(\widehat{\nabla_{a_1}}(c_1)+_C\gamma_1^\uparrow(c_1),
      a_1\right)=\left(\widehat{\nabla_{a_2}}(c_2)+_C\gamma_2^\uparrow(c_2),
      a_2\right),
  \]
  and so $a_1=a_2$, as well as
  $c_1=p_C(\widehat{\nabla_{a_1}}(c_1)+_C\gamma_1^\uparrow(c_1))=p_C(\widehat{\nabla_{a_2}}(c_2)+_C\gamma_2^\uparrow(c_2))=c_2$.
  Then
  $\widehat{\nabla_{a_1}}(c_1)+_C\gamma_1^\uparrow(c_1)=\widehat{\nabla_{a_1}}(c_1)+_C\gamma_2^\uparrow(c_1)$
  implies $\gamma_1^\uparrow(c_1)=\gamma_2^\uparrow(c_1)$, and so
  $\gamma_1=\gamma_2$. This shows that $I_\nabla$ is injective.

  Now consider $(V,a)\in R$ and set $p_C(V)=c$. Then
  $Tq_C(\widehat{\nabla_a}(c))=\rho_A(a)=Tq_C(V)$, which implies the
  existence of $\gamma\in q_C^{-1}(q_C(c))$ such that
  $V=\widehat{\nabla_a}(c)+_C\gamma^\uparrow(c)$.  Hence
  $(V,a)=I_\nabla(a,c,\gamma)$, which shows that $I_\nabla$ is surjective.

  The compatibilities of $I_\nabla$ with the projections and additions
  are immediate and left to the reader.
  \end{proof}
  This implies immediately that $R$ is a double vector bundle with
  sides $A$ and $C$, with core $C$ and the projections and additions
  above.

\begin{cor}\label{dvb_R_cor}
  $R$ is a double vector bundle with side projections
  $\pi_A\colon R\to A$, $(V,a)\to a$ and $\pi_C\colon R\to C$,
  $(V,a)\to p_C(V)$ with additions in \eqref{add_A} and \eqref{add_C}
  and with core $C$.
 
  \end{cor}

  In the situation of Proposition \ref{dvb_R}, the isomorphism
  $I_\nabla$ of double vector bundles is a decomposition of $R$.
  Recall that $\nabla$ defines as a consequence linear horizontal
  lifts $\sigma^A_\nabla\colon\Gamma(A)\to\Gamma_C^l(R)$ and
  $\sigma^C_\nabla\colon \Gamma(C)\to\Gamma_A^l(R)$. More precisely,
  given a section $a\in \Gamma(A)$, the section
   $\widehat{a}:=\sigma^A_\nabla(a)\in\Gamma_{C}(\ldbl)$ is defined by
 \begin{equation}
\label{eq:ahat}
 \widehat{a}(c_m)=(\Hat{\nabla_a}(c_m), a(m)),\quad c_m\in C.
\end{equation}
Given a section $c\in\Gamma(C)$, the section $\Check c:=\sigma^C_\nabla(c)\in\Gamma_{A}(\ldbl)$ is defined by
\begin{equation}
\label{eq:ccheck}
\Check c(a_m) =  (\Hat{\nabla_a}(c(m)),  a_m),
\end{equation}
for $a_m\in A$ and any $a\in\Gamma(A)$ such that $a(m)=a_m$.
Given a section $\gamma\in\Gamma(C)$, the corresponding core section
  $\gamma^\times\in\Gamma_{C}(\ldbl)$ is given by
 \begin{equation}\label{gamma_times}
 \gamma^\times(c_m)=(\gamma^\uparrow(c_m), 0^A_m), \quad c_m\in C.
\end{equation}
The corresponding core section 
 $\gamma^\ddagger\in\Gamma_{A}(\ldbl)$ is given by
 \begin{equation}\label{gamma_ddagger}
 \gamma^\ddagger(a_m)=(\gamma^\dagger(\rho_A(a_m)), a_m)=(T0^C(\rho_A(a_m))+_C\gamma^\uparrow(0^C_m),a_m) \quad a\in A.
\end{equation}

\subsection{The comma double Lie algebroid}\label{ruths}
Consider again two Lie algebroids $A$ and $C$ over a common base $M$
and assume a Lie algebroid morphism $\da:=\da_A\colon C\to A$ as in Section
\ref{sec:action}. This section equips $\ldbl$ constructed from $A$ and
$C$ as in Section \ref{lie_algebroid_structures} with two Lie
algebroid structures, over $C$ and over $A$, and shows that they
constitute together a double Lie algebroid $(\ldbl, A, C, M)$ with
core $C$.

As a Lie
algebroid over $C$, $\ldbl$ is the pullback Lie algebroid structure of
$A$ over $q_C\co C\to M$, denoted $q_C^{!!}A$ in \cite{Mackenzie05}.
As a Lie algebroid over $A$, $\ldbl$ is the action Lie algebroid
structure for the action $\Phi$ of $TC\to TM$ on $\rho_A$ as defined
in \eqref{eq:Phi} and \eqref{def_phi}. The remainder of this section
describes these structures in detail.

\begin{prop}\label{lie_algebroid_R_C1}
  Consider a Lie algebroid morphism $\da\colon C\to A$ over $M$
  and equip $R=TC\oplus_{TM}A$ with the
  pullback Lie algebroid structure of $A$ over $q_C$. Consider an $A$-connection
  $\nabla\colon\Gamma(A)\times\Gamma(C)\to \Gamma(C)$ on $C$ and write
  $\Theta_C$ for the anchor $\ldbl\to TC$. Then
  $\Theta_C\circ \widehat{a}=\Hat{\nabla_a}\in \mx^l(C)$ and
  $\Theta_C\circ \gamma^\times=\gamma^\uparrow\in\mx(C)$ for
  $a\in\Gamma(A)$ and $\gamma\in\Gamma(C)$.  Furthermore, the Lie
  algebroid bracket on $\ldbl\to C$ is given by:
 \begin{equation}
 \label{eq:RC}
 \left[\widehat{a_1},
   \widehat{a_2}\right] =
 \widehat{[a_1,a_2]}-\widetilde{R_\nabla(a_1,a_2)},\quad
 \left[\widehat{a}, \gamma^\times\right] = (\nabla_a\gamma)^\times,\quad
 \left[\gamma_1^\times\,,\gamma_2^\times\right] = 0
 \end{equation}
 for $a,a_1,a_2\in\Gamma(A)$ and  $\gamma,\gamma_1,\gamma_2\in\Gamma(C)$.
\end{prop}

\begin{cor}
  In the situation of Proposition \ref{lie_algebroid_R_C1},
  the Lie algebroid structure on $R\to C$ is linear over $A\to M$. In other words,
  $(R\to C, A\to M)$ is a VB-algebroid.
\end{cor}

 \begin{proof}[Proof of Proposition \ref{lie_algebroid_R_C1}]
 The proof follows from the basic properties of pullbacks of Lie
 algebroids \cite[p.156-157]{Mackenzie05}. Following the formula (19) there
 \[\Theta_C(\widehat{a})=\Theta_C(\widehat{\nabla_a},a)=\widehat{\nabla_a}\in\mx^l(C)
 \]
 for all $a\in\Gamma(A)$, and
 \[\Theta_C(\gamma^\times)=\Theta_C(\gamma^\uparrow, 0^A)=\gamma^\uparrow\in\mx^c(C)
   \]
   for all $\gamma\in\Gamma(C)$.

   Following (20) in \cite[p.157]{Mackenzie05},
   \begin{equation*}
     \begin{split}
       \left[\widehat{a_1},\widehat{a_2}\right]&=\left[\left(\widehat{\nabla_{a_1}},a_1\right), \left(\widehat{\nabla_{a_2}},a_2\right)\right]
       =\left(\left[\widehat{\nabla_{a_1}}, \widehat{\nabla_{a_2}}\right], [a_1,a_2]\right)\\
       &=\left(\widehat{\nabla_{[a_1,a_2]}}-\widetilde{R_\nabla(a_1,a_2)}, [a_1,a_2]\right)=\widehat{[a_1,a_2]}-\widetilde{R_\nabla(a_1,a_2)},
     \end{split}
   \end{equation*}
\[ \left[\widehat{a},\gamma^\times\right]=\left[\left(\widehat{\nabla_{a}},a\right), \left(\gamma^\uparrow,0^A\right)\right]
     =\left(\left[\widehat{\nabla_{a}}, \gamma^\uparrow\right], [a,0^A]\right)
     =\left( (\nabla_a\gamma)^\uparrow, 0^A\right)=(\nabla_a\gamma)^\times 
   \]
   and 
   \[ \left[\gamma_1^\times, \gamma_2^\times\right]=\left[\left(\gamma_1^\uparrow,0^A\right),\left(\gamma_2^\uparrow,0^A\right)\right]
     =\left(\left[\gamma_1^\uparrow, \gamma_2^\uparrow\right], [0^A,0^A]\right)
     =\left( 0^{TC}_C,0^A\right)=0
   \]
   for all $a,a_1,a_2\in\Gamma(A)$ and $\gamma,\gamma_1,\gamma_2\in\Gamma(C)$.
   \end{proof}

Define now the connections 
 $\nabla^{\da}\colon \Gamma(C)\times\Gamma(A)\to \Gamma(A)$, 
\[\nabla^{\da}_ca=[\partial c,a]+\partial(\nabla_ac)
\]
and 
$\nabla^{\da}\colon \Gamma(C)\times\Gamma(C)\to \Gamma(C)$, 
\[\nabla^{\da}_cc'=[c,c']+\nabla_{\partial c'}c,
\]
and the tensor $R_\nabla^{\rm
  bas}\in\Omega^2(C,\operatorname{Hom}(A,C))$:
\begin{equation}
  R_\nabla^{\da}(c,c')(a)
  =-\nabla_a[c,c']+[\nabla_ac,c']
  +[c,\nabla_ac']-\nabla_{\nabla^{\da}_ca}c'+\nabla_{\nabla^{\da}_{c'}a}c.
\end{equation}

\begin{prop}\label{lie_algebroid_R_C2}
Consider a Lie algebroid morphism $\da\colon C\to A$ over $M$
  and equip $R=TC\oplus_{TM}A$ with the
  action Lie algebroid structure over $A$. 
  Consider an $A$-connection
  $\nabla\colon\Gamma(A)\times\Gamma(C)\to \Gamma(C)$ on $C$ and write
  $\Theta_A$ for the anchor $\ldbl\to TA$. Then
  $\Theta_A(\Check{c})=\widehat{\nabla^{\da}_c}\in\mx^l(A)$ and
  $\Theta_A(\gamma^\ddagger )=(\partial \gamma)^\uparrow\in\mx(A)$ for
  $c,\gamma\in\Gamma(C)$.  Furthermore, the Lie algebroid bracket on
  $\ldbl\to A$ is given by:
 \begin{equation}
 \label{eq:RA}
 \left[\Check{c_1},
   \Check{c_2}\right] =
 \widecheck{[c_1,c_2]}-\widetilde{R_\nabla^{\da}(c_1,c_2)},\quad
 \left[\Check{c}, \gamma^\ddagger\right] = (\nabla^{\da}_c\gamma)^\ddagger,\quad
 \left[\gamma_1^\ddagger\, \gamma_2^\ddagger\right] = 0
 \end{equation}
 for $c,c_1,c_2,\gamma,\gamma_1,\gamma_2\in\Gamma(C)$.
\end{prop}

\begin{cor}
  In the situation of Proposition \ref{lie_algebroid_R_C2},
  the Lie algebroid structure on $R\to A$ is linear over $C\to M$. In other words,
  $(R\to A, C\to M)$ is a VB-algebroid.
  \end{cor}
 \begin{proof}[Proof of Proposition \ref{lie_algebroid_R_C2}]
   The action Lie algebroid defined by the action \eqref{eq:Phi} of
   $TC\to TM$ on $A\to TM$ lives on the pullback vector bundle of $TC\to TM$
   across $\rho_A$ \cite[Prop.~4.1.2]{Mackenzie05}. As a manifold this
   is $R$, and the bundle projection is the projection onto the second
   factor, $(V,a)\mapsto a$. As with any pullback vector bundle, the
   sections of $R\to A$ are finite sums $\sum_i F_i\, \chi_i^!$ where
   $F_i\in\sfn{A}$ and $\chi_i\in\tmtc$.  Here, for
   $\chi\in\Gamma_{TM}(TC)$, the section $\chi^!\in\Gamma_A(R)$ is
   given by $\chi^!(a)=(\chi(\rho_A(a)),a)$ for all $a\in A$.

   The anchor $\Theta_A\co R\to TA$ is then defined by
   $\Theta_A(\sum_i F_i\, \chi_i) = \sum_i F_i\, \Phi(\chi_i)$. The Lie bracket on $\Gamma_A(R)$ is defined by
 \begin{equation}
 \label{eq:RA1}
 \left[\chi_1^!,\chi_2^!\right] = [\chi_1, \chi_2]^!
 \end{equation}
 for all $\chi_1,\chi_2\in\Gamma_{TM}(TC)$.

 By Definition, the section $\gamma^\ddagger\in\Gamma^c_A(R)$ defined
 by $\gamma\in\Gamma(C)$ equals
 $(\gamma^\dagger)^!\in\Gamma_A(R)$. Consider $c\in\Gamma(C)$, and set
 \[\Gamma(\Hom(A,C))\ni \nabla_\cdot c=\sum_{j=1}^n\alpha_j\otimes c_j\] with
 $\alpha_1,\ldots,\alpha_n\in\Gamma(A^*)$ and $c_1,\ldots, c_n\in\Gamma(C)$.
 It is then easy to check that
 \[ \Gamma^l_A(R)\ni\widecheck{c}=(Tc)^!-\sum_{j=1}^n\ell_{\alpha_j}\cdot(c_j^\dagger)^!.
 \]

 Simple computations yield  $\Theta_A(\gamma^\ddagger)=\Phi(\gamma^\dagger)=(\da\gamma)^\uparrow$
 and
 \[\Theta_A(\widecheck{c})=\Phi(Tc)-\sum_{j=1}^n\ell_{\alpha_j}\cdot\Phi(c_j^\dagger)
   =\widehat{[\da c,\cdot]}-\sum_{j=1}^n\ell_{\alpha_j}\cdot(\da c_j)^\uparrow=\widehat{\nabla^\da_c}
 \]
 for $\gamma,c\in\Gamma(C)$.  The Lie bracket on $\Gamma_A(R)$ is then given by
 \[ \left[ \gamma_1^\ddagger,\gamma_2^\ddagger\right]=\left[\gamma_1^\dagger, \gamma_2^\dagger\right]^!=0,
   \]
   \begin{equation*}
     \begin{split}
       \left[\widecheck{c}, \gamma^\ddagger\right]&=\left[(Tc)^!-\sum_{j=1}^n\ell_{\alpha_j}\cdot(c_j^\dagger)^!, (\gamma^\dagger)^!\right]
       =[c,\gamma]^\ddagger+\sum_{j=1}^n\ldr{\Theta_A((\gamma^\dagger)^!)}(\ell_{\alpha_j})\cdot c_j^\ddagger\\
       &=[c,\gamma]^\ddagger+\sum_{j=1}^n\ldr{(\da\gamma)^\uparrow}(\ell_{\alpha_j})\cdot c_j^\ddagger
       =[c,\gamma]^\ddagger+\sum_{j=1}^nq_A^*\langle\da\gamma,\alpha_j\rangle\cdot c_j^\ddagger
       =([c,\gamma]+\nabla_{\da \gamma}c)^\ddagger=(\nabla^\da_c\gamma)^\ddagger,
     \end{split}
   \end{equation*}
   and for $\widecheck c=(Tc)^!-\sum_{j=1}^n\ell_{\alpha_j}\cdot(c_j^\dagger)^!$ and $\widecheck{c'}=(Tc)^!-\sum_{k=1}^n\ell_{\alpha_k'}\cdot(c_k^\dagger)^!$,
   \begin{equation*}
     \begin{split}
       \left[\widecheck{c},\widecheck{c'}\right]&=\left[(Tc)^!, (Tc')^! \right]-\left[\widecheck{c}, \sum_{k=1}^n\ell_{\alpha_k'}\cdot c_k^\ddagger\right]
       -\left[\sum_{j=1}^n\ell_{\alpha_j}\cdot c_j^\ddagger, \widecheck{c'}\right]-\left[\sum_{j=1}^n\ell_{\alpha_j}\cdot c_j^\ddagger, \sum_{k=1}^n\ell_{\alpha_k'}\cdot c_k^\ddagger\right]\\
       &=(T[c,c'])^!-\sum_{k=1}^n\ell_{(\nabla^\da)^*_c\alpha_k'}\cdot c_k^\ddagger-\sum_{k=1}^n\ell_{\alpha_k'}\cdot(\nabla^\da_cc_k)^\ddagger
       +\sum_{j=1}^n\ell_{(\nabla^\da)^*_{c'}\alpha_j} \cdot c_j^\ddagger+\sum_{j=1}^n\ell_{\alpha_j}\cdot(\nabla^\da_{c'}c_j)^\ddagger\\
       &\qquad\qquad -\sum_{j,k=1}^n\ell_{\alpha_j}\cdot q_A^*\langle \da c_j,\alpha_k'\rangle \cdot c_k^\ddagger+\sum_{j,k=1}^n\ell_{\alpha_k'}\cdot q_A^*\langle \da c_k,\alpha_j\rangle \cdot c_j^\ddagger.
     \end{split}
   \end{equation*}
   A straightforward computation yields that this is $\left[\widecheck{c},\widecheck{c'}\right]=\widecheck{[c,c']}-\widetilde{R_\nabla^\da(c,c')}$.
\end{proof}

\bigskip

Let $A$ and $C$ be two Lie algebroids over $M$ and
$\partial\colon C\to A$ a Lie algebroid morphism over $M$. Choose an
$A$-connection on $C$:
\[\nabla\colon\Gamma(A)\times\Gamma(C)\to \Gamma(C).
\]
\eqref{eq:RC} and \eqref{eq:RA} give the representations up to
homotopy describing the VB-algebroid structures $R\to C$ and
respectively $R\to A$ in the splitting defined by $\nabla$ as in
Proposition \ref{dvb_R}. 
Via the splitting given by $\nabla$, the Lie algebroid $R\to C$ is described by
\begin{equation}\label{ruth_A}
  \id\colon C\to C,\quad \nabla\colon\Gamma(A)\times\Gamma(C)\to\Gamma(C), \quad
R_\nabla\in\Omega^2(A,\operatorname{Hom}(C,C)).
\end{equation}
The Lie algebroid $R\to A$ is described by
\begin{equation}\label{ruth_C}
  \partial\colon C\to A,\quad \nabla^{\da}\colon\Gamma(C)\times\Gamma(C)\to\Gamma(C), \quad \nabla^{\da}\colon\Gamma(C)\times\Gamma(A)\to\Gamma(A),
  \end{equation}
and
\begin{equation}\label{ruth_C_1}
  R^{\da}_\nabla\in\Omega^2(C,\operatorname{Hom}(A,C)).
\end{equation}
Straightforward computations (carried out in Appendix
\ref{proof_of_R_lba}) yield that these two representations up to
homotopy are matched in the sense of \cite{GrJoMaMe18} and so that
$R=TC\oplus_{TM}A$, with its two VB-Lie algebroid structures, is a
double Lie algebroid with sides $A$ and $C$, and with core $C$.
This completes the proof of the following theorem.
\begin{thm}\label{main_comma}
  Let $A$ and $C$ be Lie algebroids over $M$ and let
  $\partial\colon C\to A$ be a morphism of Lie algebroids.  Then the
  double vector bundle $(R;A,C;M)$ constructed in Section
  \ref{lie_algebroid_structures} inherits a double Lie algebroid
  structure with the following two Lie algebroid structures on $R\to A$ and
  $R\to C$:
  \begin{enumerate}
  \item As a Lie
algebroid over $C$, $\ldbl$ is the pullback Lie algebroid structure of
$A$ over $q_C\co C\to M$.
\item As a Lie algebroid over $A$, $\ldbl$ is the action Lie algebroid
structure for the action $\Phi$ of $TC\to TM$ on $\rho_A$ as defined
in \eqref{eq:Phi} and \eqref{def_phi}.
  \end{enumerate}
  The core of $(R;A,C;M)$ is $C$ and its
  core diagram is as in Figure~\ref{fig:R}.
As a consequence, $\ldbl$ is transitive if and only if $\da$ is surjective.
\end{thm}

\begin{figure}[h]
$$ 
\xymatrix{
&\operatorname{kern}{\da}\ar[d]\ar[dr]&\\
0\ar[r]\ar[dr]&C\ar[r]_{\id}\ar[d]^{\da}&C\\
&A&
}
$$ 
\caption{The core diagram of R.\label{fig:R}}
\end{figure}
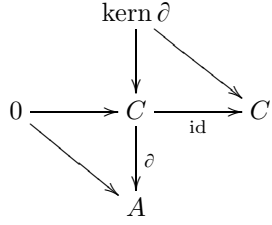

\subsection{Morphisms of comma-double Lie algebroids}
Consider in this section a square of Lie algebroid morphisms over $M$
as in Figure \ref{fig:square}.
\begin{figure}[h]
$$ 
\xymatrix{
  C\ar[r]^{\da_A}\ar[d]_{\phi_C}&A\ar[d]^{\phi_A}\\
  C'\ar[r]_{\da_{A'}}&A'
}
$$ 
\caption{Square of Lie algebroid morphisms.\label{fig:square}}
\end{figure}
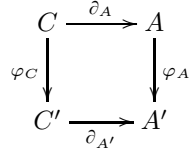
This section proves that such a square induces a morphism of the
corresponding comma double Lie algebroids
$TC\oplus_{TM}A\to TC'\oplus_{TM'}A'$, with side morphisms $\phi_A$ and
$\phi_C$ and with core morphism $\phi_C$.

\begin{prop}\label{mor_dvb}
  Consider a square of Lie algebroid morphisms as in Figure
  \ref{fig:square}.  Then
  \[ \Phi\colon TC\oplus_{TM}A\to TC'\oplus_{TM'} A', \qquad (V,a)\mapsto (T\phi_CV, \phi_A(a))
  \]
  defines a morphism of double vector bundles with side morphisms
  $\phi_A\colon A\to A'$ and $\phi_C\colon C\to C'$ and with core
  morphism $\varphi_C\colon C\to C'$.
\end{prop}

\begin{proof}
  First of all, $\Phi$ is well-defined since
  for $(V,a)\in TC\oplus_{TM}A$, the image
  $\Phi(V,a)=(T\phi_CV, \phi_A(a))$ satisfies
  \[ Tq_{C'}T\phi_CV=Tq_CV=\rho_A(a)=\rho_{A'}\phi_A(a)
  \]
  and is hence an element of $TC'\oplus_{TM}C'$.
Choose an $A$-connection
$\nabla\colon\Gamma(A)\times\Gamma(C)\to \Gamma(C)$ on $C$, and
an $A'$-connection $\nabla'\colon\Gamma(A')\times\Gamma(C')\to \Gamma(C')$ on $C'$.
Consider the induced linear splittings as in Proposition \ref{dvb_R}
  \[ I_\nabla\colon A\times_MC\times_MC \to TC\oplus_{TM}A,\qquad 
  (a,c,\gamma)\mapsto \left(\widehat{\nabla_a}(c)+_C\gamma^\uparrow(c),  a\right)
\]
and
\[ I_{\nabla'}\colon A'\times_MC'\times_MC' \to TC'\oplus_{TM}A',\qquad 
  (a',c',\gamma')\mapsto \left(\widehat{\nabla'_{a'}}(c')+_{C'}{\gamma'} ^\uparrow(c'),  a'\right).
\]
Define $\omega_{\nabla,\nabla'}\in\Gamma(A^*\otimes C^*\otimes C')$ by
\begin{equation}\label{omega_nabla_nabla'} \omega_{\nabla,\nabla'}(a,c)=\nabla'_{\phi_A(a)}(\phi_C(c))-\phi_C(\nabla_ac)
\end{equation}
for $a\in\Gamma(A)$ and $c\in\Gamma(C)$.
Then an easy computation proves that
\[ T_c\phi_c\widehat{\nabla_a}(c)-\widehat{\nabla'_{\phi(a)}}(\phi_C(c))=\omega_{\nabla,\nabla'}(a,c)^\uparrow(\phi_C(c))
\]
for all $c\in C$ and $a\in A$. Therefore,
\begin{equation}\label{Phi_in_decs}
  \begin{split}
    (I_{\nabla'}\inv\circ \Phi\circ I_\nabla)(a,c,\gamma)&=\left(I_{\nabla'}\inv\circ \Phi)(\widehat{\nabla_a}(c)+_C\gamma^\uparrow(c), a\right)=
    I_{\nabla'}\inv\left(T\phi_C(\widehat{\nabla_a}(c))+_CT\phi_C(\gamma^\uparrow(c)), \phi_A(a)\right)\\
    &=I_{\nabla'}\inv\left(\widehat{\nabla'_{\phi_A(a)}}(\phi(c))+_C(\phi_C\gamma+\omega_{\nabla,\nabla'}(a,c))^\uparrow(\phi_C(c)), \phi_A(a)\right)\\
    &=(\phi_A(a),\phi_C(c), \phi_C(\gamma)+\omega_{\nabla,\nabla'}(a,c)).
    \end{split}
  \end{equation}
  This shows that
  $I_{\nabla'}\inv\circ \Phi\circ I_\nabla\colon
  A\times_MC\times_MC\to A'\times_MC'\times_MC'$ is a double vector
  bundle morphism, and consequently that $\Phi$ is a double vector
  bundle morphism.
\end{proof}

The following theorem is then easily proved in decompositions, see Appendix \ref{morphisms_proofs}.
\begin{thm}\label{comma_mor}
  Consider a square of Lie algebroid morphisms as in Figure
  \ref{fig:square}. Then  the morphism 
  \[ \Phi\colon TC\oplus_{TM}A\to TC'\oplus_{TM'} A', \qquad (V,a)\mapsto (T\phi_CV, \phi_A(a))
  \]
  of double vector bundles as in Proposition \ref{mor_dvb}
  is a morphism of the double Lie algebroids constructed in Theorem \ref{main_comma}.
\end{thm}

\section{The quotient}\label{sec:quotient}
Let now $A$ and $B$ be Lie algebroids over a common base $M$ and
consider a transitive core diagram for $A$ and $B$ as in Figure
\ref{fig:calC} (see Definition \ref{df:cd}). 
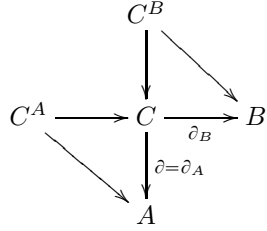
\begin{figure}[h]
$$ 
\xymatrix{
&C^B\ar[d]\ar[dr]&\\
C^A\ar[r]\ar[dr]&C\ar[r]_{\da_B}\ar[d]^{\da=\da_A}&B\\
&A&
}
$$ 
\caption{Setting of Section \ref{sec:quotient}, the transitive core
  diagram $\mathcal C$. \label{fig:calC}}
\end{figure}
Recall the representation $\nabla^A\co A\to\Der(C^A)$ defined in (\ref{eq:cdxm}).
 For $a\in\Ga(A)$,  consider the vector field 
\[\widehat{\nabla^A_a}\in\mx^l(C^A)\] 
corresponding to the derivation $\nabla^A_a$ of $C^A$ over $\rho_A(a)$.

 \begin{lem}\label{useful1}
   In the situation above, choose
   $c\in\Gamma(C)$ and set
   $a=\partial_Ac$.  Then the vector field
   $\Hat{[c,\cdot]}\in\mx^l(C)$ restricts to
   $\widehat{\nabla^A_a}$ on points of
   $C^A$. In other words, if $\iota\colon C^A\hookrightarrow
   C$ is the inclusion, then
   $\widehat{\nabla^A_a}\sim_\iota\Hat{[c,\cdot]}$.
 \end{lem}

\begin{proof}
  Check first that $\Hat{[c,\cdot]}(\gamma_m)\in T_{\gamma_m}C^A$ for
  all $\gamma_m\in C^A$. In order to do this, it is sufficient to show
  $\Hat{[c,\cdot]}(F)=0$ for all $F\in C^\infty(C)$ such that
  $F\an{C^A}$ is constant. Since $C^A\subseteq C$ is a subbundle, it
  suffices to check this on the linear functions
  $\ell_\mu\in C^\infty(C)$ for
  $\mu\in\Gamma\left((C^A)^\circ\right)\subseteq \Gamma(C^*)$.  But
  this is immediate since
\[\Hat{[c,\cdot]}(\ell_\mu)(\gamma)=\rho_C(c)\langle \mu, \gamma\rangle -\langle \mu, [c,\gamma]\rangle=0
\]
for all $\gamma\in\Gamma(C^A)$, since $[c,\gamma]\in\Gamma(C^A)$.

\medskip

 Then
$\iota^*q_C^*f=q_{C^A}^*f$ for all $f\in C^\infty(M)$ and 
$\iota^*\ell_\mu=\ell_{\iota^t\mu}$ for all $\mu\in\Gamma(C^*)$.
The definitions of $\Hat{[c,\cdot]}$ and 
$\widehat{\nabla^A_a}$ yield easily $(\nabla_a^A)^*(\iota^t\mu)=\iota^t\ldr{c}\mu$ and so
\[ \ldr{\widehat{\nabla^A_a}} (\iota^*\ell_{\mu})
  =\ldr{\widehat{\nabla^A_a}} \ell_{\iota^t\mu}=\ell_{(\nabla^A_a)^*\iota^t\mu}=\iota^*\ell_{\ldr{c}\mu}
  =\iota^*(\Hat{[c,\cdot]}(\ell_\mu))\]
for $\mu\in\Gamma(C^*)$.
Similarly, together with $\rho_C(c)=\rho_A(\da c)=\rho_A(a)$, the definitions of $\Hat{[c,\cdot]}$ and 
$\widehat{\nabla^A_a}$ yield
    \[
\ldr{\widehat{\nabla^A_a}} (\iota^*q_C^*f)=\ldr{\widehat{\nabla^A_a}}(q_{C^A}^*f)=q_{C^A}^*(\ldr{\rho_A(a)}(f))=\iota^*q_C^*(\ldr{\rho_C(c)}(f))
=\iota^*(\ldr{\Hat{[c,\cdot]}}(q_C^*f))
\]
for $f\in C^\infty(M)$.
\end{proof}

\medskip

 Define the subset $\Pi\subseteq \ldbl$ by
 \begin{equation}\label{def_pi}
 \Pi:=\left\{\left.\left(\widehat{\nabla^A_{a}} (\gamma_m), a(m)\right)\in TC^A\oplus_{TM} A\,\right|
       a\in \Ga(A),\ \gamma_m\in C^A\right\}.
 \end{equation}
 By  Lemma \ref{useful1}, it equals
     \begin{equation*}  
\left\{\left.\left(\widehat{[c,\cdot]} (\gamma_m), \da c(m)\right)\in TC^A\oplus_{TM} A\,\right|
 c\in \Ga(C),\ \gamma_m\in C^A\right\}.
\end{equation*}

\begin{prop}
  Choose any connection
  $\nabla\colon \Gamma(A)\times\Gamma(C)\to\Gamma(C)$ such that
  $\nabla_a\gamma=\nabla^A_a\gamma$ for all $a\in\Gamma(A)$ and
  $\gamma\in\Gamma(C^A)$.  Then the isomorphism
  $I_\nabla\colon A\times_MC\times_MC\to R$ of double vector bundles
  in Proposition \ref{dvb_R} and Corollary \ref{dvb_R_cor} does
  \[ I_\nabla(A\times_MC^A\times_M0^C)=\Pi.
  \]
  As a consequence, $\Pi$ is a sub-double vector bundle of $R$ with
  sides $A$ and $C^A$ and with trivial core.
\end{prop}

\begin{proof}
  As in the proof of Lemma \ref{useful1}, prove that
  $\mx^l(C^A)\ni\widehat{\nabla^A_a}\sim_\iota\widehat{\nabla_a}\in\mx^l(C)$.
  Then given $(a_m,\gamma_m,0^C_m)\in A\times_MC^A\times_M0^C$ and any $a\in\Gamma(A)$ such that $a(m)=a_m$:
  \[ I_\nabla(a_m,\gamma_m,0^C_m)=\left(\widehat{\nabla_a}(\gamma_m), a_m\right)
    =\left(\widehat{\nabla^A_a}(\gamma_m), a_m\right)\in \Pi\subseteq TC^A\oplus_{TM} A.
    \qedhere\]
  \end{proof}

  The following corollary is immediate, using the notation \eqref{eq:ccheck}.
  \begin{cor}\label{spanning_R_A}
    Choose any connection
  $\nabla\colon \Gamma(A)\times\Gamma(C)\to\Gamma(C)$ such that
  $\nabla_a\gamma=\nabla^A_a\gamma$ for all $a\in\Gamma(A)$ and
  $\gamma\in\Gamma(C^A)$.  Then
the subbundle $\Pi\to A$ of $R\to A$ is spanned by the sections 
\[ \sigma_\nabla^C(\gamma)=\Check\gamma\in\Gamma^l_A(R)
  \]
 for all $\gamma\in\Gamma(C^A)$.
\end{cor}

\subsection{The quotient of $R$ by $\Pi$.}
This section defines the quotient $R/\Pi$ and shows that it inherits a
double vector bundle structure. In the situation above define $\sim$
on $R$ by
\[ (V,a)\sim (V', a'):\Leftrightarrow a=a' \text{ and } \,\,\exists
  \,W\in TC^A: (W,a)\in \Pi \text{ and } V+_{TM}W=V'.\] Then it is
easy to check that $\sim$ defines an equivalence relation on $R$. Two
pairs $(V,a)$ and $(V', a)\in R$ are equivalent if and only if there
exists $(W,a)\in \Pi$ such that $(V,a)+_A(W,a)=(V',a)$.  Hence, it is
easy to see that $R/_\sim=: R/\Pi$ has a vector bundle structure over
$A$: the quotient of the vector bundle $R\to A$ by its subbundle
$\Pi\to A$. The reader should think about $R/\Pi$ as this
quotient. The following theorem shows that the vector bundle structure
on $R\to C$ induces a vector bundle structure on $R/\Pi\to C/C^A$, such that
$(R/\Pi,A,C/C^A,M)$ becomes a double vector bundle.

\begin{thm}\label{main_thm_quotient}
  \begin{enumerate}
\item  The space $R/\Pi$ inherits from $R$ a double vector bundle structure
  with core $C$ and with side $C/C^A\simeq B$.
\item Choose any connection
  $\nabla\colon \Gamma(A)\times\Gamma(C)\to\Gamma(C)$ such that
  $\nabla_a\gamma=\nabla^A_a\gamma$ for all $a\in\Gamma(A)$ and
  $\gamma\in\Gamma(C^A)$. Then the induced decomposition
  $I_\nabla\colon A\times_MC\times_MC\to R$ induces a decomposition
  \[ \overline{I_\nabla}\colon A\times_M(C/C^A)\times_MC\to R/\Pi,
    \qquad \overline{I_\nabla}(a,\overline c,\gamma):=
    I_\nabla(a,c,\gamma)+_A\Pi.
    \]
\end{enumerate}
\end{thm}

The remainder of this section proves this theorem.  Define
$F_C\subseteq TC$ as the subbundle over $C$ spanned by the vertical
vector fields corresponding to sections of $C^A$, i.e. for all
$c_m\in C$:
 \[F_C(c_m)=\left\{\gamma^\uparrow(c_m)\mid \gamma\in\Gamma(C^A)\right\}
 \]
 and so
 \[
   \Gamma(F_C)=\operatorname{span}_{C^\infty(C)}\left\{\gamma^\uparrow\mid
     \gamma\in\Gamma(C^A)\right\}.\] The subbundle $F_C\subseteq TC$
 is involutive and simple since its leaf space is $C/C^A\simeq B$.

 Choose a linear $A$-connection
 $\nabla\colon \Gamma(A)\times\Gamma(C)\to\Gamma(C)$ that restricts to
 $\nabla^A$ on $\Gamma(C^A)$, and consider the induced sections
 $\widehat a\in\Gamma_C^l(R)$ for all $a\in\Gamma(A)$, as in
 \eqref{eq:ahat}. Recall also from \eqref{gamma_times} that each
 section $\gamma\in\Gamma(C)$ defines a core section
 $\gamma^\times\in\Gamma^c_C(R)$.
 \begin{prop}
   In the situation above, there is a unique flat connection
\[ \nabla^i\colon\Gamma_C(F_C)\times\Gamma_C(R)\to\Gamma_C(R),
\]
such that
\[ \nabla^i\widehat a=0=\nabla^i\gamma^\times
\]
for all $a\in\Gamma(A)$ and all $\gamma\in\Gamma(C)$. The connection
$\nabla^i$ does not depend on the choice of $\nabla$ extending
$\nabla^A$.
\end{prop}

\begin{proof}
  The space of sections of the vector bundle $R\to C$ is spanned as a
  $C^\infty(C)$-module by sections of the form $\widehat{a}$ and
  $\gamma^\times$ for $a\in\Gamma(A)$ and $\gamma\in\Gamma(C)$. Define
  \[\nabla^i\colon\Gamma_C(F_C)\times\Gamma_C(R)\to\Gamma_C(R)
  \]
  by
  \[ \nabla^i_\chi\left(\sum_{j=1}^sF_j\widehat{a_j}+\sum_{l=1}^tG_l\gamma_l^\times\right)
    =\sum_{j=1}^s\ldr{\chi}(F_j)\cdot \widehat{a_j}+\sum_{l=1}^t\ldr{\chi}(G_l)\cdot\gamma_l^\times
  \]
  for $F_1,\ldots,F_s,G_1,\ldots,G_t\in C^\infty(C)$,
  $a_1,\ldots,a_s\in\Gamma(A)$ and $\gamma_1,\ldots,\gamma_t\in\Gamma(C)$.

  For $f\in C^\infty(M)$, $a\in\Gamma(A)$ and $\gamma\in\Gamma(C)$
  \[ \widehat{fa}=q_C^*f\cdot\widehat{a} \qquad \text{ and }\qquad (f\gamma)^\times =q_C^*f\cdot\gamma^\times.
  \]
  Hence $\nabla^i$ is well-defined if
  $\nabla^i_\cdot
  \widehat{fa}=\nabla^i_\cdot\left(q_C^*f\cdot\widehat{a}\right)$ and
  $\nabla^i_\cdot
  (f\gamma)^\times=\nabla^i_\cdot\left(q_C^*f\cdot\gamma^\times\right)$
  for all $f\in C^\infty(M)$, $a\in\Gamma(A)$ and
  $\gamma\in\Gamma(C)$.  On the one hand,
  $\nabla^i_\chi\widehat{fa}=0=\nabla^i_\chi(f\gamma)^\times$ and on
  the other hand
  \[\nabla^i_\chi\left(q_C^*f\cdot\widehat{a}\right)=\ldr{\chi}(q_C^*f)\cdot\widehat{a}=0
  \]
  and
  \[\nabla^i_\chi\left(q_C^*f\cdot\gamma^\times\right)=\ldr{\chi}(q_C^*f)\cdot\gamma^\times=0
  \]
  since $\chi\in\Gamma(F_C)$ and $F_C\subseteq T^{q_C}C$.

  The connection $\nabla^i$ is obviously flat with the required flat
  sections. Let $\nabla'\colon\Gamma(A)\times\Gamma(C)\to\Gamma(C)$ be
  a second extension of
  $\nabla^A\colon\Gamma(A)\times\Gamma(C^A)\to\Gamma(C^A)$.  Then
  $\phi:=\nabla-\nabla'\in\Omega^1(A,(C^A)^\circ\otimes C)$ does
  $\sigma_A^{\nabla'}(a)=\sigma_A^\nabla(a)-\widetilde{\varphi(a)}=\widehat{a}-\widetilde{\varphi(a)}$
  for all $a\in\Gamma(A)$.  Writing
  $\varphi(a)=\sum_{j=1}^k\nu_j\otimes c_j$ with
  $\nu_1,\ldots,\nu_k\in\Gamma((C^A)^\circ)$ and
  $c_1,\ldots,c_k\in\Gamma(C)$ yields
  \[\nabla^i_\chi\left(\sigma_A^{\nabla'}(a)\right)=\nabla^i_\chi\left(\widehat{a}-\widetilde{\varphi(a)}\right)
    =\nabla^i_\chi\left(\sum_{j=1}^k\ell_{\nu_j}\cdot c_j^\times\right)=\sum_{j=1}^k\ldr{\chi}(\ell_{\nu_j})\cdot c_j^\times=0
  \]
  since $\chi\in\Gamma(F_C)$ and $\ell_{\nu_j}\in C^\infty(C)^{F_C}$.
  \end{proof}

 The last proposition is reformulated as follows.
  \begin{cor}
    The triple $(F_C,0,\nabla^i)$ is an infinitesimal ideal system in
    the vector bundle $R\to C$ considered as a trivial Lie algebroid.
  \end{cor}

  The quotient $R/\nabla^{i}\to C/F_C$ is construted as follows.  Two
  vectors $(V,a)$ and $(V',a')\in R$ are equivalent if $c_m=p_C(V)$
  and $c_m'=p_C(V')$ lie in the same leaf of $F_C$, i.e. they can be
  joined by an integral path of $F_C$ and there exists a parallel
  section along this path taking values $(V,a)$ at the point $c_m$ and
  $(V',a')$ at $c_m'$. Note that here $c_m=p_C(V)$ and $c_m'=p_C(V')$
  lie in the same leaf of $F_C$ if an only if $c_m-c_m'\in C^A$.
  Recall as well that $(V,a)$ and $(V',a') \in R$ project to the same
  element in $R/\Pi$ if
 \[a=a', \qquad p_C(V)-p_C(V')\in C^A \quad \text{ and }\quad
 V-_{TM}V'=\widehat{\nabla_a^A}(p_C(V)-p_C(V')).
 \]

  \begin{prop}\label{eq_of_eqs}
    In the situation above, the two equivalence relations $\sim$ and
    $\sim_{\nabla^i}$ coincide.
    \end{prop}

\begin{proof}
  Assume first that $(V,a)$ and $(V',a')$ are equivalent via
  $\sim_{\nabla^i}$. Let $c_m=p_C(V)$ and $d_{m'}=p_C(V')$.  Since
  $c_m$ and $d_{m'}$ are in the same leaf of $F_C$, they must be in
  the same fiber of $q_C\colon C\to M$ and hence $m=m'$.  Furthermore
  $c_m-d_m\in C^A(m)$ and there exists without loss of generality
one $\nabla^{i}$-parallel
  section $\chi$ of $R\to C$ such that $\chi(c_m)=(V, a)$ and
  $\chi(d_m)=(V', a')$.  Write
\[\chi=\sum_{i=1}^kF_i\widehat{a_i} +\sum_{j=1}^lG_jc_j^\times
\]
with $a_1,\ldots,a_k\in\Gamma(A)$, $c_1,\ldots,c_l\in\Gamma(C)$ and 
$F_1,\ldots,F_k,G_1,\ldots,G_l\in C^\infty(C)^{F_C}$.
Since the functions $F_i,G_j$ are all $F_C$-invariant, 
$F_i(c_m)=F_i(d_m)=:\alpha_i$ and $G_j(c_m)=G_j(d_m)=\beta_j$
for $i=1,\ldots,k$ and $j=1,\ldots,l$.

Thus
\begin{align*}
  (V,a)&=\sum_{i=1}^k\alpha_i\widehat{a_i}(c_m) +\sum_{j=1}^l\beta_jc_j^\times(c_m)=\left(\sum_{i=1}^k\alpha_i\widehat{\nabla_{a_i}}(c_m)+\sum_{j=1}^l\beta_jc_j^\uparrow(c_m), 
         \sum_{i=1}^k\alpha_ia_i(m)\right) 
\end{align*}
and 
\begin{align*}
  (V',a')&=\sum_{i=1}^k\alpha_i\widehat{a_i}(d_m) +\sum_{j=1}^l\beta_jc_j^\times(d_m)=\left(\sum_{i=1}^k\alpha_i\widehat{\nabla_{a_i}}(d_m)+\sum_{j=1}^l\beta_jc_j^\uparrow(d_m), 
           \sum_{i=1}^k\alpha_ia_i(m)\right) 
\end{align*}
This implies $a=\sum_{i=1}^k\alpha_ia_i(m)=a'$ and the difference
$V-_{TM}V'$ in the fiber of $TC\to TM$ over $\rho_A(a)$ equals by the interchange law
\begin{align*}
  V-_{TM}V'&=\left(\sum_{i=1}^k\alpha_i\widehat{\nabla_{a_i}}(c_m)-_{TM}\sum_{i=1}^k\alpha_i\widehat{\nabla_{a_i}}(d_m)\right)+_C\left(\sum_{j=1}^l\beta_jc_j^\uparrow(c_m)
  -_{TM}\sum_{j=1}^l\beta_jc_j^\uparrow(d_m)\right)\\
&=\left(\sum_{i=1}^k\alpha_i\Hat{\nabla_{a_i}}(c_m-d_m)\right)+_{C}0^{TC}_{c_m-d_m}=\left(\sum_{i=1}^k\alpha_i\Hat{\nabla^A_{a_i}}(c_m-d_m)\right)=\Hat{\nabla_a^A}(c_m-d_m).
\end{align*}
This shows that $(V,a)$ and $(V',a')$ are equivalent in $R/\Pi$.

\medskip

Conversely, consider $(V,a_m)$ and $(V',a_m)\in R$, set $c_m=p_C(V)$
and $d_{m}=p_C(V')$ and choose a section $a\in\Gamma(A)$ with
$a(m)=a_m$.  If $(V,a_m)$ and $(V',a_m)$ are equivalent in $R/\Pi$,
i.e.~if $c_m-d_m\in C^A(m)$ and $V-V'=\Hat{\nabla^A_a}(c_m-d_m)$,
then, since $T{q_C}V=\rho_A(a_m)=T{q_C}\Hat{\nabla_a}(c_m)$, there
exists $\gamma\in\Gamma(C)$ such that
$V=\Hat{\nabla_a} (c_m)+\gamma^\uparrow(c_m)$.  In the same manner,
there is a section $\gamma'\in\Gamma(C)$ such that
$V'=\Hat{\nabla_a} (d_m)+{\gamma'}^\uparrow(d_m)$.  Furthermore, since
$V-V'=\Hat{\nabla_a^A}(c_m-d_m)=\Hat{\nabla_a} (c_m-d_m)$ by the
choice of $\nabla$, $V-V'=\Hat{\nabla_a} (c_m)-\Hat{\nabla_a} (d_m)$
and hence $\gamma(m)=\gamma'(m)$.  Set
$\chi=\widehat{a}+\gamma^\times$.  Then $\chi$ is a
$\nabla^{i}$-parallel section of $R\to C$ and
\[(V,a_m)=\chi(c_m)\qquad \text{ and }\qquad (V',a_m)=\chi(d_m).
\qedhere\] 
\end{proof}

The previous proposition shows that there is a well-defined
projection
\[ \pi_B\colon R/\Pi\to B, \qquad (V,a)+\Pi\mapsto \da_B(p_C(V))
\]
and
a well-defined addition
\[ +_B\colon (R/\Pi)\times_B(R/\Pi)\to R/\Pi
\]
in the fibers of $\pi_B$ which is defined as follows.  If $(V,a_m)$
and $(V',a'_m)\in R$ are such that
$\pi_B((V,a_m)+\Pi)=\da_B(p_C(V))=\da_B(p_C(V'))=\pi_B((V',a'_m)+\Pi)$,
then $p_C(V)-p_C(V')\in C^A_m$.  Set $c_m=p_C(V)$ and $c'_m=p_C(V')$. Choose a section
$a'\in\Gamma(A)$ such that $a'(m)=a'_m$, and $\gamma\in\Gamma(C)$ such
that $(V',a'_m)=\widehat{a'}(c_m')+_C\gamma^\times(c_m')$.
Since $\widehat{a'}$ and $\gamma^\times$ are $\nabla^i$-flat and $c_m-c_m'\in C^A$, the vectors
$(V',a')$ and $(\widehat{a'}(c_m)+\gamma^\times(c_m)$ are $\nabla^i$-equivalent and
so
\[(V',a')\sim \left(\widehat{a'}(c_m)+\gamma^\times(c_m), a'\right)
\]
by Proposition \ref{eq_of_eqs}. Then set
\[  ((V,a)+_A\Pi)+_B((V',a')+_A\Pi):=\Bigl((V,a)+_C\left((\widehat{a'}(c_m)+_C\gamma^\times(c_m)\right)\Bigr)+_A\Pi.
\]

\begin{proof}[Proof of Theorem \ref{main_thm_quotient}]
  Choose as before a connection
  $\nabla\colon \Gamma(A)\times\Gamma(C)\to\Gamma(C)$ such that
  $\nabla_a\gamma=\nabla^A_a\gamma$ for all $a\in\Gamma(A)$ and
  $\gamma\in\Gamma(C^A)$.  Consider the induced decomposition
  $I_\nabla\colon A\times_MC\times_MC\to R$. Since
  $I_\nabla(a,c,0)\in \Pi$ for $(a,c,0)\in A\times_MC^A\times_M0$, the
  morphism
  \begin{equation*}
\begin{xy}
\xymatrix{
A\times_M C \times_M C\ar[r]^{\qquad I_\nabla}\ar[d]_{\pi_A}& R\ar[d]^{\pi_A}\\
A\ar[r]_{\id_A} & A}
\end{xy}
\end{equation*}
of vector bundles over $A$
factors to a  morphism
 \begin{equation*}
\begin{xy}
\xymatrix{
A\times_M C/C^A \times_M C\ar[r]^{\qquad \overline{I_\nabla}}\ar[d]_{\pi_A}& R/\Pi\ar[d]^{\pi_A}\\
A\ar[r]_{\id_A} & A}
\end{xy}
\end{equation*}
which is given by
  \[ \overline{I_\nabla}\colon A\times_M(C/C^A)\times_MC\to R/\Pi,
    \qquad \overline{I_\nabla}(a,\overline c,\gamma):=
    I_\nabla(a,c,\gamma)+_A\Pi.
  \]
  Since $I_\nabla(A\times_MC^A\times_M0)=\Pi$, the reduced morphism
  $\overline{I_\nabla}$ is again an isomorphism of vector bundles over
  $A$.

  Consider $b_m=\da_Bc_m=\da_B c'_m\in C/C^A$,
  $\gamma,\gamma'\in \Gamma(C)$ and $a,a'\in \Gamma(A)$.  Then
  \begin{equation*}
    \begin{split}
      &\overline{I_\nabla}(a(m),\overline{c_m},\gamma_m)+_B\overline{I_\nabla}(a'(m),\overline{c_m'},\gamma'_m)\\
    =\,&\left(\left(\widehat{a}(c_m)+_C\gamma^\times(c_m)\right)+_A\Pi\right)+_B\left(\left(\widehat{a'}(c'_m)+_C{\gamma'} ^\times(c_m')\right)+_A\Pi\right)\\
    =\,&\left(\left(\widehat{a}(c_m)+_C\gamma^\times(c_m)\right)+_A\Pi\right)+_B\left(\left(\widehat{a'}(c_m)+_C{\gamma'}^\times(c_m)\right)+_A\Pi\right)\\
    =\,&\left(\widehat{a+a'}(c_m)+_C(\gamma+\gamma')^\times(c_m)\right)+_A\Pi=\overline{I_\nabla}(a(m)+a'(m),\overline{c_m},\gamma(m)+\gamma'(m)).
  \end{split}
\end{equation*}
Hence $\overline{I_\nabla}$ is compatible with the addition over
$B\simeq C/C^A$ in $A\times_M C/C^A\times_MC$ and with the addition
$+_B$ in $R/\Pi$. Since $\overline{I_\nabla}$ is bijective and linear
over $A$ as well, $R/\Pi$ is a double vector bundle,
 of which $\overline{I_\nabla}$ is a decomposition.
\end{proof}

From now on the quotient $C/ C^A$ is identified with $B$ via the
factorisation of the morphism $\da_B\colon C\to B$ of vector
bundles. The linear splitting
$\overline{I_\nabla}\colon A\times_MB\times_MC\to R/\Pi$ defines
a linear horizontal lift $\overline{\sigma^\nabla_B}\colon \Gamma(B)\to\Gamma^l_A(R/\Pi)$: for all $b\in\Gamma(B)$ and $a_m\in A$
\begin{equation}
  \begin{split}
    \widecheck{b}(a_m):= \overline{\sigma^\nabla_B}(b)(a_m)&=\overline{I_\nabla}(a_m,b(m),0^C_m)=I_\nabla(a_m,c(m),0^C_m)+_A\Pi=\left(\widehat{\nabla_a}(c(m)), a(m)\right)+_A\Pi\\
    &=\widecheck{c}(a_m)+\Pi
  \end{split}
\end{equation}
for any $a\in\Gamma(A)$ such that $a(m)=a_m$ and $c\in\Gamma(C)$ such that $\da_Bc=b$.
The core sections of $R/\Pi\to A$ are $\gamma^\ddagger\in\Gamma_A^c(R/\Pi)$ defined by sections $\gamma\in\Gamma(C)$: for $a_m\in A$
\begin{equation}
  \begin{split}
    \gamma^\ddagger(a_m)&=\overline{I_\nabla}(a_m,0^B_m,\gamma(m))=I_\nabla(a_m,0^C,\gamma(m))+_A\Pi=\left(\gamma^\dagger(\rho_A(a_m)), a_m\right)+_A\Pi\\
    &=\gamma^\ddagger(a_m)+_A\Pi.
 \end{split}
\end{equation}
Note that the same notation is used for
$\gamma^\ddagger\in\Gamma_A^c(R/\Pi)$ and
$\gamma^\ddagger\in\Gamma_A^c(R)$. It is always clear from the context
which section is meant by the notation.

The linear splitting
$\overline{I_\nabla}\colon A\times_MB\times_MC\to R/\Pi$ defines as well
a linear horizontal lift $\overline{\sigma^\nabla_A}\colon \Gamma(A)\to\Gamma^l_B(R/\Pi)$: for all $a\in\Gamma(B)$ and $b_m\in B$
\begin{equation}\label{sec_a_red}
  \begin{split}
    \widehat{a}(b_m):= \overline{\sigma^\nabla_A}(a)(b_m)&=\overline{I_\nabla}(a(m),b_m,0^C_m)=I_\nabla(a(m),c_m,0^C_m)+_A\Pi
    =\left(\widehat{\nabla_a}(c_m), a(m)\right)+_A\Pi\\
    &=\widehat{a}(c_m)+_A\Pi
  \end{split}
\end{equation}
for any  $c_m\in C$ such that $\da_Bc_m=b_m$.
The core sections of $R/\Pi\to B$ are $\gamma^\times\in\Gamma_B^c(R/\Pi)$ defined by sections $\gamma\in\Gamma(C)$: for $b_m\in B$ and any $c_m\in C$ such that $\da_Bc_m=b_m$
\begin{equation}\label{sec_gamma_red}
  \begin{split}
    \gamma^\times(b_m)&=\overline{I_\nabla}(0^A_m,b_m,\gamma(m))=I_\nabla(0^A_m,c_m,\gamma(m))+_A\Pi=\left(\gamma^\uparrow(c_m), 0^A_m\right)+_A\Pi\\
    &=\gamma^\times(c_m)+_A\Pi.
 \end{split}
\end{equation}

The following corollary of Theorem \ref{main_thm_quotient} is
immediate.

\begin{cor}\label{cor_pi_mor_dvb}
  In the situation of Theorem \ref{main_thm_quotient}, the map
  $\pi\colon R\to R/\Pi$ is a double vector bundle morphism with sides
  $\id_A$ and $\da_B$, and with core $\id_C$.

  More precisely, choose a connection
  $\nabla\colon \Gamma(A)\times\Gamma(C)\to\Gamma(C)$ such that
  $\nabla_a\gamma=\nabla^A_a\gamma$ for all $a\in\Gamma(A)$ and
  $\gamma\in\Gamma(C^A)$. Then in the induced decompositions
  $I_\nabla\colon A\times_MC\times_MC\to R$ of $R$ and
  $\overline{I_\nabla}\colon A\times_M(C/C^A)\times_MC\to R/\Pi$ of
  $R/\Pi$, the map $\pi$ reads
  \begin{equation}\label{pi_in_dec}
    \begin{split}
    \overline{I_\nabla}\inv\circ \pi\circ I_\nabla\colon A\times_MC\times_MC&\to A\times_MB\times_MC, \qquad 
    (a,c,\gamma)\mapsto (a,\da_Bc,\gamma).
    \end{split}
    \end{equation}
 \end{cor}

  \bigskip

  Finally, consider a morphism of transitive core diagrams as in
  Figure \ref{fig:mcd} and consider the morphism
  $\Phi\colon TC\oplus_{TM}A\to TC'\oplus_{TM}A'$ as in Proposition
  \ref{mor_dvb}.  Then for all $a\in\Gamma(A)$ and
  $\gamma\in \Gamma(C^A)$,
  \begin{equation}
    \begin{split}
    T\phi_C\widehat{\nabla^A_a}(\gamma(m))&=T\phi_CT\gamma\rho_A(a_m)-\left.\frac{d}{dt}\right\an{t=0}\phi_C(\gamma(m))+t\phi_C((\nabla^A_a\gamma)(m))\\
    &=T(\phi_C(\gamma))\rho_{A'}(\phi_A(a_m))-\left.\frac{d}{dt}\right\an{t=0}\phi_C(\gamma)(m)+t\nabla^{A'}_{\phi_A(a)}(\phi_C\gamma)(m))\\
    &=\widehat{\nabla^{A'}_{\phi_A(a)}}(\phi_C(\gamma(m))),
    \end{split}
  \end{equation}
  since for a section $c\in\Gamma(C)$ such that $\da_Ac=a$, the equality $\da_{A'}(\phi_Cc)=\phi_A(\da_Ac)=\phi_A(a)$ follows and so
  \begin{equation}\label{omega_van_CA}
    \phi_C(\nabla^A_a\gamma)=\phi_C[c,\gamma]=[\phi_Cc,\phi_C\gamma]=\nabla^{A'}_{\phi_A(a)}(\phi_C\gamma).
  \end{equation}
  Hence, $\Phi(\Pi)\subseteq \Pi'$ and the vector bundle morphism $\Phi\colon TC\oplus_{TM}A\to TC'\oplus_{TM}A'$ over $\phi_A\colon A\to A'$ quotients
  to a vector bundle morphism
  \[\overline{\Phi}\colon \frac{TC\oplus_{TM}A}{\Pi}\to \frac{TC'\oplus_{TM}A'}{\Pi'}
  \]
  over $\phi_A\colon A\to A'$; such that
  $\pi\circ\Phi=\overline{\Phi}\circ\pi'$ where
  $\pi\colon TC\oplus_{TM} A\to \frac{TC\oplus_{TM}A}{\Pi}$ and
  $\pi'\colon TC'\oplus_{TM} A'\to \frac{TC'\oplus_{TM}A'}{\Pi'}$ are the projections.

  As before, the most effective way to
  prove that $\overline{\Phi}$ is a double vector bundle morphism from
  $R:=\frac{TC\oplus_{TM}A}{\Pi}$ to
  $R':=\frac{TC'\oplus_{TM}A'}{\Pi'}$ with side morphisms $\phi_A$ and
  $\phi_B$ and with core morphism $\phi_C$ is to show that the induced
  map in decompositions is a double vector bundle morphism.  Choose a
  linear $A$-connection $\nabla$ on $C$ extending
  $\nabla^A\colon \Gamma(A)\times\Gamma(C^A)\to\Gamma(C^A)$ and a
  linear $A'$-connection $\nabla'$ on $C'$ extending
  $\nabla^{A'}\colon \Gamma(A')\times\Gamma(C^{A'})\to\Gamma(C^{A'})$.
  Then \eqref{omega_van_CA} shows that the induced form
  $\omega_{\nabla,\nabla'}\in\Gamma(A^*\otimes C^*\otimes C')$ as in
  \eqref{omega_nabla_nabla'} vanishes on $C^A$. Hence it induces
  $\overline{\omega_{\nabla,\nabla'}}\in\Gamma(A^*\otimes B^*\otimes C')$:
  \[\overline{\omega_{\nabla,\nabla'}}(a,\da_Bc)=\omega_{\nabla,\nabla'}(a,c)
  \]
  for all $(a,c)\in A\times_M C$.
  A computation yields then
  \begin{equation}\label{1-form_red_Phi}
    \begin{split}
      \overline{\Phi}(\overline{I_{\nabla}}(a,\da_Bc,\gamma))&\,=\overline{\Phi}(I_\nabla(a,c,\gamma)+_A\Pi)
      =\, \Phi(I_\nabla(a,c,\gamma))+_{A'}\Pi'\\
      &\overset{\eqref{Phi_in_decs}}{=}I_{\nabla'}(\phi_A(a),\phi_C(c),\phi_C(\gamma)+\omega_{\nabla,\nabla'}(a,c))+_{A'}\Pi'\\
      &\,=\,\overline{I_{\nabla'}}(\phi_A(a),\da_{B'}\phi_C(c),\phi_C(\gamma)+\omega_{\nabla,\nabla'}(a,c))\\
      &\,=\,\overline{I_{\nabla'}}(\phi_A(a),\phi_B(\da_Bc),\phi_C(\gamma)+\overline{\omega_{\nabla,\nabla'}}(a,\da_Bc))
    \end{split}
  \end{equation}
  for $(a,c,\gamma)\in A\times_MC\times_MC$.
  Hence
  \[ \overline{I_{\nabla'}}\inv\circ \overline{\Phi}\circ\overline{I_{\nabla}}\colon A\times_MB\times_MC\to A'\times_MB'\times_MC', \quad
    (a,b,c)\mapsto (\phi_A(a),\phi_B(b),\phi_C(\gamma)+\overline{\omega_{\nabla,\nabla'}}(a,b)),
  \]
  which shows the following theorem.

  \begin{thm}\label{thm_mor_quot_dvb}
    Given a morphism of transitive core diagrams as in Figure
    \ref{fig:mcd}, the induced morphism
    $\Phi\colon TC\oplus_{TM}A\to TC'\oplus_{TM}A'$ of double vector
    bundles as in Proposition \ref{mor_dvb} quotients to a double
    vector bundle morphism
    \[\overline{\Phi}\colon
      \frac{TC\oplus_{TM}A}{\Pi}\to\frac{TC'\oplus_{TM}A'}{\Pi'},\]
    i.e.~with $\overline{\Phi}\circ \pi=\pi'\circ\Phi$. The morphism
    $\overline{\Phi}$ has side morphisms $\phi_A$ and $\phi_B$ and
    core morphism $\phi_C$.
\end{thm}

 \subsection{$\Pi$ as an ideal of $\ldbl\to A$}
 This section shows that $\Pi$ defined in \eqref{def_pi} is an ideal
 in $R\to A$: i.e.~it is totally intransitive and
 \[ [\Gamma_A(\Pi), \Gamma_A(\ldbl)]\subseteq \Gamma_A(\Pi).
 \]

 \begin{prop}\label{prop_ruth_0}
   Choose any connection
  $\nabla\colon \Gamma(A)\times\Gamma(C)\to\Gamma(C)$ such that
  $\nabla_a\gamma=\nabla^A_a\gamma$ for all $a\in\Gamma(A)$ and
  $\gamma\in\Gamma(C^A)$.
  Then \[\nabla^\da_\gamma a=0, \qquad 
 \nabla^\da_\gamma c=0 \quad \text{ and } \quad 
    R^\da_\nabla(\gamma,c)=0\]
    for all $a\in\Gamma(A)$, $c\in\Gamma(C)$ and $\gamma\in\Gamma(C^A)$.
  \end{prop}
  \begin{proof}Choose any section $a\in\Gamma(A)$ and any section
    $c\in\Gamma(C)$ with $\da c=a$. Then
    $\nabla^\da_\gamma a=[\da\gamma,\da
    c]+\da(\nabla_a\gamma)=-\da[c,\gamma]+\da(\nabla^{A}_a\gamma)
    =\da(\nabla_a^{A}\gamma)+\da(\nabla^{A}_a\gamma) =0$.  The second
    identity is immediate and the last one is easy and left to the
    reader.
    \end{proof}

 \begin{prop}\label{prop_intransitive}
 The anchor $\Theta_{A}\co \ldbl\to TA$ sends $\Pi$ to the zero section of $TA\to A$. 
 \end{prop}

 \begin{proof}
   Choose any connection
   $\nabla\colon \Gamma(A)\times\Gamma(C)\to\Gamma(C)$ such that
   $\nabla_a\gamma=\nabla^A_a\gamma$ for all $a\in\Gamma(A)$ and
   $\gamma\in\Gamma(C^A)$.  By Propositions \ref{lie_algebroid_R_C2}
   and \ref{prop_ruth_0}
   $\Theta_A(\Check\gamma)=\widehat{\nabla^\da_\gamma}=0\in\mx^l(A)$
   for all $\gamma\in\Gamma(C^A)$. This completes the proof with
   Corollary \ref{spanning_R_A}.
 \end{proof}

\begin{prop}\label{Pi_ideal}
 The subundle $\Pi$ of $R$ over $A$ is an ideal in $R\to A$.
 \end{prop}

 \begin{proof}
   Choose any connection
  $\nabla\colon \Gamma(A)\times\Gamma(C)\to\Gamma(C)$ such that
  $\nabla_a\gamma=\nabla^A_a\gamma$ for all $a\in\Gamma(A)$ and
  $\gamma\in\Gamma(C^A)$.
  By the Leibnitz identity and since $\Pi\to A$ is totally intransitive by Proposition \ref{prop_intransitive},
   it is sufficient to check that
   \[ [\Check\gamma,\Check c]\in\Gamma_A(\Pi) \quad \text{ and } \quad [\Check\gamma, c^\ddagger]\in\Gamma_A(\Pi)
   \]
   for all $\gamma\in\Gamma(C^A)$ and $c\in\Gamma(C)$. But this
   follows immediately from Proposition \ref{prop_ruth_0}, Proposition \ref{lie_algebroid_R_C2}
   and $[\gamma,c]\in\Gamma(C^A)$ for all $\gamma\in\Gamma(C^A)$ and $c\in\Gamma(C)$.
   \end{proof}

   By Proposition \ref{prop_ruth_0}, the connections
   $\nabla^{\da}\colon\Gamma(C)\times\Gamma(C)\to\Gamma(C)$ and
   $\nabla^{\da}\colon\Gamma(C)\times\Gamma(A)\to\Gamma(A)$ define
   connections
   $\nabla^{\rm red}\colon\Gamma(B)\times\Gamma(C)\to\Gamma(C)$ and
   $\nabla^{\rm red}\colon\Gamma(B)\times\Gamma(A)\to\Gamma(A)$:
\[\nabla^{\rm red}_{\da_Bc_1}c_2=\nabla^{\da}_{c_1}c_2\qquad \nabla^{\rm red}_{\da_Bc}a=\nabla^{\da}_{c}a
\]
for $c_1,c_2\in\Gamma(C)$ and $a\in\Gamma(A)$. 
By the same proposition, $R_\nabla^{\da}$ defines a new tensor
\[R^{\rm red}_\nabla\in\Omega^2(B,\operatorname{Hom}(A,C)), \qquad
R^{\rm red}_\nabla(\da_Bc,\da_Bc')=R_\nabla^{\da}(c,c'),
\]
for $c,c'\in\Gamma(C)$.

 \begin{cor}\label{reduced_VB_A}
   The vector bundle $R/\Pi\to A$ has a linear Lie algebroid structure
   over the Lie algebroid $B\to M$.  Choose any connection
   $\nabla\colon \Gamma(A)\times\Gamma(C)\to\Gamma(C)$ such that
   $\nabla_a\gamma=\nabla^A_a\gamma$ for all $a\in\Gamma(A)$ and
   $\gamma\in\Gamma(C^A)$. Then the Lie algebroid structure on
   $R/\Pi\to A$ can be described as follows:
   \[ \Theta_A(\Check{b})=\widehat{\nabla^{\rm red}_b}\in\mx^l(A) \quad \text{ and }\quad 
  \Theta_A(\gamma^\ddagger )=(\partial \gamma)^\uparrow\in\mx(A),\]
   \begin{equation}
     \label{eq:RA}
     [\Check b_1, \Check b_1]
     =\widecheck{[b_1,b_2]_B}-\widetilde{R^{\rm red}_\nabla(b_1,b_2)},\quad
 \left[\Check{b}, \gamma^\ddagger\right] = (\nabla^{\rm red}_b\gamma)^\ddagger,\quad
 \left[\gamma_1^\ddagger\, \gamma_2^\ddagger\right] = 0
 \end{equation}
 for $b,b_1,b_2\in\Gamma(B)$ and $\gamma,\gamma_1,\gamma_2\in\Gamma(C)$.

 In other words, via the splitting $\overline{I_\nabla}$ of $R/\Pi$
 given by $\nabla$, the Lie algebroid $R/\Pi\to A$ is described by the
 following representation up to homotopy of $B$:
\begin{equation}\label{ruth_B_D}
  \partial\colon C\to A,\quad \nabla^{\rm red}\colon\Gamma(B)\times\Gamma(C)\to\Gamma(C), \quad \nabla^{\rm red}\colon\Gamma(B)\times\Gamma(A)\to\Gamma(A),
  \end{equation}
and
\begin{equation}\label{ruth_B_D1}
  R^{\rm red}_\nabla\in\Omega^2(B,\operatorname{Hom}(A,C)).
\end{equation}
\end{cor}

\begin{proof}
  Since $\Pi\to A$ is an ideal in $R\to A$, the anchor $\Theta_A$ of
  $R/\Pi\to A$ is given by
  \[\Theta_A(\xi+\Pi)=\Theta_A(\xi)\]
  for all $\xi\in\Gamma_A(R)$, and the bracket is given by
  \[\left[\xi_1+_A\Pi, \xi_2+_A\Pi\right]=[\xi_1,\xi_2]+_A\Pi
  \]
  for $\xi_1,\xi_2\in\Gamma_A(R)$.
  This yields
  \[\Theta_A(\Check b)=\Theta_A(\Check c+_A\Pi)=\Theta_A(\Check c)=\widehat{\nabla^\da_c}=\widehat{\nabla_b^{\rm red}}\in\mx^l(A)
  \]
  for $b\in\Gamma(B)$ and a section $c\in\Gamma(C)$ such that
  $\da_Bc=b$, and
  \[\Theta_A(\gamma^\ddagger)=\Theta_A(\gamma^\ddagger+_A\Pi)=(\da\gamma)^\uparrow
  \]
  for all $\gamma\in\Gamma(C)$. Further computations show for $b,b_1,b_2\in\Gamma(B)$ and $\gamma,\gamma_1,\gamma_2\in\Gamma(C)$:
  \[ \left[\Check{b_1},\Check{b_2}\right]=\left[\Check{c_1}+_A\Pi,\Check{c_2}+_A\Pi\right]=\left(\Check{[c_1,c_2]}-\widetilde{R_\nabla^\da(c_1,c_2)}\right)+_A\Pi
    =\Check{[b_1,b_2]}-\widetilde{R_\nabla^\da(b_1,b_2)},
  \]
  where $c_1,c_2\in\Gamma(C)$ are such that $\da_B c_1=b_1$ und
  $\da_B c_2=b_2$,
  \[\left[\Check{b},\gamma^\ddagger\right]=(\nabla^\da_c\gamma)^\uparrow
  \]
  where $c\in\Gamma(C)$ is such that $\da_B c=b$,
  and $\left[\gamma_1^\ddagger,\gamma_2^\ddagger\right]=0$ is immediate.
  \end{proof}

 \subsection{The VB-algebroid $R/\Pi\to B$}
 Recall that the setting of this section is the transitive core diagram $\mathcal C$
 in Figure \ref{fig:calC}: $A$ and $B$ are Lie algebroids over a
 common base $M$ and $C$ is a Lie algebroid on $M$ together with
 morphisms $\da:=\da_A\co C\to A$ and $\da_B\co C\to B$, both of which are
 surjective and such that the kernels $C^A:= \ker(\da_B)$ and
 $C^B:=\ker(\da_A)$ commute in $C$.  Recall as well the representation
 $\nabla^A\co A\to\Der(C^A)$ defined in (\ref{eq:cdxm}):
 $\nabla^A_a\gamma=[c,\gamma]$ for any $c\in\Gamma(C)$ such that
 $\da c=a$.

 By its choice, the connection
 $\nabla\colon \Gamma(A)\times\Gamma(C)\to\Gamma(C)$ restricts as
 before to $\nabla^A\colon\Gamma(A)\times\Gamma(C^A)\to\Gamma(C^A)$.
 That is, it induces a connection
\[\nabla^B\colon \Gamma(A)\times\Gamma(B)\to\Gamma(B), \qquad
\nabla^B_a(\da_Bc)=\partial_B(\nabla_ac)
\]
for $c\in\Gamma(C)$ and $a\in\Gamma(A)$.  Since $\nabla^A$ is flat,
the curvature $R_\nabla$ of $\nabla$ vanishes on sections of $C^A$
and induces similarly a tensor
$\overline{R_\nabla}\in\Omega^2(A,\operatorname{Hom}(B,C))$,
\[\overline{R_\nabla}(a_1,a_2)(\da_Bc)=R_\nabla(a_1,a_2)c
\]
for $a_1,a_2\in\Gamma(A)$ and $c\in\Gamma(C)$. The following lemma is
easy to prove using \eqref{explicit_hat_D}.
\begin{lem}\label{lem_vfC_vfB}
  In the situation above,
  $\mx^l(C)\ni
  \widehat{\nabla_a}\sim_{\da_B}\widehat{\nabla_a^B}\in\mx^l(B)$ for
  all $a\in\Gamma(A)$ and
  $\mx^c(C)\ni\gamma^\uparrow\sim_{\da_B}(\da_B\gamma)^\uparrow\in\mx^c(B)$
  for all $\gamma\in\Gamma(C)$.
  \end{lem}

  Recall that the flat connection
\[ \nabla^i\colon\Gamma_C(F_C)\times\Gamma_C(R)\to\Gamma_C(R),
\]
is defined by
$\nabla^i\widehat a=0=\nabla^ic^\times$
for all $a\in\Gamma(A)$ and all $c\in\Gamma(C)$.

 \begin{prop}
   The triple $( F_C, 0, \nabla^{i})$ is an infinitesimal ideal system
   in $R\to C$. That is, $F_C\subseteq TC$ is an involutive subbundle,
    and $\nabla^{i}$ a flat 
   $F_C$-connection on $R\to C$ with the following properties:
 \begin{enumerate}
 \item If $\chi_1,\chi_2\in\Gamma_C(R)$ are $\nabla^{i}$-parallel, then $[\chi_1,\chi_2]_C$ is also parallel.
 \item If $\chi\in\Gamma_C(R)$ is $\nabla^{i}$-parallel, then
   $[\Theta_C(\chi), X]\in\Gamma(F_C)$ for all $X\in\Gamma(F_C)$. That
   is, $\Theta_C(\chi)$ is $\nabla^{F_C}$-parallel, where
   $\nabla^{F_C}$ is the Bott connection associated to
   $F_C\subseteq TC$.
 \end{enumerate}
\end{prop}

\begin{proof}
  By the properties of an infinitesimal ideal system, (2) follows from
  (1).  Since sections $\widehat{a}\in\Gamma_C^l(R)$ and
  $\gamma^\times$ for $a\in\Gamma(A)$ and $\gamma\in\Gamma(C)$ are
  $\nabla^i$-parallel and span $\Gamma_C(R)$ as a
  $C^\infty(C)$-module, it suffices to show (1) for $\chi_1$ and
  $\chi_2$ of the type $F\cdot \widehat{a}$ and $F\cdot \gamma^\times$
  with $F\in C^\infty(C)^{F_C}$ and $a\in\Gamma(A)$,
  $\gamma\in\Gamma(C)$.  Since $C/F_C=C/C^A=B$, the space
  $C^\infty(C)^{F_C}$ equals $\partial_B^*C^\infty(B)$.

  Choose $F_1,F_2\in C^\infty(B)$, $a,a_1,a_2\in\Gamma(A)$ and 
  $\gamma,\gamma_1,\gamma_2\in\Gamma(C)$. Then using  Lemma
  \ref{lem_vfC_vfB},
  \begin{equation*}
    \begin{split}
      \left[\da_B^*F_1\cdot\widehat{a_1},
        \da_B^*F_2\cdot\widehat{a_2}\right]&=\da_B^*(F_1F_2)\cdot\left[\widehat{a_1},\widehat{a_2}\right]
      +\da_B^*F_1\cdot\ldr{\widehat{\nabla_{a_1}}}(\da_B^*F_2)\cdot\widehat{a_2}-\da_B^*F_2\cdot\ldr{\widehat{\nabla_{a_2}}}(\da_B^*F_1)\cdot\widehat{a_1}\\
      &=\da_B^*(F_1F_2)\cdot\left(\widehat{[a_1,a_2]}-\widetilde{R_\nabla(a_1,a_2)}\right)\\
      &\qquad
      +\da_B^*\left(F_1\cdot\ldr{\widehat{\nabla^B_{a_1}}}F_2\right)\cdot\widehat{a_2}-\da_B^*\left(F_2\cdot\ldr{\widehat{\nabla^B_{a_2}}}F_1\right)\cdot\widehat{a_1}
    \end{split}
  \end{equation*}
  shows that
  $\left[\da_B^*F_1\cdot\widehat{a_1},
    \da_B^*F_2\cdot\widehat{a_2}\right]$ is a $\nabla^i$-flat section
  of $R\to C$ if $\widetilde{R_\nabla(a_1,a_2)}$ is
  $\nabla^i$-flat. Since $\nabla^A$ is flat,
  $R_\nabla(a_1,a_2)\in\Gamma((C^A)^\circ\otimes C)=\Gamma(B^*\otimes
  C)$. Hence it equals without loss of generality
  $(\da_B^t\nu)\otimes c$ with $\nu\in\Gamma(B)$ and $c\in\Gamma(C)$,
  and consequently
  $\widetilde{R_\nabla(a_1,a_2)}=\ell_{\da_B^t\nu}\cdot\gamma^\times=\da_B^*\ell_\nu\cdot\gamma^\times$
  is a $\nabla^i$-flat section of $R\to C$.

  Similarly,
   \begin{equation*}
    \begin{split}
      \left[\da_B^*F_1\cdot\widehat{a},
        \da_B^*F_2\cdot\gamma^\times\right]&=\da_B^*(F_1F_2)\cdot\left[\widehat{a},\gamma^\times\right]
      +\da_B^*F_1\cdot\ldr{\widehat{\nabla_{a}}}(\da_B^*F_2)\cdot\gamma^\times-\da_B^*F_2\cdot\ldr{\gamma^\uparrow}(\da_B^*F_1)\cdot\widehat{a}\\
      &=\da_B^*(F_1F_2)\cdot(\nabla_a\gamma)^\times
      +\da_B^*\left(F_1\cdot\ldr{\widehat{\nabla^B_{a_1}}}F_2\right)\cdot\gamma^\times-\da_B^*\left(F_2\cdot\ldr{(\da_B\gamma)^\uparrow}F_1\right)\cdot\widehat{a_1}
    \end{split}
  \end{equation*}
  and
   \begin{equation*}
    \begin{split}
      \left[\da_B^*F_1\cdot\gamma_1^\times, \da_B^*F_2\cdot\gamma_2^\times\right]&=\da_B^*(F_1F_2)\cdot\left[\gamma_1^\times, \gamma_2^\times\right]
      +\da_B^*F_1\cdot\ldr{\gamma_1^\uparrow}(\da_B^*F_2)\cdot\gamma_2^\times-\da_B^*F_2\cdot\ldr{\gamma_2^\uparrow}(\da_B^*F_1)\cdot\gamma_1^\times\\
      &=\da_B^*\left(F_1\cdot\ldr{(\da_B\gamma_1)^\uparrow}F_2\right)\cdot\gamma_2^\times-\da_B^*\left(F_2\cdot\ldr{(\da_B\gamma_2)^\uparrow}F_1\right)\cdot\gamma_1^\times
    \end{split}
  \end{equation*}
  show that
  $\left[\da_B^*F_1\cdot\widehat{a},
    \da_B^*F_2\cdot\gamma^\times\right]$ and
  $\left[\da_B^*F_1\cdot\gamma_1^\times,
    \da_B^*F_2\cdot\gamma_2^\times\right]$ are $\nabla^i$-flat
  sections of $R\to C$.
\end{proof}

The proof above shows as well the following lemma.
\begin{lem}\label{proj_core_linear}
  For $\phi\in\Gamma\left( (C^A)^\circ\otimes C\right)$, the
  corresponding core-linear section $\widetilde{\phi}\in\Gamma^l_C(R)$
  is $\nabla^i$-flat. It projects to
  $\widetilde{\bar\phi}\in\Gamma_B^l(R/\Pi)$, where
  $\bar\phi\in\Gamma(B^*\otimes C)$ is defined by
  $\bar\phi(\partial_Bc)=\phi(c)$ for all $c\in C$.
  
  \end{lem}

 \begin{cor}\label{quotient_IM}
   There is a linear Lie algebroid structure on $R/\Pi\simeq  R/\nabla^{i} \to C/F_C\simeq B$ such that
   \begin{equation}\label{pic_fibration}
  \begin{xy}
    \xymatrix{
      R \ar[r]^{\pi_{\nabla^i}}\ar[d]_{\pi_C}&   R/\nabla^i\ar[d]^{q}\\
       C\ar[r]_{\partial_B}                   &  B\\
    }
  \end{xy}
\end{equation}
is a fibration of VB-Lie algebroids.
\end{cor}

\begin{proof}
  By Theorem \ref{main_thm_quotient}, $R/\Pi$ has a double vector
  bundle structure with sides $A$ and $B$ and with core $C$. Since
  $R/\Pi\simeq R/\nabla^i\to C/F_C\simeq B$, and $(F_C,0,\nabla^i)$ is
  an infinitesimal ideal system in $R\to C$, the quotient $R/\Pi$ has a
  Lie algebroid structure such that \eqref{pic_fibration} is a
  fibration of Lie algebroids, see Theorem \ref{red_lie_alg},(2). It
  remains to show that the Lie algebroid structure on $R/\Pi\to B$ is
  linear.

  Consider a linear $A$-connection $\nabla$ on $C$ that preserves
  $C^A$. Then the sections of the type $\widehat{a}$ and
  $\gamma^\times$ for $a\in\Gamma(A)$ and $\gamma\in \Gamma(C)$ are
  $\nabla^i$ invariant and project to the sections $\widehat{a}$ and
  $\gamma^\times$ of $R/\Pi\to B$ as in \eqref{sec_a_red} and
  \eqref{sec_gamma_red}:
\[ \pi_{\nabla^i}\circ \widehat{a}=\widehat{a}\circ\da_B\qquad
    \text{ and }\qquad \pi_{\nabla^i}\circ \gamma^\times=\gamma^\times\circ\da_B
\]

  The anchor $\Theta_B\colon R/\Pi\to TB$ is defined by
  \[ \Theta_B(\widehat{a})=\widehat{\nabla_a^B}\in\mx^l(B) \qquad
    \text{ and }\qquad
    \Theta_B(\gamma^\times)=(\da_B\gamma)^\uparrow\in\mx^c(B),
  \]
  since by Lemma \ref{lem_vfC_vfB},
  $\Theta_C(\widehat{a})=\widehat{\nabla_a}\in\mx(C)$ is $\da_B$
  related to $\widehat{\nabla_a^B}\in\mx(B)$ and
  $\Theta_C(\gamma^\times)=\gamma^\uparrow\in\mx(C)$ is
  $\da_B$-related to $(\da_B\gamma)^\uparrow\in\mx(B)$.

  The Lie algebroid bracket of core sections of $R/\Pi\to B$ is
  defined by
    \[ \left[\gamma_1^\times, \gamma_2^\times\right]_{R/\Pi}=0
    \]
    for $\gamma_1,\gamma_2\in\Gamma(C)$, since the bracket
    $\left[\gamma_1^\times, \gamma_2^\times\right]$ of
    $\gamma_1^\times,\gamma_2^\times\in\Gamma_C(R)$ vanishes and
    projects hence to the zero section of $R/\Pi\to B$.
    Similarly,
    \[ \left[\widehat{a}, \gamma^\times\right]_{R/\Pi}=(\nabla_a\gamma)^\times
    \]
    for $a\in\Gamma(A)$ and $\gamma\in\Gamma(C)$, since the bracket
    $\left[\widehat{a}, \gamma^\times\right]$ of $\widehat{a}$ and
    $\gamma^\times\in\Gamma_C(R)$ equals $(\nabla_a\gamma)^\times$ and
    projects hence to $(\nabla_a\gamma)^\times\in\Gamma_B^c(R/\Pi)$.
    Finally, as in Lemma \ref{proj_core_linear}, for
    $a_1,a_2\in\Gamma(A)$ the curvature $R_\nabla(a_1,a_2)$ gives a
    $\nabla^i$-flat core-linear section
    $\widetilde{R_\nabla(a_1,a_2)}\in\Gamma^l_C(R)$, which projects
    under $\pi_{\nabla^i}$ and $\partial_B$ to
    $\widetilde{\overline{R_\nabla}(a_1,a_2)}\in\Gamma^l_B(R/\Pi)$.
    As a consequence, the Lie bracket
    $\left[\widehat{a_1},\widehat{a_2}\right]=\widehat{[a_1,a_2]}-\widetilde{R_\nabla(a_1,a_2)}\in\Gamma^l_C(R)$
    of $\widehat{a_1}$ and $\widehat{a_2}\in\Gamma^l_C(R)$ projects to
    \[\left[\widehat{a_1},\widehat{a_2}\right]_{R/\Pi}:=\widehat{[a_1,a_2]}-\widetilde{\overline{R_\nabla}(a_1,a_2)}\in\Gamma^l_B(R/\Pi). \qedhere\]
    \end{proof}

The proof of Corollary \ref{quotient_IM} shows as well the following Proposition.
  \begin{prop}
    Let $R/\Pi$ be constructed as above and consider as before an
    $A$-connection $\nabla\colon\Gamma(A)\times\Gamma(C)\to \Gamma(C)$
    on $C$ that extends $\nabla^A$.   Then the anchor
    $\ldbl\to TB$ is defined by 
    $\Theta_B\circ \widehat{a}=\Hat{\nabla^B_a}\in \mx^l(B)$ and
    $\Theta_B\circ \gamma^\times=(\da_B\gamma)^\uparrow\in\mx(B)$ for
    $a\in\Gamma(A)$ and $\gamma\in\Gamma(C)$.  Furthermore, the Lie
    algebroid bracket on $\ldbl/\Pi\to B$ is given by:
 \begin{equation}
 \label{eq:RC}
 \left[\widehat{a_1},
   \widehat{a_2}\right] =
 \widehat{[a_1,a_2]}-\widetilde{\overline{R_\nabla}(a_1,a_2)},\quad
 \left[\widehat{a}, \gamma^\times\right] = (\nabla_a\gamma)^\times,\quad
 \left[\gamma_1^\times\,,\gamma_2^\times\right] = 0
 \end{equation}
 for $a,a_1,a_2\in\Gamma(A)$ and  $\gamma,\gamma_1,\gamma_2\in\Gamma(C)$.

 In other words, via the splitting $\overline{I_\nabla}$ of $R/\Pi$
 given by $\nabla$, the Lie algebroid $R/\Pi\to B$ is described by the
 following representation up to homotopy of $A$:
\begin{equation}\label{ruth_A_D}
  \da_B\colon C\to B,\quad \nabla\colon\Gamma(A)\times\Gamma(C)\to\Gamma(C),  \quad\nabla^B\colon\Gamma(A)\times\Gamma(B)\to\Gamma(B),
\end{equation}
and
\begin{equation}\label{ruth_A_D1}
  \overline{R_\nabla}\in\Omega^2(A,\operatorname{Hom}(B,C)).
\end{equation}
    \end{prop}

 \begin{rem}
   The sections above show that the quotient $R/\Pi$ can be defined
   as the ``double quotient'' of the double Lie algebroid $R$ by the
   two IM-foliations
 \[(0,\Pi,0)\qquad \text{ and }\qquad (F_C,0,\nabla^{i})
 \]
 in $R\to A$ and $R\to C$, respectively.
 \end{rem}

\subsection{The quotient $R/\Pi$ as a double Lie algebroid.}

In this section, the quotient double vector bundle is written
$R/\Pi=R/\nabla^i=:\tilde D$ for simplicity. The next theorem shows
that the two `reduced' VB-algebroid structures found in Corollary
\ref{reduced_VB_A} and Corollary \ref{quotient_IM} on the sides of
$\tilde D$ are still compatible, so that $\tilde D=R/\Pi$ is again a
double Lie algebroid.

The setting of this section is the same as in the rest of Section
\ref{sec:quotient}. Recall from Corollary \ref{reduced_VB_A} that
$\tilde D\to A$ inherits a reduced Lie algebroid as the quotient of
$R\to A$ by the ideal $\Pi\to A$. By Corollary \ref{quotient_IM},
$\tilde D\to B$ inherits a Lie algebroid structure, as the quotient of
$R\to C$ by the infinitesimal ideal system $(F_C,0,\nabla^i)$.
\begin{thm}\label{tildeD_dla}
  With these two VB-Lie algebroid structures, the quotient
  $\tilde D=R/\Pi$ has the structure of a transitive double Lie
  algebroid over the sides $A$ and $B$, with core $C$ and (transitive)
  core diagram $\mathcal C$ in Figure \ref{fig:calC}.
\end{thm}

\begin{proof}
  Let as before $\nabla\colon \Gamma(A)\times\Gamma(C)\to\Gamma(C)$ be
  a linear connection extending the representation $\nabla^A$ of $A$
  on $C^A$. In the induced linear splitting
  $\overline{I_\nabla}\colon A\times_MB\to \tilde D$ of $\tilde D$,
  the two VB-algebroids $\tilde D\to A$ and $\tilde D\to B$ are
  represented by the two representations up to homotopy in
  \eqref{ruth_B_D}, \eqref{ruth_B_D1}, and in \eqref{ruth_A_D},
  \eqref{ruth_A_D1}, respectively.  It is not difficult to check that
  these two representations up to homotopy are matched (see Appendix
  \ref{proof_of_R_lba}). Hence, $\tilde D$ is a double Lie algebroid
  with sides $A$ and $B$ and with core $C$.

  \medskip By \eqref{bracket_on_C}, the Lie algebroid structure
  induced on the core $C$ by the double Lie algebroid $\tilde D$ is
  given by the anchor $\rho_A\circ\da=\rho_B\circ\da_B$, which is
  $\rho_C$, and by the bracket
  \[ \nabla_{\da_A c_1}c_2-\nabla^{\rm red}_{\da_B c_2}c_1=\nabla_{\da_A c_1}c_2-\nabla^\da_{c_2}c_1=-[c_2,c_1]=[c_1,c_2]
  \]
  for all $c_1,c_2\in\Gamma(C)$. Hence, the Lie algebroid $C$ in the
  core diagram $\mathcal C$ is the Lie algebroid $C$ as the core of
  the double Lie algebroid $\tilde D$.  By \eqref{ruth_B_D}, the
  core-anchor of $\tilde D\to A$ is $\da=\da_A\colon C\to A$, and by
  \eqref{ruth_A_D}, the core-anchor of $\tilde D\to B$ is
  $\da_B\colon C\to B$. Therefore, the core diagram $\mathcal C$ is
  the core-diagram of $\tilde D$.
\end{proof}

\begin{prop}
  In the situation of Theorem \ref{tildeD_dla}, the morphism
  $\pi\colon R\to \tilde D$ of double vector bundles is a morphism of
  double Lie algebroids.
\end{prop}

\begin{proof}
  By Corollary \ref{cor_pi_mor_dvb}, a choice of linear connection
  $\nabla\colon\Gamma(A)\times\Gamma(C)\to\Gamma(C)$ that extends
  $\nabla^A\colon \Gamma(A)\times\Gamma(C^A)\to \Gamma(C^A)$ yields
  linear splittings $\Sigma_\nabla\colon A\times_MC\to R$ and
  $\overline\Sigma_\nabla\colon A\times_MB\to \tilde D$ such that
  $\pi(\Sigma_\nabla(a,c))=\overline\Sigma_\nabla(a,\da_Bc)$ for all
  $(a,b)\in A\times_M B$.

  Therefore \eqref{comp1}, \eqref{comp2} and \eqref{comp3} need to be
  checked for \begin{enumerate}
  \item the triple $(\id_C,\da_B,0)$ from the $2$-representation
  $(\id_C\colon C[0]\to C[1], \nabla, \nabla, R_\nabla)$ of $A$ as in 
  \eqref{ruth_A} to the $2$-representation
  $(\da_B\colon C[0]\to B[1], \nabla, \nabla^B, \overline{R_\nabla})$ of $A$
  as in \eqref{ruth_A_D} and \eqref{ruth_A_D1}.
  \item
    the triple $(\id_C,\id_A,0)$ from the $2$-representation
  $(\da_A\colon C[0]\to A[1], \nabla^\da, \nabla^\da, R_\nabla^\da)$ of $C$ as in 
  \eqref{ruth_C} and \eqref{ruth_C_1} to the $2$-representation
  $\da_B^!(\da_A\colon C[0]\to A[1], \nabla^{\rm red}, \nabla^{\rm red}, R_\nabla^{\rm red})$ of $C$
  as in \eqref{ruth_B_D} and \eqref{ruth_B_D1}.
  \end{enumerate}
  In the first case, \eqref{comp1} is trivial and \eqref{comp2} and
  \eqref{comp3} follow immediately from the definitions of $\nabla^B$
  and $\overline{R_\nabla}$. In the second case,
  $\da_B^!(\da_A\colon C[0]\to A[1], \nabla^{\rm red}, \nabla^{\rm
    red}, R_\nabla^{\rm red})=(\da_A\colon C[0]\to A[1], \nabla^\da,
  \nabla^\da, R_\nabla^\da)$ by the definitions of $\nabla^{\rm red}$
  and $R^{\rm red}_\nabla$, so the three conditions are immediate.
  \end{proof}

\subsection{Equivalence of transitive double Lie algebroids with transitive core diagrams}\label{eq_cat_la}
Let $\CD(M)$ be the category of transitive core diagrams over $M$, and
let $\DLA(M)$ be the category of transitive double Lie algebroids over
$M$. As discussed in Section \ref{func_trdLA_codiag}, considering the
core diagram of a transitive double Lie algebroid with double base $M$
defines a functor \[\mathcal C\colon \DLA(M)\to \CD(M).\]

This section shows that the construction of a transitive double Lie
algebroid $\tilde D=R/\Pi$ from a given transitive core diagram defines a functor
\[  \D\colon \CD(M)\to\DLA(M),
  \]
  and proves the following theorem.

  \begin{thm}\label{eq_cat}
    Let $M$ be a smooth manifold. The two functors
    $\mathcal C\colon\DLA(M)\to \CD(M)$ and
    $\D\colon \CD(M)\to\DLA(M)$ establish an equivalence between the
    category of transitive core diagrams over $M$ and the one of
    transitive double Lie algebroids with double base $M$.
    \end{thm}

  \medskip

  First consider a morphism of transitive core diagrams as in Figure
  \ref{fig:mcd} and the induced vector bundle morphism
    \[\overline{\Phi}\colon
      \frac{TC\oplus_{TM}A}{\Pi}\to\frac{TC'\oplus_{TM}A'}{\Pi'},\] as
    in Theorem \ref{thm_mor_quot_dvb}, with side morphisms $\phi_A$
    and $\phi_B$ and core morphism $\phi_C$.
Choose a linear $A$-connection $\nabla$ on $C$ extending
    $\nabla^A$, and a linear $A'$-connection on $C'$ extending
    $\nabla^{A'}$.  Then by \eqref{1-form_red_Phi},
    \[
      \overline{\Phi}(\overline{I_{\nabla}}(a,b,\gamma))=\overline{I_{\nabla'}}(\phi_A(a),\phi_B(b),\phi_C(\gamma)+\overline{\omega_{\nabla,\nabla'}}(a,b))\]
    for $(a,b,c)\in A\times_MB\times_MC$. In order to check that
    $\overline{\Phi}$ is a double Lie algebroid morphim, 
     \eqref{comp1}, \eqref{comp2} and \eqref{comp3} need to be
  checked for \begin{enumerate}
  \item the triple $(\phi_C,\phi_B,\overline{\omega_{\nabla,\nabla'}})$ from the $2$-representation
  $(\da_B\colon C[0]\to B[1], \nabla, \nabla^B, \overline{R_\nabla})$ of $A$ as in 
   as in \eqref{ruth_A_D} and \eqref{ruth_A_D1} to the $2$-representation
  $\phi_A^*(\da_{B'}\colon C'[0]\to B'[1], \nabla', {\nabla'}^{B'}, \overline{R_{\nabla'}})$ of $A$.
  \item
    the triple $(\phi_C,\phi_A,\overline{\omega_{\nabla,\nabla'}})$ from the $2$-representation
  $(\da_A\colon C[0]\to A[1], \nabla^{\rm red}, \nabla^{\rm red}, R_\nabla^{\rm red})$ of $B$ as in 
  \eqref{ruth_B_D} and \eqref{ruth_B_D1} to the $2$-representation
  $\phi_B^*(\da_{A'} \colon C'[0]\to A'[1], {\nabla'}^{\rm red}, {\nabla'} ^{\rm red}, R_{\nabla'}^{\rm red})$ of $B$.
\end{enumerate}
This is done in Section \ref{proof_of_dla_mor}, yielding the following theorem.

\begin{thm}\label{dla_mor}
  Given a morphism of transitive core diagrams as in Figure
  \ref{fig:mcd}, then the induced double vector bundle morphism
    \[\overline{\Phi}\colon \frac{TC\oplus_{TM}A}{\Pi}\to\frac{TC'\oplus_{TM}A'}{\Pi'} ,\] as
    in Theorem \ref{thm_mor_quot_dvb} is a morphism of double Lie
    algebroids with side morphisms $\phi_A$ and $\phi_B$ and core
    morphism $\phi_C$. That is, the induced morphism of core diagrams
    is again $(\phi_A,\phi_B,\phi_C)$.
  \end{thm}

  Hence this paper has constructed the functor
  $\D\colon \CD(M)\to\DLA(M)$, sending a transitive core diagram
  $\mathcal C$ over $M$ as in Figure \ref{fig:cd} to the transitive
  double Lie algebroid $\frac{TC\oplus_{TM}A}{\Pi}$, and a morphism of
  transitive core diagrams as in Figure \ref{fig:mcd} to the
  corresponding double Lie algebroid morphism $\overline\Phi$ as in
  the last theorem.  By Theorem \ref{tildeD_dla}, the composition
  $\mathcal C\circ \D\colon \CD(M)\to \CD(M)$ is the identity functor.
  The remainder of this section proves that the composition
  $\D\circ\, \mathcal C\colon \DLA(M)\to \DLA(M)$ is a natural
  isomorphism, hence completing the proof of Theorem
  \ref{eq_cat}. Note that by construction, $\D\circ\, \mathcal C$
  sends a transitive double Lie algebroid to a transitive double Lie
  algebroid with the same core, see Theorem \ref{tildeD_dla}.
  
\medskip

The needed natural isomorphism follows from the following theorem.
\begin{thm}\label{iso_dvbs}
  \begin{enumerate}
  \item Let $(D_1;A,B,M)$ and $(D_2;A,B;M)$ be two transitive double
    Lie algebroids with the same core diagram $\mathcal C$, as in
    Figure \ref{fig:calC}.  Then there is a canonical isomorphism
    $\Phi_{D_1,D_2}\colon D_1\to D_2$ of double Lie algebroids, which
    is the identity on the sides and on the core.
  \item Let $(D_1;A,B,M)$ and $(D_2;A,B;M)$ be two transitive double
    Lie algebroids with the same core diagram $\mathcal C$, and let
    $(D'_1;A',B',M)$ and $(D_2';A',B';M)$ be two transitive double Lie
    algebroids with the same core diagram $\mathcal C'$. Consider two
    double Lie algebroid morphisms $\Phi_1\colon D_1\to D_1'$ and
    $\Phi_2\colon D_2\to D_2'$ that both induce the same morphism of
    core diagrams $\mathcal C\to\mathcal C'$. Then
  \[ \Phi_2\circ\Phi_{D_1,D_2}=\Phi_{D_1',D_2'}\circ\Phi_1.
    \]
  \end{enumerate}
  
\end{thm}

\begin{proof}
 \begin{enumerate}
    \item Choose a linear splitting $\Sigma^1\colon A\times_MB\to D_1$ of
  $D_1$ and $\Sigma^2\colon A\times_MB\to D_2$ of $D_2$. These two linear splittings induce
  four representations up to homotopy:
  \begin{enumerate}
    \item For $i=1,2$, the Lie algebroid $D_i\to A$ is described by the representation up to homotopy 
\begin{equation*}
  \partial_A\colon C\to A,\quad \nabla^{Bi}\colon\Gamma(B)\times\Gamma(C)\to\Gamma(C), \quad \nabla^{Bi}\colon\Gamma(B)\times\Gamma(A)\to\Gamma(A),
  \end{equation*}
and
\begin{equation*}
  R^{Bi}_\nabla\in\Omega^2(B,\operatorname{Hom}(A,C))
\end{equation*}
of $B\to M$ on $A[0]\oplus C[1]$.
\item For $i=1,2$, the Lie algebroid $D_i\to B$ is described by the representation up to homotopy 
\begin{equation*}
  \partial_B\colon C\to B,\quad \nabla^{i}\colon\Gamma(A)\times\Gamma(C)\to\Gamma(C), \quad \nabla^{Ai}\colon\Gamma(A)\times\Gamma(B)\to\Gamma(B),
  \end{equation*}
and
\begin{equation*}
  R^{Ai}_\nabla\in\Omega^2(A,\operatorname{Hom}(B,C))
\end{equation*}
of $A\to M$ on $B[0]\oplus C[1]$.
\end{enumerate}
Choose $\gamma\in\Gamma(C^A)=\Gamma(\ker\da_B)$ and $a\in\Gamma(A)$. Choose further a section $c\in\Gamma(C)$ with $\da c=a$.
By \eqref{bracket_on_C},
\[ \nabla^i_a\gamma=\nabla^i_{\da c}\gamma=[c,\gamma]+\nabla^{Bi}_{\da_B\gamma}c=[c,\gamma]
\]
for $i=1,2$. Hence,
$\nabla^1-\nabla^2\in\Omega^1(A, (C^A)^\circ\otimes
C)=\Gamma(A^*\otimes B^*\otimes C)$; i.e.~there exists a (unique) section
$\phi\in\Gamma(A^*\otimes B^*\otimes C)$ so dass
$\nabla^1_ac-\nabla^2_ac=\phi(a,\da_Bc)$ for all $a\in\Gamma(A)$ and
$c\in\Gamma(C)$. Set $\tilde\Sigma^2\colon A\times_MB\to D_2$,
\[ (a,b)\mapsto \Sigma^2(a,b)+_A(0^{D_2}_a+_B\overline{\phi(a,b)}).
\]
Then, by standard computations (see
e.g.~\cite[Remark 2.12]{GrJoMaMe18}), the connection
$\tilde\nabla^2\colon \Gamma(A)\times\Gamma(C)\to\Gamma(C)$ induced by the
VB-algebroid $D_2\to B$ and the splitting $\tilde \Sigma^2$ is given by
\[ \tilde\nabla^2_ac=\nabla^2_ac+\varphi(a,\da_Bc)=\nabla^2_ac+\nabla^1_ac-\nabla^2_ac=\nabla^1_ac.
  \]
  This shows that $\Sigma^1$ and $\Sigma^2$ can be chosen such that
  $\nabla:=\nabla^1=\nabla^2\colon
  \Gamma(A)\times\Gamma(C)\to\Gamma(C)$.  Consider then the linear
  decompositions $I_1\colon A\times_MB\times_MC\to D_1$ and
  $I_2\colon A\times_MB\times_MC\to D_2$ that are equivalent to
  $\Sigma^1$ and $\Sigma^2$, respectively, and the isomorphism
  $\Phi_{D_1,D_2}:=I_2\circ I_1\inv\colon D_1\to D_2$ of double vector
  bundles.  Choose two different linear splittings
  $\tilde \Sigma^1\colon A\times_M B\to D_1$ and $\tilde\Sigma^2\colon A\times_M B\to D_2$ of $D_1$ and $D_2$ such that
  the induced $A$-connections $\tilde\nabla^1$ and $\tilde \nabla^2$ on $C$ are equal: $\tilde \nabla^1=\tilde \nabla^2=:\tilde \nabla$.
  Then there exist forms $\omega_1,\omega_2\in\Gamma(A^*\otimes B^*\otimes C)$ such that
  $\tilde\Sigma^i(a,b)=\Sigma^i(a,b)+_B(0^{D_2}_b+_A\overline{\omega_i(a,b)})$
for $i=1,2$ and all $(a,b)\in A\times_M B$.
By \cite[Remark 2.12]{GrJoMaMe18},
\[ \nabla_ac+\omega_1(a,\da_Bc)=\tilde \nabla_ac=\nabla_ac+\omega_2(a,\da_Bc)
\]
for all $a\in\Gamma(A)$ and $c\in\Gamma(C)$. Since $\da_B$ is
surjective, this shows $\omega_1=\omega_2$. The equality
$\tilde I_2\circ\tilde I_1\inv=I_2\circ I_1\inv$ follows easily for
the decompositions $\tilde I_1$ and $\tilde I_2$ of $D_1$ and $D_2$
that are equivalent to $\tilde\Sigma^1$ and $\tilde\Sigma^2$.  Hence
$\Phi_{D_1,D_2}$ does not depend on the choice of the compatible
linear splittings of $D_1$ and $D_2$.

  \medskip

  It remains to show that $\Phi_{D_1,D_2}$ is an isomorphism of double
  Lie algebroids, i.e.~that the representations up to homotopy in (a)
  coincide and that the ones in (b) coincide, if
  $\nabla^1=\nabla^2=:\nabla$. Choose $a\in\Gamma(A)$, $b\in\Gamma(B)$
  and a section $c\in\Gamma(C)$ such that $\da_Bc=b$.  Then
\[
  \nabla^{A1}_ab=\nabla_a^{A1}\partial_Bc=\partial_B\nabla_ac=\nabla_a^{A2}\partial_Bc=\nabla_a^{A2}b
\]
by (R2) in the definition of a 2-representation, and 
\[\nabla^{B1}_ba=\nabla^{B1}_{\da_Bc}a=[\partial_Ac,a]+\partial_A(\nabla_ac)
=\nabla^{B2}_{\da_Bc}a
=\nabla^{B2}_ba
\]
by (M2).

Similarly, for  $b=\da_Bc\in\Gamma(B)$ as above  and $c'\in\Gamma(C)$:
\[\nabla^{B1}_bc'=\nabla^{B1}_{\da_Bc}c'=[c,c']+\nabla_{\partial_Ac'}c=\nabla^{B2}_{\partial_Bc}c'=\nabla^{B2}_{b}c'
\]
by \eqref{bracket_on_C} and (M1).
Furthermore, for $a_1,a_2\in\Gamma(A)$: 
\begin{align*}
R^{A1}(a_1,a_2)b&=R^{A1}(a_1,a_2)(\partial_Bc)
=R_{\nabla}(a_1,a_2)c\\
&=R^{A2}(a_1,a_2)(\partial_Bc)=R^{A2}(a_1,a_2)b
\end{align*}
by (R3). It remains to prove that $R^{B1}=R^{B2}$. This is done as follows using (M4).
Choose $a\in\Gamma(A)$, $b_1,b_2\in\Gamma(B)$ and $c_2\in\Gamma(C)$ such that $\partial_Bc_2=b_2$.
Then
\begin{align*}
R^{B1}(b_1,b_2)(a)&=R^{B1}(b_1,\partial_Bc_2)a\\
&=R^{A1}(a,\partial_Ac_2)b_1+\nabla^{B}_{b_1}\nabla_ac_2-\nabla_a\nabla^{B}_{b_1}c_2-\nabla_{\nabla_{b_1}^{B}a}c_2+\nabla^{B}_{\nabla_a^{A}b_1}c_2\\
&=R^{A2}(a,\partial_Ac_2)b_1+\nabla^{B}_{b_1}\nabla_ac_2-\nabla_a\nabla^{B}_{b_1}c_2-\nabla_{\nabla_{b_1}^{B}a}c_2+\nabla^{B}_{\nabla_a^{A}b_1}c_2\\
&= R^{B2}(b_1,\partial_Bc_2)a=R^{B2}(b_1,b_2)(a).
\end{align*}
\item Choose compatible linear splittings $\Sigma_1$ and $\Sigma_2$
  for $D_1$ and $D_2$ as in (1), and compatible linear splittings
  $\Sigma_1'$ and $\Sigma_2'$ for $D_1'$ and $D_2'$.  Then the
  $A$-connections on $C$ induced as in (1) by $\Sigma_1$ and
  $\Sigma_2$ are equal, and the $A'$-connections on $C'$ induced as in
  (1) by $\Sigma_1'$ and $\Sigma_2'$ are equal. Let
  $(\phi_A\colon A\to A',\phi_B\colon B\to B',\phi_C\colon C\to C')$
  be the morphism of core diagrams underlying both $\Phi_1$ and
  $\Phi_2$.  The linear splittings and $\Phi_1,\Phi_2$ define
  $\omega_1,\omega_2\in\Gamma(A^*\otimes B^*\otimes C')$ such that
  \begin{equation}\label{same_dec}
    \Phi_i(\Sigma_i(a,b))=\Sigma_i'(\phi_A(a),\phi_B(b))+_{B'}(0^{D_i'}_{\phi_B(b)}+_{A'}\overline{\omega_i(a,b)})
  \end{equation}
  for $i=1,2$ and $(a,b)\in A\times_MB$.
  Since $\Phi_i\colon D_i\to D_i'$ is a VB-algebroid morphism over $\phi_B\colon B\to B'$ for $i=1,2$,
  \eqref{comp2} yields 
  \[(\nabla^{\rm Hom}_a\phi_C)(c)=\omega_i(a,\da_Bc)
  \]
  for $i=1,2$ and $(a,c)\in A\times_MC$. Since $\da_B$ is surjective, this shows $\omega_1=\omega_2$.
  The equality
  $\Phi_2\circ\Phi_{D_1,D_2}=\Phi_{D_1',D_2'}\circ\Phi_1$
  follows then immediately from \eqref{same_dec} and the constructions of $\Phi_{D_1,D_2}$ and $\Phi_{D_1',D_2'}$.
  \end{enumerate}
\end{proof}

To conclude the proof of Theorem \ref{eq_cat},
consider for each transitive double Lie algebroid $D$ the isomorphism
\[ \mathcal I(D):=\Phi_{(\D\circ\,\mathcal C)(D),D}\colon (\D\circ\,\mathcal C)(D)\to D.
\]
By the last theorem it defines a natural isomorphism
\[\mathcal I\colon \D\circ\,\mathcal C\Rightarrow \id_{\DLA},
\]
since
 \begin{equation*}
  \begin{xy}
    \xymatrix{
      (\D\circ\,\mathcal C)(D_1)\ar[r]^{\qquad \mathcal I(D_1)}\ar[d]_{(\D\circ\,\mathcal C)(\Phi)}&   D_1\ar[d]^{\Phi}\\
       (\D\circ\,\mathcal C)(D_2)\ar[r]_{\qquad \mathcal I(D_2)}                   &  D_2\\
    }
  \end{xy}
\end{equation*}
commutes for any morphism $\Phi\colon D_1\to D_2$ of transitive double Lie algebroids.

    \section{Integration of transitive double Lie algebroids}\label{integration}
    Let $(\Gamma,G,H,M)$ be a double Lie groupoid with core $K$. Then
    taking the Lie algebroids $A_G(\Gamma)\to G$ and $A(H)\to M$
    yields an \emph{LA-groupoid}, see \cite{Mackenzie92}.
    \[\begin{tikzcd}
	{A_G(\Gamma)} & G \\
	{A(H)} & M
	\arrow[from=1-1, to=1-2]
	\arrow[shift left=1, from=1-1, to=2-1]
	\arrow[shift right=2, from=1-1, to=2-1]
	\arrow[shift right=1, from=1-2, to=2-2]
	\arrow[shift left=1, from=1-2, to=2-2]
	\arrow[from=2-1, to=2-2]
      \end{tikzcd}\] That is, the horizontal arrows in the diagram
    above carry Lie algebroid stuctures, and the vertical structures
    are Lie groupoids, with a compatibility condition. Taking the Lie
    algebroids of $A_G(\Gamma)\rr H$ and $G\rr M$ yields a double Lie algebroid
    \[\begin{tikzcd}
	{A_{A(H)}(A_G(\Gamma))} & {A(G)} \\
	{A(H)} & M
	\arrow[from=1-1, to=1-2]
	\arrow[from=1-1, to=2-1]
	\arrow[from=1-2, to=2-2]
	\arrow[from=2-1, to=2-2]
      \end{tikzcd}\]
    which is isomorphic to the double Lie algebroid
    \[\begin{tikzcd}
	{A_{A(G)}(A_H(\Gamma))} & {A(G)} \\
	{A(H)} & M
	\arrow[from=1-1, to=1-2]
	\arrow[from=1-1, to=2-1]
	\arrow[from=1-2, to=2-2]
	\arrow[from=2-1, to=2-2]
      \end{tikzcd}\] see \cite{Mackenzie00}. Any of these two double
    Lie algebroids is understood as the double Lie algebroid of the
    double Lie groupoid $(\Gamma,G,H,M)$
    \cite{Mackenzie00}. LA-algebroids have been integrated by
    \cite{BuCaHo16} to LA-groupoids, see also \cite{BrCaOr18}, and a
    certain class of LA-groupoids were integrated by
    \cite{Stefanini09} to double Lie groupoids, but as far as the
    authors know the ``direct'' integration of (integrable) double Lie
    algebroids to double Lie groupoids has not been discussed yet in
    the literature.

    Consider here a transitive double Lie algebroid
    \begin{equation*}
  \begin{xy}
    \xymatrix{
      D \ar[r]^{\pi_A}\ar[d]_{\pi_B}&   A\ar[d]^{q_A}\\
       B\ar[r]_{q_B}                   &  M\\
    }
  \end{xy}
\end{equation*}
with transitive core diagram as in Figure \ref{fig:cd_simple}. Assume
that $K\rr M$, $G\rr M$ and $H\rr M$ are source-simply connected Lie
groupoids with $A(K)=C\to M$, $A(G)=A\to M$ and $A(H)=B\to M$ -- that
is, assume that the Lie algebroids $A$, $B$ and $C$ are all
integrable, to $G$, $H$ and $K$, respectively. Then according to Lie's
second theorem for Lie algebroid morphisms (see the appendix of
\cite{MaXu00}), the core diagram in Figure \ref{fig:cd_simple}
integrates to a diagram of Lie groupoid morphisms as in Figure
\ref{cd_lg}. A study of the proof in \cite{MaXu00} reveals that the
surjectivity of $\da_A$ and $\da_B$ imply that the corresponding Lie
groupoid morphisms $\da_G\colon K\to G$ and $\da_H\colon K\to H$ are
surjective submersions.

For each $m\in M$, let $G_m$ be the source fiber
$\s\inv(m)\subseteq G$, let $K_m\subseteq K$ be similarly the source
fiber of $K$ through $1_m$, and let $F_m:=\ker(\da_G)_m\subseteq K_m$
be space of elements $h\in K_m$ with $\da_G(h)=1_m\in G$. Then it is
easy to see that $F_m$ is a Lie group, that acts smoothly, freely and
properly on the right on $K_m$ via the multiplication of $K$. The
orbit space of that action is $G_m$, and the quotient map the
surjective submersion $\da_G\an{K_m}\colon K_m \to G_m$.  That is,
$\da_G\an{K_m}\colon K_m \to G_m$ is a principal $F_m$-bundle.  It
satisfies therefore the homotopy lifting property and the long exact
sequence in homotopy yields that $F_m$ is connected\footnote{This can
  also be seen without using further general tools: take $k,k'\in F_m$
  and a path $\gamma\colon [0,1]\to F_m$ joining $k$ and $k'$. Then
  $\da_G\circ \gamma$ is a loop based at $1_m$ in $G_m$.  Since $G_m$
  is simply connected, this loop is homotopic (with fixed endpoints)
  to the constant loop at $1_m$. Lifting the homotopy to $K_m$ yields
  a homotopy between $\gamma$ and a \emph{path in $F_m$}, by paths all
  with endpoints in $F_m$. This shows that $k$ and $k'$ can be joined
  by a path in $F_m$.}  since $G_m$ and $K_m$ are simply connected
(see \cite[Chapter 4]{Hatcher02}).

That is, the Lie subgroupoid $\ker(\da_G)\subseteq K$ integrating
$\ker(\da_A)$ is source-connected, and similarly, the Lie subgroupoid
$\ker(\da_H)\subseteq K$ integrating $\ker(\da_B)$ is
source-connected. Hence the condition $[c_1,c_2]=0$ for all
$c_1\in\Gamma(\ker\da_A)$ and all $c_2\in\Gamma(\ker\da_B)$ integrates
to the elements of $\ker\da_G$ commuting with those of $\ker\da_H$. In
other words, the core diagram of the double Lie algebroid $(D,A,B,M)$
integrates to a core diagram of Lie groupoids as in Figure
\ref{cd_lg}.

According to the discussion in Section \ref{expand_BrMa92}, there
exists then a (up to isomorphism) unique transitive double Lie
groupoid \[\begin{tikzcd}
    \Theta & G \\
    H& M \arrow[shift left=1, from=1-2, to=2-2] \arrow[shift right=1,
    from=1-2, to=2-2] \arrow[shift right=1, from=2-1, to=2-2]
    \arrow[shift left=1, from=2-1, to=2-2] \arrow[shift right=1,
    from=1-1, to=1-2] \arrow[shift right=1, from=1-1, to=2-1]
    \arrow[shift left=1, from=1-1, to=2-1] \arrow[shift left=1,
    from=1-1, to=1-2]
  \end{tikzcd}\] with the core diagram in Figure
\ref{cd_lg}. Following Propositions 2.6 and 2.9 in \cite{Mackenzie00},
and Definition 5.1 in \cite{Mackenzie92}, the transitive
core diagram of the transitive double Lie algebroid
\[(A_{A(G)}(A_H(\Theta)), A(G), A(H), M)\simeq (A_{A(H)}(A_G(\Theta)),
  A(G), A(H), M)\] of the double Lie groupoid $(\Theta,G,H,M)$ is the
one in Figure \ref{fig:cd_simple}. Hence by Theorem \ref{iso_dvbs}
the double Lie algebroid of $(\Theta,G,H,M)$ is isomorphic to
$(D,A,B,M)$, and so \emph{$(\Theta,G,H,M)$ integrates $(D,A,B,M)$}.

Since by Lie's second theorem morphisms of transitive core diagrams
(of Lie algebroids) integrate to morphisms of
transitive core diagrams (of Lie groupoids), it is easy to see in the
same manner that a morphism of transitive double Lie algebroids over a
fixed double base integrates to a morphism of transitive double Lie
groupoids over the same double base.  This proves the following
theorem.
    \begin{thm}\label{integration_thm}
      Let \begin{equation*}
  \begin{xy}
    \xymatrix{
      D \ar[r]^{\pi_A}\ar[d]_{\pi_B}&   A\ar[d]^{q_A}\\
       B\ar[r]_{q_B}                   &  M\\
    }
  \end{xy}
\end{equation*} be a transitive double Lie algebroid with core diagram
$$ 
\xymatrix{
&C\ar[r]_{\da_B}\ar[d]^{\da_A}&B\\
&A&
}
$$ 
Assume that
$A$,$B$ and $C$ are integrable Lie algebroids, with
  $A\simeq A(G)$, $B\simeq A(H)$ and $C\simeq A(K)$ for
  source simply-connected Lie groupoids $G\rr M$, $H\rr M$ and
  $K\rr M$.
    Then the double Lie algebroid
$(D,A,B,M)$ is integrable to a transitive double Lie groupoid
$(\Theta,G,H,M)$, which is uniquely defined by its core diagram \[\begin{tikzcd}
	K & G \\
	H & M 
	\arrow[shift left=1, from=1-2, to=2-2]
	\arrow[shift right=1, from=1-2, to=2-2]
	\arrow[shift right=1, from=1-1, to=2-2]
	\arrow[shift left=1, from=1-1, to=2-2]
	\arrow[shift right=1, from=2-1, to=2-2]
	\arrow[shift left=1, from=2-1, to=2-2]
	\arrow["{\partial_G}", from=1-1, to=1-2]
	\arrow["{\partial_H}"', from=1-1, to=2-1]
      \end{tikzcd}\]
    such that $A(\da_G)=\da_A$ and $A(\da_H)=\da_B$.
    A morphism of integrable transitive double Lie algebroids over a
    fixed base $M$ integrates further to a morphism of transitive
    double Lie groupoids over the double base $M$.
  \end{thm}

\appendix 
\section{The Lie bialgebroid condition for $R$ and $\tilde D$}\label{proof_of_R_lba}
This section checks in detail the seven conditions in Definition 3.1
and Theorem 3.6 of \cite{GrJoMaMe18} for the two VB-Lie algebroid
structures on $R$, and verifies at the same time that the conditions
still hold for the ``quotient representations up to homotopy''
describing $\tilde D$. For clarity, each verification is numbered as
the equation is in \cite{GrJoMaMe18}.

Note that in the case where $A=TM$ and
$\partial=\rho_C$, this is an alternative proof for 
$(TC;TM,C;M)$ satisfying the Lie bialgebroid condition
\cite{Mackenzie11}, which is already given in \cite{GrJoMaMe18}.
Recall that, after a choice of linear $A$-connection $\nabla$ on $C$ extending $\nabla^A$,
the two representations up to homotopy encoding the sides of $R$ are
the one of $A$ given by
$(\id_C\colon C\to C, \nabla, \nabla, R_\nabla)$ and the one of $C$
given by
$(\partial\colon C\to A,\nabla^{\da}, \nabla^{\da}, R_\nabla^{\da})$.
The two representations up to homotopy encoding the sides of the
quotient $\tilde D$ are then  $(\partial_B\colon C\to B, \nabla, \nabla^B,
\bar R_\nabla)$ and 
$(\partial_A\colon C\to A,\nabla^{\rm red}, \nabla^{\rm red}, R_\nabla^{\rm
  red})$ -- now $\da\colon C\to A$ is written $\da_A$ for clarity.
\begin{enumerate}
\item[(M1)] \begin{itemize}
  \item[$R$:]
    By definition of $\nabla^{\da}$, 
\[\nabla_{\da_A c_1}c_2-\nabla^{\da}_{c_2}c_1=[c_1,c_2]=-\nabla_{\da_Ac_2}c_1+\nabla^{\da}_{c_1}c_2\]
for
all $c_1,c_2\in\Gamma(C)$.
\item[$\tilde D$:]
From this follows immediately
\[\nabla_{\da_Ac_1}c_2-\nabla^{\rm red}_{\da_B c_2}c_1=[c_1,c_2]=-\nabla_{\da_Ac_2}c_1+\nabla^{\rm red}_{\da_B c_1}c_2.\]
\end{itemize}
\item[(M2)] \begin{itemize}
  \item[$R$:]
    Choose $a\in\Gamma(A)$ and $c\in\Gamma(C)$. 
The equation $[a,\da_Ac]=\da_A(\nabla_ac)-\nabla^{\da}_{\id_Cc}a$ is the
definition of $\nabla^{\da}$.
\item[$\tilde D$:]
By definition
  of $\nabla^{\rm red}$, this yields:
  $[a,\da_Ac]=\da_A(\nabla_ac)-\nabla^{\rm red}_{\da_Bc}a$.
  \end{itemize}
\item[(M3)]\begin{itemize}
  \item[$R$:] $[c_1,\id_Cc_2]=\id_C(\nabla^{\rm  bas}_{c_1}c_2)-\nabla_{\da_Ac_2}c_1$ is again the
    definition of $\nabla^{\da}$.
    \item[$\tilde D$:] Applying $\da_B$ to both sides of
  this equation yields
$[b,\da_Bc]=\da_B(\nabla^{\rm red}_{b}c)-\nabla^B_{\da_Ac}b$
for all $b\in\Gamma(B)$ and $c\in\Gamma(C)$.
\end{itemize}

\item[(M4)] \begin{itemize}
  \item[$R$:] For $a\in\Gamma(A)$ and $c_1,c_2\in\Gamma(C)$:
\begin{align*}
\nabla^{\da}_{c_1}\nabla_ac_2-\nabla_a \nabla^{\da}_{c_1
}c_2-\nabla_{\nabla^{\da}_{c_1}a}c_2+\nabla^{\da}_{\nabla_ac_1}c_2
=&[c_1,\nabla_ac_2]+\nabla_{\partial_A\nabla_ac_2}c_1-\nabla_a
[c_1,c_2]-\nabla_a\nabla_{\da_Ac_2}c_1\\
&-\nabla_{\nabla^{\da}_{c_1}a}c_2+[\nabla_ac_1,c_2]+\nabla_{\da_Ac_2}\nabla_ac_1.
\end{align*}
Since
$\nabla_{\partial_A\nabla_ac_2}c_1=\nabla_{-[\da_Ac_2,a]+\nabla^{\da}_{c_2}a}c_1$, this yields
\begin{align*}
&\nabla^{\da}_{c_1}\nabla_ac_2-\nabla_a \nabla^{\da}_{c_1
}c_2-\nabla_{\nabla^{\da}_{c_1}a}c_2+\nabla^{\da}_{\nabla_ac_1}c_2=R_\nabla^{\da}(c_1,\id_Cc_2)a-R_\nabla(a,\da_Ac_2)c_1.
\end{align*}
\item[$\tilde D$:]
Set $c=c_2\in\Gamma(C)$ and $b=\da_Bc_1\in\Gamma(B)$. Then
\begin{align*}
&\nabla^{\rm red}_{b}\nabla_ac-\nabla_a \nabla^{\rm red}_{b
}c-\nabla_{\nabla^{\rm red}_{b}a}c+\nabla^{\rm
  red}_{\nabla^B_ab}c=R_\nabla^{\rm red}(b,\da_Bc)a-\bar R_\nabla(a,\da_Ac)b
\end{align*}
by definition of the $B$-connections $\nabla^{\rm red}$
and the $A$-connection $\nabla$ on $B$.
\end{itemize}
\item[(M5)] \begin{itemize}
  \item[$R$:]  Choose $c\in\Gamma(C)$ and $a_1,a_2\in\Gamma(A)$. Then
\begin{align*}
&-\nabla^{\da}_c[a_1,a_2]+[\nabla^{\da}_ca_1,a_2]+[a_1,\nabla^{\da}_ca_2]
+\nabla^{\da}_{\nabla_{a_2}c}a_1-\nabla^{\da}_{\nabla_{a_1}c}a_2\\
=&-[\da_Ac,[a_1,a_2]]-\da_A(\nabla_{[a_1,a_2]}c)+[[\da_Ac,a_1]+\cancel{\da_A(\nabla_{a_1}c)},a_2]\\
&+[a_1,[\da_Ac,a_2]+\cancel{\da_A(\nabla_{a_2}c)}]\\
&+\cancel{[\da_A(\nabla_{a_2}c),a_1]}+\da_A(\nabla_{a_1}\nabla_{a_2}c)-\cancel{[\da_A(\nabla_{a_1}c),a_2]}-\da_A(\nabla_{a_2}\nabla_{a_1}c)=\da_A(R_\nabla(a_1,a_2)c)
\end{align*}
by the Jacobi identity.
\item[$\tilde D$:] The equation 
\begin{align*}
&-\nabla^{\rm red}_b[a_1,a_2]+[\nabla^{\rm
  red}_ba_1,a_2]+[a_1,\nabla^{\rm red}_ba_2]
+\nabla^{\rm red}_{\nabla^B_{a_2}b}a_1-\nabla^{\rm
  red}_{\nabla^B_{a_1}b}a_2
=\da_A(\bar R_\nabla(a_1,a_2)b
\end{align*}
follows again immediately for $b=\da_Bc$. Since $\da_B$
is surjective, this shows (M5) for $\tilde D$.
\end{itemize}
\item[(M6)]\begin{itemize}
  \item[$R$:] By definition of $R_\nabla^{\da}$,
\[ \id_C(R_\nabla^{\da}(c_1,c_2)(a))
  =-\nabla_a[c_1,c_2]+[\nabla_ac_1,c_2]
  +[c_1,\nabla_ac_2]-\nabla_{\nabla^{\da}_{c_1}a}c_2+\nabla_{\nabla^{\da}_{c_2}a}c_1.
\]
\item[$\tilde D$:] 
Setting $b_i=\da_Bc_i\in\Gamma(B)$ for $i=1,2$ and applying $\da_B$ to
both sides of the equation yields then also
\[ \da_B(R_\nabla^{\rm red}(b_1,b_2)(a))
  =-\nabla^B_a[b_1,b_2]+[\nabla^B_ab_1,b_2]
  +[b_1,\nabla^B_ab_2]-\nabla^B_{\nabla^{\rm red}_{b_1}a}b_2+\nabla^B_{\nabla^{\rm red}_{b_2}a}b_1.
\]
\end{itemize}
\item[(M7)] \begin{itemize}
  \item[$R$:]
    Checking the last equation is long, but
  straightforward. Choose $c_1,c_2\in\Gamma(C)$ and
  $a_1,a_2\in\Gamma(A)$.
Then  $(d_{\nabla}R_\nabla^{\da})(a_1,a_2)$ evaluated on $c_1,c_2\in\Gamma(C)$ reads 
\begin{align*}
&-R_\nabla^{\da}(c_1,c_2)[a_1,a_2]-\nabla_{a_2}(R_\nabla^{\da}(c_1,c_2)a_1)+\nabla_{a_1}(R_\nabla^{\da}(c_1,c_2)a_2)\\
&+R_\nabla^{\da}(\nabla_{a_2}c_1,c_2)a_1 +R_\nabla^{\da}(c_1,\nabla_{a_2}c_2)a_1-R_\nabla^{\da}(\nabla_{a_1}c_1,c_2)a_2 -R_\nabla^{\da}(c_1,\nabla_{a_1}c_2)a_2,
\end{align*}
which equals
\begin{align*}
&\nabla_{[a_1,a_2]}[c_1,c_2]-[\nabla_{[a_1,a_2]} c_1,c_2]
  -[c_1,\nabla_{[a_1,a_2]} c_2]+\nabla_{\nabla^{\da}_{c_1}{[a_1,a_2]}}c_2-\nabla_{\nabla^{\da}_{c_2}{[a_1,a_2]}}c_1\\
&-\nabla_{a_2}\left(-\nabla_{a_1}[c_1,c_2]+\cancel{[\nabla_{a_1} c_1,c_2]}
  +\cancel{[c_1,\nabla_{a_1} c_2]}-\nabla_{\nabla^{\da}_{c_1}{a_1}}c_2+\nabla_{\nabla^{\da}_{c_2}{a_1}}c_1
\right)\\
&+\nabla_{a_1}\left(-\nabla_{a_2}[c_1,c_2]+\cancel{[\nabla_{a_2} c_1,c_2]}
  +\cancel{[c_1,\nabla_{a_2} c_2]}-\nabla_{\nabla^{\da}_{c_1}{a_2}}c_2+\nabla_{\nabla^{\da}_{c_2}{a_2}}c_1\right)\\
&-\cancel{\nabla_{a_1}[\nabla_{a_2}c_1,c_2]}+[\nabla_{a_1} \nabla_{a_2}c_1,c_2]
  +\cancel{[\nabla_{a_2}c_1,\nabla_{a_1} c_2]}-\nabla_{\nabla^{\da}_{\nabla_{a_2}c_1}{a_1}}c_2+\nabla_{\nabla^{\da}_{c_2}{a_1}}\nabla_{a_2}c_1\\
&-\cancel{\nabla_{a_1}[c_1,\nabla_{a_2}c_2]}+\cancel{[\nabla_{a_1} c_1,\nabla_{a_2}c_2]}
  +[c_1,\nabla_{a_1} \nabla_{a_2}c_2]-\nabla_{\nabla^{\da}_{c_1}{a_1}}\nabla_{a_2}c_2+\nabla_{\nabla^{\da}_{\nabla_{a_2}c_2}{a_1}}c_1\\
&+\cancel{\nabla_{a_2}[\nabla_{a_1}c_1,c_2]}-[\nabla_{a_2} \nabla_{a_1}c_1,c_2]
  -\cancel{[\nabla_{a_1}c_1,\nabla_{a_2} c_2]}+\nabla_{\nabla^{\da}_{\nabla_{a_1}c_1}{a_2}}c_2-\nabla_{\nabla^{\da}_{c_2}{a_2}}\nabla_{a_1}c_1\\
&+\cancel{\nabla_{a_2}[c_1,\nabla_{a_1}c_2]}-\cancel{[\nabla_{a_2} c_1,\nabla_{a_1}c_2]}
  -[c_1,\nabla_{a_2} \nabla_{a_1}c_2]+\nabla_{\nabla^{\da}_{c_1}{a_2}}\nabla_{a_1}c_2-\nabla_{\nabla^{\da}_{\nabla_{a_1}c_2}{a_2}}c_1.
\end{align*}
Twelve terms of this equation cancel pairwise, and a reordering 
of the remaining terms yields
\begin{align*}
&-R_\nabla(a_1,a_2)[c_1,c_2]+[R_\nabla(a_1,a_2)c_1,c_2]+[c_1,R_\nabla(a_1,a_2)c_2]\\
  &  +\nabla_{\nabla^{\da}_{c_1}{[a_1,a_2]}}c_2-\nabla_{\nabla^{\da}_{c_2}{[a_1,a_2]}}c_1-\nabla_{a_2}\left(-\nabla_{\nabla^{\da}_{c_1}{a_1}}c_2+\nabla_{\nabla^{\da}_{c_2}{a_1}}c_1\right)
    +\nabla_{a_1}\left(-\nabla_{\nabla^{\da}_{c_1}{a_2}}c_2+\nabla_{\nabla^{\da}_{c_2}{a_2}}c_1\right)\\
  &-\nabla_{\nabla^{\da}_{\nabla_{a_2}c_1}{a_1}}c_2+\nabla_{\nabla^{\da}_{c_2}{a_1}}\nabla_{a_2}c_1-\nabla_{\nabla^{\da}_{c_1}{a_1}}\nabla_{a_2}c_2+\nabla_{\nabla^{\da}_{\nabla_{a_2}c_2}{a_1}}c_1
  +\nabla_{\nabla^{\da}_{\nabla_{a_1}c_1}{a_2}}c_2-\nabla_{\nabla^{\da}_{c_2}{a_2}}\nabla_{a_1}c_1\\
& +\nabla_{\nabla^{\da}_{c_1}{a_2}}\nabla_{a_1}c_2-\nabla_{\nabla^{\da}_{\nabla_{a_1}c_2}{a_2}}c_1\\
=&-R_\nabla(a_1,a_2)[c_1,c_2]+[R_\nabla(a_1,a_2)c_1,c_2]+[c_1,R_\nabla(a_1,a_2)c_2]\\
&+R_\nabla(a_2,\nabla^{\da}_{c_1}{a_1})c_2+\nabla_{[a_2,\nabla^{\da}_{c_1}a_1]}c_2-R_\nabla(a_2,\nabla^{\da}_{c_2}{a_1})c_1-\nabla_{[a_2,\nabla^{\da}_{c_2}a_1]}c_1\\
&-R_\nabla(a_1,\nabla^{\da}_{c_1}{a_2})c_2-\nabla_{[a_1,\nabla^{\da}_{c_1}a_2]}c_2+R_\nabla(a_1,\nabla^{\da}_{c_2}{a_2})c_1+\nabla_{[a_1,\nabla^{\da}_{c_2}a_2]}c_1\\
&  +\nabla_{\nabla^{\da}_{c_1}{[a_1,a_2]}}c_2-\nabla_{\nabla^{\da}_{c_2}{[a_1,a_2]}}c_1-\nabla_{\nabla^{\da}_{\nabla_{a_2}c_1}{a_1}}c_2\\
&+\nabla_{\nabla^{\da}_{\nabla_{a_2}c_2}{a_1}}c_1+\nabla_{\nabla^{\da}_{\nabla_{a_1}c_1}{a_2}}c_2-\nabla_{\nabla^{\da}_{\nabla_{a_1}c_2}{a_2}}c_1.
\end{align*}
By (M5), this equals
\begin{align*}
&-R_\nabla(a_1,a_2)[c_1,c_2]-[c_2, R_\nabla(a_1,a_2)c_1]+[c_1,R_\nabla(a_1,a_2)c_2]\\
  &+R_\nabla(a_2,\nabla^{\da}_{c_1}{a_1})c_2-R_\nabla(a_2,\nabla^{\da}_{c_2}{a_1})c_1-R_\nabla(a_1,\nabla^{\da}_{c_1}{a_2})c_2\\
  &+R_\nabla(a_1,\nabla^{\da}_{c_2}{a_2})c_1+\nabla_{\da_AR_\nabla(a_1,a_2)c_2}c_1-\nabla_{\da_AR_\nabla(a_1,a_2)c_1}c_2,
\end{align*}
which is 
\begin{align*}
&-R_\nabla(a_1,a_2)[c_1,c_2]-\nabla^{\da}_{c_2}(R_\nabla(a_1,a_2)c_1)+\nabla^{\da}_{c_1}(R_\nabla(a_1,a_2)c_2)\\
&-R_\nabla(\nabla^{\da}_{c_1}{a_1},a_2)c_2+R_\nabla(\nabla^{\da}_{c_2}{a_1},a_2)c_1-R_\nabla(a_1,\nabla^{\da}_{c_1}{a_2})c_2+R_\nabla(a_1,\nabla^{\da}_{c_2}{a_2})c_1.
\end{align*}
Since this is $(d_{\nabla^\da}R_\nabla)(c_1,c_2)$ on $a_1,a_2\in\Gamma(A)$, 
 Condition (M7) is checked for $R$.
\item[$\tilde D$:] Setting $b_i=\da_Bc_i\in\Gamma(B)$ for $i=1,2$ and
  using the definitions of all the involved objects yields again
  immediately
\begin{align*}
&-R_\nabla^{\rm red}(b_1,b_2)[a_1,a_2]-\nabla_{a_2}(R_\nabla^{\rm
  red}(b_1,b_2)a_1)+\nabla_{a_1}(R_\nabla^{\rm
  red}(b_1,b_2)a_2)\\
&+R_\nabla^{\rm red}(\nabla^B_{a_2}b_1,b_2)a_1 +R_\nabla^{\rm red}(b_1,\nabla^B_{a_2}b_2)a_1-R_\nabla^{\rm red}(\nabla^B_{a_1}b_1,b_2)a_2 -R_\nabla^{\rm
  red}(b_1,\nabla^B_{a_1}b_2)a_2\\
=&-\bar R_\nabla(a_1,a_2)[b_1,b_2]-\nabla^{\rm
  red}_{b_2}\bar R_\nabla(a_1,a_2)b_1+\nabla^{\rm red}_{b_1}\bar R_\nabla(a_1,a_2)b_2\\
&-\bar R_\nabla(\nabla^{\rm
  red}_{b_1}{a_1},a_2)b_2+\bar R_\nabla(\nabla^{\rm
  red}_{b_2}{a_1},a_2)b_1-\bar R_\nabla(a_1,\nabla^{\rm
  red}_{b_1}{a_2})b_2+\bar R_\nabla(a_1,\nabla^{\rm
  red}_{b_2}{a_2})b_1,
\end{align*}
and the seventh condition is verified for $\tilde D$.
\end{itemize}
\end{enumerate}

\section{Double Lie algebroid morphisms -- proofs}\label{morphisms_proofs}

\subsection{Proof of Theorem \ref{comma_mor}}
  The double vector bundle morphism $\Phi$ has sides $\phi_C$ and
  $\phi_A$ and core $\phi_C$.  By the proof of Proposition
  \ref{mor_dvb}, a pair of choices of linear splittings
  $\Sigma^\nabla\colon A\times_MC\to TC\oplus_{TM}A$ and
  $\Sigma^{\nabla'}\colon A'\times_MC'\to TC'\oplus_{TM}A'$ defines
  the form
  $\omega_{\nabla,\nabla'}\in\Gamma(A^*\otimes C^*\otimes C')$ as in
  \eqref{omega_nabla_nabla'}.

  The two splittings define as well the following $2$-representations.
  \begin{enumerate}
  \item The VB-algebroid $(TC\oplus_{TM}A\to C, A\to M)$ induces
    $\mathcal D_A:=(\nabla,\nabla,R_\nabla)$ of $A$ on
    $\id\colon C\to C$, as in \eqref{ruth_A}.
  \item The VB-algebroid $(TC\oplus_{TM}A\to A, C\to M)$ induces the
    $2$-representation
    $\mathcal D_C:=(\nabla^{\da_A},\nabla^{\da_A},R_\nabla^{\da_A})$
    of $C$ on $\da_A\colon C\to A$, as in \eqref{ruth_C} and
    \eqref{ruth_C_1}.
    \item The VB-algebroid $(TC'\oplus_{TM}A'\to C', A'\to M)$ induces
    $\mathcal D_{A'}:=(\nabla',\nabla',R_{\nabla'})$ of $A'$ on
    $\id\colon C'\to C'$, as in \eqref{ruth_A}.
  \item The VB-algebroid $(TC'\oplus_{TM}A'\to A', C'\to M)$ induces the
    $2$-representation
    $\mathcal D_{C'}:=(\nabla^{\da_{A'}},\nabla^{\da_{A'}},R_{\nabla'}^{\da_{A'}})$
    of $C'$ on $\da_{A'}\colon C'\to A'$, as in \eqref{ruth_C} and
    \eqref{ruth_C_1}.
    \end{enumerate}
  
By Section
\ref{sec:mor_dla} it suffices to check that  $(\phi_C,\phi_C,\omega_{\nabla,\nabla'})$ defines a morphism
\[ \mathcal D_A\to \phi_A^*\mathcal D_{A'}
\]
and that
$(\phi_C,\phi_A,\omega_{\nabla,\nabla'})$ defines a morphism
\[ \mathcal D_C\to \phi_C^*\mathcal D_{C'}.
\]
\eqref{comp1} for $(\phi_C,\phi_C,\omega_{\nabla,\nabla'})$ is
immediate, and \eqref{comp2} is the definition of
$\omega_{\nabla,\nabla'}$. \eqref{comp3} is an easy computation, which
is left to the reader.
\eqref{comp1} for $(\phi_C,\phi_A,\omega_{\nabla,\nabla'})$ is
\[ \phi_A\circ \da_A=\da_{A'}\circ\phi_C
\]
and \eqref{comp2} is checked as follows.
For $c\in\Gamma(C)$, $a\in\Gamma(A)$ and $c'\in\Gamma(C)$:
\begin{equation*}
  \begin{split}
    (\nabla^{\rm Hom}_c\phi_A)(a)&=(\phi_C^*\nabla^{\da_{A'}})_c(\phi_A(a))-\phi_A(\nabla^{\da_A}_ca)=\nabla^{\da_{A'}}_{\phi_C(c)}(\phi_A(a))-\phi_A(\nabla^{\da_A}_ca)\\
    &=\cancel{[\da_{A'}\phi_C(c),\phi_A(a)]}+\da_{A'}\nabla'_{\phi_A(a)}(\phi_C(c))-\phi_A(\cancel{[\da_Ac,a]}+\da_A(\nabla_ac))\\
    &=\da_{A'}\left((\phi_A^*\nabla')_a(\phi_C(c))-\phi_C(\nabla_ac)\right)=\da_{A'}(\omega_{\nabla,\nabla'}(a,c)),
  \end{split}
\end{equation*}
and
\begin{equation*}
  \begin{split}
    (\nabla^{\rm Hom}_c\phi_C)(c')&=(\phi_C^*\nabla^{\da_{A'}})_c(\phi_C(c'))-\phi_C(\nabla^{\da_A}_cc')=\nabla^{\da_{A'}}_{\phi_C(c)}(\phi_C(c'))-\phi_C(\nabla^{\da_A}_cc')\\
    &=\cancel{[\phi_C(c),\phi_C(c')]}+\nabla'_{\da_{A'}\phi_C(c')}(\phi_C(c))-\phi_C(\cancel{[c,c']}+\nabla_{\da_Ac'}c)\\
    &=\nabla'_{\phi_A\da_A(c')}(\phi_C(c))-\phi_C(\nabla_{\da_Ac'}c)=\omega_{\nabla,\nabla'}(\da_Ac',c).
  \end{split}
\end{equation*}
Finally, \eqref{comp3} is
\begin{equation}\label{long_comp_mor}
  \begin{split}
    &(\dr_{\nabla^{\rm Hom}}\omega_{\nabla,\nabla'})(c_1,c_2)(a)-(\phi_C^*R_{\nabla'}^{\da_{A'}})(c_1,c_2)(\phi_Aa) +\phi_C(R_\nabla^{\da_A}(c_1,c_2)a)\\
    =\,&\nabla^{\rm Hom}_{c_1}(\omega_{\nabla,\nabla'}c_2)(a)-\nabla^{\rm Hom}_{c_2}(\omega_{\nabla,\nabla'}c_1)(a)-\cancel{\omega_{\nabla,\nabla'}(a,[c_1,c_2])}\\
    &   +\cancel{\nabla'_{\phi_Aa}[\phi_Cc_1,\phi_Cc_2]}-[\nabla'_{\phi_A(a)}\phi_Cc_1, \phi_Cc_2]-[\phi_Cc_1, \nabla'_{\phi_A(a)}\phi_Cc_2]\\
    &\qquad +\nabla'_{\nabla^{\da_{A'}}_{\phi_Cc_1}\phi_Aa}\phi_Cc_2-\nabla'_{\nabla^{\da_{A'}}_{\phi_Cc_2}\phi_Aa}\phi_Cc_1\\
    &+\phi_C\left(-\cancel{\nabla_a[c_1,c_2]}+[\nabla_ac_1,c_2]
      +[c_1,\nabla_ac_2]-\nabla_{\nabla^{\da_A}_{c_1}a}c_2+\nabla_{\nabla^{\da_A}_{c_2}a}c_1\right).
  \end{split}
\end{equation}
for $c_1,c_2\in\Gamma(C)$ and $a\in\Gamma(A)$.
Since \begin{equation*}
  \begin{split}
    &\nabla^{\rm Hom}_{c_1}(\omega_{\nabla,\nabla'}c_2)(a)-[
    \phi_Cc_1, \nabla'_{\phi_A(a)}\phi_Cc_2]+\phi_C[c_1,\nabla_ac_2]
    +\nabla'_{\nabla^{\da_{A'}}_{\phi_Cc_1}\phi_Aa}\phi_Cc_2-\phi_C\nabla_{\nabla^{\da_A}_{c_1}a}c_2\\
    =\,&\nabla^{\da_{A'}}_{\phi_Cc_1}(\omega_{\nabla,\nabla'}(a,c_2))-\omega_{\nabla,\nabla'}(\nabla^{\da_A}_{c_1}a,c_2)-[\phi_Cc_1,\nabla'_{\phi_A(a)}\phi_Cc_2]+[\phi_Cc_1, \phi_C\nabla_ac_2]\\
    &+\nabla'_{\nabla^{\da_{A'}}_{\phi_Cc_1}\phi_Aa}\phi_Cc_2-\phi_C\nabla_{\nabla^{\da_A}_{c_1}a}c_2\\
    =\,&\nabla^{\da_{A'}}_{\phi_Cc_1}(\omega_{\nabla,\nabla'}(a,c_2))-\nabla'_{\phi_A\nabla^{\da_A}_{c_1}a}\phi_Cc_2-[ \phi_Cc_1, \omega_{\nabla,\nabla'}(a,c_2)]
    +\nabla'_{\nabla^{\da_{A'}}_{\phi_Cc_1}\phi_Aa}\phi_Cc_2\\
    =\,&\nabla'_{\da_{A'}\omega_{\nabla,\nabla'}(a,c_2)}\phi_Cc_1-\nabla'_{\phi_A\nabla^{\da_A}_{c_1}a}\phi_Cc_2
    +\nabla'_{\nabla^{\da_{A'}}_{\phi_Cc_1}\phi_Aa}\phi_Cc_2,
  \end{split}
\end{equation*}
\eqref{long_comp_mor} becomes
\begin{equation*}
  \begin{split}
    &\nabla'_{\da_{A'}\omega_{\nabla,\nabla'}(a,c_2)}\phi_Cc_1-\nabla'_{\phi_A\nabla^{\da_A}_{c_1}a}\phi_Cc_2
    +\nabla'_{\nabla^{\da_{A'}}_{\phi_Cc_1}\phi_Aa}\phi_Cc_2\\
    &\,-\nabla'_{\da_{A'}\omega_{\nabla,\nabla'}(a,c_1)}\phi_Cc_2+\nabla'_{\phi_A\nabla^{\da_A}_{c_2}a}\phi_Cc_1
    -\nabla'_{\nabla^{\da_{A'}}_{\phi_Cc_2}\phi_Aa}\phi_Cc_1\\
    =\,&\nabla'_{\da_{A'}\omega_{\nabla,\nabla'}(a,c_2)-(\nabla^{\rm
        Hom}_{c_2}\phi_A)(a)}\phi_Cc_1
    -\nabla'_{\da_{A'}\omega_{\nabla,\nabla'}(a,c_1)-(\nabla^{\rm
        Hom}_{c_1}\phi_A)(a)}\phi_Cc_2\overset{\eqref{comp2}}{=}0.\qedhere
  \end{split}
  \end{equation*}

  \subsection{Proof of Theorem \ref{dla_mor}}\label{proof_of_dla_mor}
 Here \eqref{comp1}, \eqref{comp2} and \eqref{comp3} need to be
  checked for \begin{enumerate}
  \item the triple
    $(\phi_C,\phi_B,\overline{\omega_{\nabla,\nabla'}})$ from the
    $2$-representation
    $(\da_B\colon C[0]\to B[1], \nabla, \nabla^B,
    \overline{R_\nabla})$ of $A$ as in as in \eqref{ruth_A_D} and
    \eqref{ruth_A_D1} to the $2$-representation
    $\phi_A^*(\da_{B'}\colon C'[0]\to B'[1], \nabla', {\nabla'}^{B'},
    \overline{R_{\nabla'}})$ of $A$.
\item the triple $(\phi_C,\phi_A,\overline{\omega_{\nabla,\nabla'}})$
  from the $2$-representation
  $(\da_A\colon C[0]\to A[1], \nabla^{\rm red}, \nabla^{\rm red},
  R_\nabla^{\rm red})$ of $B$ as in \eqref{ruth_B_D} and
  \eqref{ruth_B_D1} to the $2$-representation
  $\phi_B^*(\da_{A'} \colon C'[0]\to A'[1], {\nabla'}^{\rm red},
  {\nabla'} ^{\rm red}, R_{\nabla'}^{\rm red})$ of $B$.
\end{enumerate}

\eqref{comp1} is given by the morphism of core diagrams since in the
first case it is $\phi_B\circ\da_B=\da_{B'}\circ\phi_C$, and in the
second case it is
$\phi_A\circ\da_A=\da_{A'}\circ\phi_C$. \eqref{comp2} is in both cases
an easy computation, which is left to the reader.  \eqref{comp1} is
checked in the first case as follows. Choose $a,a'\in\Gamma(A)$,
$b\in\Gamma(B)$ and $c\in\Gamma(C)$ such that $\da_Bc=b$. Then
with $\nabla^B_ab=\da_B\nabla_ac$ and $\da_{B'}\circ\phi_C=\phi_B\circ\da_B$:
\begin{equation*}
  \begin{split}
    (\dr_{\nabla^{\rm Hom}}\overline{\omega_{\nabla,\nabla'}})(a,a')(b)&=
    \nabla^{\rm Hom}_a(\overline{\omega_{\nabla,\nabla'}}(a'))(b)-\nabla^{\rm Hom}_{a'}(\overline{\omega_{\nabla,\nabla'}}(a))(b)-\overline{\omega_{\nabla,\nabla'}}([a,a'],b)\\
    &=\nabla'_{\phi_A(a)}(\omega_{\nabla,\nabla'}(a',c))-\omega_{\nabla,\nabla'}(a',\nabla_ac)\\
    &\quad -\nabla'_{\phi_A(a')}(\omega_{\nabla,\nabla'}(a,c))+\omega_{\nabla,\nabla'}(a,\nabla_{a'}c)-\omega_{\nabla,\nabla'}([a,a'],c)\\
    &=(\phi_A^*R_{\nabla'})(a,a')(\phi_C(c))-\phi_C(R_\nabla(a,a')c)\\
    &=(\phi_A^*\overline{R_{\nabla'}})(a,a')(\phi_B(b))-\phi_C(\overline{R_\nabla}(a,a')b).
    \end{split}
  \end{equation*}

  Finally, in order to check \eqref{comp3} in the second case, note
  that for $b,b'\in\Gamma(B)$, $a\in\Gamma(A)$ and
  $c,c'\in\Gamma(C)$ such that $\da_Bc=b$ and $\da_Bc'=b'$,
  \[(\dr_{\nabla^{\rm Hom}}\overline{\omega_{\nabla,\nabla'}})(b,b')(a)=(\dr_{\nabla^{\rm Hom}}\omega_{\nabla,\nabla'})(c,c')(a).
  \]
  By the proof of Theorem  \ref{comma_mor},
  the right-hand side
  is
  \[ (\phi_C^*R_{\nabla'}^{\da_{A'}})(c,c')(\phi_A(a))-\phi_C(R_\nabla^{\da_A}(c,c')a),
  \]
  which equals
  \[(\phi_B^*R_{\nabla'}^{\rm red})(b,b')(\phi_A(a))-\phi_C(R_\nabla^{\rm red}(b,b')a),
  \]
  by the definition of $R_{\nabla'}^{\rm red}$ and $R_\nabla^{\rm red}$.

\section{Notation for linear and core sections.}
The main notations of this paper are summarised here for the
convenience of the reader.

\begin{enumerate}
\item Let $(D,A,B,M)$ be any double vector bundle with sides $A$ and
  $B$ and with core $C$.  A morphism $\Phi\in\Gamma(\Hom(A,C))$
  defines a core-linear section $\widetilde\Phi\in\Gamma_A^l(D)$, which is linear over $0^B\in\Gamma(B)$:
  \[\widetilde\Phi(a)=0^D_{a}+_B\overline{\phi(a)}.
  \]
  Similarly, a morphism $\Phi\in\Gamma(\Hom(B,C))$
  defines a core-linear section $\widetilde\Phi\in\Gamma_B^l(D)$, which is linear over $0^A\in\Gamma(A)$:
  \[\widetilde\Phi(b)=0^D_{b}+_A\overline{\phi(b)}.\]
  Core-linear sections are always written with this notation, no
  matter over which side or in which double vector bundle. It is
  always clear from the context where they are defined.
\item Let $q_E\colon E\to M$ be a vector bundle.  Given
  $e\in\Gamma(E)$, then $e^\uparrow\in \mx^c(E)$ is the vertical
  vector field on $E$ defined by $e$; i.e.~the core section of
  $TE\to E$ defined by $e$. The core section of $TE\to TM$ defined by
  $e$ is written $e^\dagger\in\Gamma_{TM}^c(TE)$.

  Given a derivation $D$ of $E$ with symbol $X\in\mx(M)$, then
  $\widehat{D}\in\mx^l(E)$ is the linear vector field on $E$ that is
  equivalent to $D$. It is a linear section of $TE\to E$ over $X$. A
  section $e\in\Gamma(E)$ defines further the linear section
  $Te\in\Gamma^l_{TM}(TE)$. It is linear over $e$.
\end{enumerate}
Let
$\nabla\colon \Gamma(A)\times\Gamma(C)\to\Gamma(C)$ be a linear
$A$-connection on $C$.  Then $\nabla$ defines a linear splitting
$\Sigma^\nabla\colon A\times_M C\to R$ of $R:=\rho_A^!TC=TC\oplus_{TM}A$. 
\begin{enumerate}
  \setcounter{enumi}{2}
\item Core sections of $R\to C$
  are written $\gamma^\times$ for $\gamma\in\Gamma(C)$:
  \[ \gamma^\times(c_m)=(\gamma^\uparrow(c_m), 0^A_m)\] for all
  $c_m\in C$. The splitting $\Sigma^\nabla$ is equivalent to a linear
  horizontal lift $\sigma^\nabla_A\colon\Gamma(A)\to\Gamma_C^l(R)$:
  for all $a\in\Gamma(A)$ the section
  $\widehat{a}:=\sigma_A^\nabla(a)\in\Gamma^l_C(R)$ is defined by
\[\widehat{a}(c_m)=\left(\widehat{\nabla_a}(c_m), a(m)\right)\]
for all $c_m\in C$. By definition it is linear over $a\in\Gamma(A)$.

\item Core sections of $R\to A$ 
  are written $\gamma^\ddagger$  for $\gamma\in\Gamma(C)$,
  and defined by $\gamma^\ddagger=(\gamma^\dagger)^!\in\Gamma_A(R)=\Gamma_A(\rho_A^!TC)$:
  \[\gamma^\ddagger(a)=\left(\gamma^\dagger(\rho_A(a)), a\right)\]
  for all $a\in \Gamma(A)$.
  A section $c\in\Gamma(C)$ defines a linear section $(TC)^!$ of $R\to A$:
  \[(Tc)^!(a)=\left(Tc(\rho_A(a)), a\right)\] for all
  $a\in \Gamma(A)$. It is linear over $c\in C$.  The splitting
  $\Sigma^\nabla$ is equivalent to a linear horizontal lift
  $\sigma^\nabla_C\colon\Gamma(C)\to\Gamma_A^l(R)$: for all
  $c\in\Gamma(C)$ the section
  $\widecheck{c}:=\sigma_C^\nabla(c)\in\Gamma^l_A(R)$ is defined
  by
  \[\widecheck{c}(a(m))=\left(\widehat{\nabla_a}(c(m)), a(m)\right)\]
  for all $a\in\Gamma(A)$. By definition it is linear over
  $c\in\Gamma(C)$.
\end{enumerate}

\def\cprime{$'$} \def\polhk#1{\setbox0=\hbox{#1}{\ooalign{\hidewidth
  \lower1.5ex\hbox{`}\hidewidth\crcr\unhbox0}}} \def\cprime{$'$}
  \def\cprime{$'$}

\end{document}